\documentclass[11pt,a4paper]{amsart}

\usepackage{amsfonts,amssymb,amsmath,amsthm}
\usepackage{mathtools,esint}
\usepackage{lmodern}
\usepackage[noadjust]{cite}
\usepackage{enumitem}
\usepackage{needspace}
\usepackage{xcolor}
\usepackage[plainpages=false,pdfpagelabels,backref=page]{hyperref}
\hypersetup{
 colorlinks=true,
 linkcolor={cyan!90!black},
 citecolor={magenta},
 urlcolor={green!40!black},
 pdftitle={Boundary Harnack inequalities for the fractional p-Laplacian in Reifenberg flat domains},
 pdfauthor={Kaj Nystr\"om},
pdfsubject={Reifenberg approximation, weighted non-local forms, and ordered boundary comparison}
}

\newtheorem{thm}{Theorem}[section]
\newtheorem{lem}[thm]{Lemma}
\newtheorem{prop}[thm]{Proposition}
\newtheorem{cor}[thm]{Corollary}

\theoremstyle{definition}
\newtheorem{defn}[thm]{Definition}
\newtheorem{rem}[thm]{Remark}
\numberwithin{equation}{section}

\usepackage{tikz}
\usetikzlibrary{arrows.meta,calc,patterns,decorations.pathreplacing}
\definecolor{KPBlue}{RGB}{35,82,111}
\definecolor{KPGreen}{RGB}{45,100,83}
\definecolor{KPGold}{RGB}{163,115,43}
\tikzset{
  kp/past/.style={draw=KPBlue,densely dashed,line width=.8pt,fill=KPBlue!7},
  kp/future/.style={draw=KPGreen,line width=.85pt,fill=KPGreen!6},
  kp/outer/.style={draw=black!55,dash pattern=on 4pt off 2pt,line width=.65pt},
  kp/arrow/.style={-{Stealth[length=2mm]},draw=black!65,line width=.7pt},
  kp/step/.style={-{Stealth[length=2.2mm]},draw=KPBlue,line width=1pt},
  kp/vertex/.style={draw=black!70,rounded corners=1.5pt,fill=white,
    minimum width=2.05cm,minimum height=1.02cm,inner sep=3pt,
    align=center,font=\small},
  kp/terminal/.style={kp/vertex,double,double distance=1.2pt},
  kp/small/.style={font=\footnotesize,align=center},
  kp/heading/.style={font=\small\bfseries,anchor=west},
  kp/leader/.style={draw=black!65,line width=.5pt},
  kp/axis/.style={-{Stealth[length=1.7mm]},line width=.55pt,draw=black!75}
}

\newcommand{\R}{\mathbb R}
\newcommand{\cH}{\mathcal H}

\usepackage{mathrsfs}
\renewcommand{\d}{\,\mathrm{d}}

\newcommand{\esssup}{\mathop{\mathrm{ess\,sup}}}
\newcommand{\essinf}{\mathop{\mathrm{ess\,inf}}}
\newcommand{\supp}{\operatorname{supp}}
\newcommand{\avgint}{\fint}
\newcommand{\Om}{\Omega}
\newcommand{\dist}{\operatorname{dist}}
\newcommand{\osc}{\operatorname{osc}}
\newcommand{\Tail}{\operatorname{Tail}}
\newcommand{\capsp}{\operatorname{cap}_{s,p}}
\newcommand{\Jp}{J_p}
\newcommand{\Fsp}{\mathcal F_{s,p}}
\newcommand{\Bu}{\mathcal B_u}
\newcommand{\HH}{\mathbb H}
\newcommand{\Ws}{W^{s,p}}
\newcommand{\Wzero}{W^{s,p}_0}
\newcommand{\Vs}{V^{s,p}}
\newcommand{\trange}{2\le p<2/(1-s)}
\newcommand{\dxdy}{\d x\,\mathrm d y}
\allowdisplaybreaks
\makeatletter
\renewcommand{\l@section}{%
  \@tocline{1}{3pt}{0pt}{}{\normalfont\small}}
\makeatother

\begin{document}

\title[Fractional boundary Harnack inequalities]
{Boundary Harnack inequalities for the fractional $p$-Laplacian
in Reifenberg flat domains}
\author{Kaj Nystr\"om}
\address{Department of Mathematics, Uppsala University, Box 480\\
SE-751 06 Uppsala, Sweden}
\email{kaj.nystrom@math.uu.se}
\renewcommand{\subjclassname}{\textup{2020} Mathematics Subject Classification}

\begin{abstract}
We prove a scale-invariant boundary Harnack inequality and
H\"older continuity up to the boundary of quotients of positive fractional
$p$-harmonic functions in Reifenberg flat domains with sufficiently
small flatness constant, for $0<s<1$ and $2\le p<2/(1-s)$.
The solutions are globally nonnegative, vanish on a common exterior
boundary patch, and belong to the natural local energy and
$(p-1)$-tail class. Their distant exterior values may be unrelated
and unbounded. After normalization at a corkscrew point, the
comparison and H\"older constants depend only on $n,s,p$.  We first establish local half-space boundary comparison for every
$p\ge2$ and prove relative approximation by the half-space profile
in Reifenberg domains. The compactness argument identifies escaped exterior mass
and excludes its constant source term. Interior $C^{1,\alpha}$ regularity then
yields gradient non-degeneracy and lower increment bounds in a
cone of directions. To treat the degenerate kernels produced by
linearization, we develop estimates for weighted non-local forms
and prove a localized Harnack inequality for bounded quotients,
retaining the full exterior interactions. Ordered variational
deformations and a common fractional $p$-harmonic majorant yield comparison for
solutions with unrelated exterior data. Truncation of distant data and fractional $p$-harmonic replacement
extend the estimates from the variational energy class to the
natural tail class.
\end{abstract}

\maketitle


\begingroup
\hypersetup{linkcolor=black}
\tableofcontents
\endgroup


\section{Introduction and statement of main results}\label{Section_0}

In a series of papers \cite{LN07,LN,LN10}, John Lewis and the author
developed boundary estimates for non-negative solutions to the
$p$-Laplace equation, $1<p<\infty$, including boundary Harnack
inequalities and H\"older continuity up to the boundary of quotients
of positive solutions vanishing on a common boundary portion.
These estimates, and the methods underlying their proofs, have found
applications in the regularity theory of two-phase free boundary
problems \cite{LN10FB,LN12FB}, in inverse free boundary problems below
the continuous threshold \cite{LN08R,LN12R}, and in the study of
$p$-harmonic measure and its Hausdorff dimension \cite{LNP11,LNV13}.

The methods developed in these papers reflect the underlying geometry.
In $C^1$ domains and Reifenberg flat domains with sufficiently small
flatness constant, approximation by hyperplanes and quantitative
non-degeneracy estimates for gradients play a central role
\cite{LN,LLN}.
The analysis in Lipschitz domains and Ahlfors regular NTA domains
also uses estimates for boundary measures and approximation by
suitable Lipschitz subdomains \cite{LN07,LN10,LN08R,LN12R}.
A common ingredient is the passage from the nonlinear equation
to associated linear degenerate elliptic equations, whose analysis
connects the boundary behavior of $p$-harmonic functions with
weighted elliptic theory.

The purpose of this paper is to develop a counterpart of the boundary
comparison theory in \cite{LN,LLN} for the fractional $p$-Laplacian in
Reifenberg flat domains.
This extension presents additional difficulties arising from
non-locality, since boundary estimates must account for interactions
with the exterior of the domain, while linearization produces kernels
whose degeneracy depends on increments of the solutions at pairs
of points.

To formulate the problem, let $n\ge2$, $0<s<1$, and $1<p<\infty$.
We consider the fractional $p$-Laplacian
\begin{equation}\label{eq:operator}
 \Fsp u(x)=\operatorname{PV}\int_{\R^n}
 \frac{|u(x)-u(y)|^{p-2}(u(x)-u(y))}
 {|x-y|^{n+sp}}\,\mathrm d y
\end{equation}
and the associated Dirichlet problem in a domain $\Om\subset\R^n$.
For this non-local operator, Dirichlet data are prescribed on
$\R^n\setminus\Om$.
Moreover, estimates localized to a ball retain contributions from
the values of the solution outside that ball.
These contributions are measured by the tail
\begin{equation}\label{eq:tail}
 \Tail(f,z,r)=
 \left(r^{sp}\int_{\R^n\setminus B_r(z)}
       \frac{|f(y)|^{p-1}}{|y-z|^{n+sp}}
       \,\mathrm d y\right)^{1/(p-1)}.
\end{equation}

For $w\in\partial\Om$ and $r>0$, we write
$D_r(w)=\Om\cap B_r(w)$.
For $x\in\Om$, we set
$d_\Om(x)=\dist(x,\R^n\setminus\Om)$.
The geometric assumptions and the weak formulation of the equation
with zero exterior values are specified in
Section~\ref{sec:dirichlet}.
A corkscrew point $A_r(w)$ is chosen in $D_r(w)$ so that
$d_\Om(A_r(w))\ge r/M_{\mathrm{geom}}$, where $M_{\mathrm{geom}}$ is a fixed geometric constant.
Pointwise values of weak solutions refer to their continuous
representatives in the region where the equation holds.

Our main result is the following theorem.

\begin{thm}\label{thm:main}
Let $0<s<1$ and $\trange$.
There exist $\delta_*>0$, $C\ge1$, and $\alpha\in(0,1)$,
depending only on $n,s,p$, with the following property.
Suppose that $\Om$ is $(\delta,r_0)$-Reifenberg flat,
$\delta\le\delta_*$, $w\in\partial\Om$, and $0<r<r_0/8$.
Let $u,v\in \mathrm L^{p-1}_{\mathrm{loc}}(\R^n)$ be  functions, nonnegative almost everywhere in $\mathbb R^n$,
neither of which vanishes almost everywhere, and assume that
\[
 \Tail(u,w,4r)+\Tail(v,w,4r)<\infty.
\]
Suppose that
\begin{equation}\label{eq:target-problem}
 \Fsp u=\Fsp v=0\quad\text{in }D_{4r}(w),\qquad
 u=v=0\quad\text{on }B_{4r}(w)\setminus\Om
\end{equation}
in the local weak sense of Definition~\ref{def:local-zero}.
Then, with $A=A_r(w)$,
\begin{equation}\label{eq:target-bhi}
 C^{-1}\frac{u(A)}{v(A)}
 \le \frac{u(x)}{v(x)}
 \le C\frac{u(A)}{v(A)}
 \qquad\text{for }x\in D_r(w),
\end{equation}
and
\begin{equation}\label{eq:target-holder}
 \left|\log\frac{u(x)}{v(x)}-\log\frac{u(y)}{v(y)}\right|
 \le C\left(\frac{|x-y|}{r}\right)^\alpha
 \qquad\text{for }x,y\in D_r(w).
\end{equation}
\end{thm}

\bigskip
\begin{rem} By Theorem \ref{thm:main},  \(q=u/v\) extends to a positive continuous function
on \(\overline\Om\cap B_r(w)\). By \eqref{eq:target-bhi} and
\eqref{eq:target-holder}, this extension satisfies
\[
 |q(x)-q(y)|\le C\frac{u(A)}{v(A)}
 \left(\frac{|x-y|}{r}\right)^\alpha
 \quad\text{for }x,y\in\overline\Om\cap B_r(w),
\]
and the two-sided comparison in \eqref{eq:target-bhi} also holds there.
\end{rem}

\begin{rem} The geometric assumptions in Theorem~\ref{thm:main} include
several familiar classes of domains.
Every bounded $C^1$ domain is $(\delta,r_0)$-Reifenberg flat
for any prescribed $\delta>0$, provided $r_0$ is chosen
sufficiently small depending on $\delta$ and the domain.
Thus the theorem applies to all bounded $C^1$ domains,
including balls, at sufficiently small scales. Lipschitz domains with sufficiently small Lipschitz constants
also satisfy the required flatness condition.
For example, the graph domain
\[
 \Om_\varphi
 =\{(x',x_n)\in\R^{n-1}\times\R:
          x_n>\varphi(x')\}
\]
is Reifenberg flat with flatness constant bounded by
$C(n)\operatorname{Lip}(\varphi)$.
Taking $\varphi(x')=\varepsilon|x'|$, with $\varepsilon>0$
sufficiently small, gives an example whose boundary is not
$C^1$ at the origin. The Reifenberg condition also allows boundaries that are
not locally Lipschitz graphs.
Suitable small-angle variants of the planar Koch snowflake
have arbitrarily small Reifenberg flatness constants at
sufficiently small scales, while their boundaries have
Hausdorff dimension greater than one.
The approximating hyperplane may vary with both the boundary
point and the scale, and the condition does not require
these planes to converge to a tangent plane.
For further examples and geometric discussion, see
\cite{LNV13}.
\end{rem}

\begin{rem}
For $0<s<1$, the restriction $\trange$ arises from our use of
the interior regularity theorem of Giovagnoli, Jesus, and
Silvestre \cite{GJS}, which establishes $C^{1,\beta}$ regularity
of fractional $p$-harmonic functions for some $\beta>0$
in this range.
\end{rem}

\begin{rem}
The numerical factors in the radii in Theorem~\ref{thm:main} may be adjusted, with corresponding
changes in the constants and the admissible range of $r$.
For example, one may impose the equation and the zero exterior
condition in a larger outer ball and obtain the conclusions in a
smaller concentric ball.
\end{rem}
\begin{rem}\label{rem:global-nonnegativity}
The nonnegativity assumption in Theorem~\ref{thm:main}
is imposed on all of $\R^n$. The distant exterior values
need not agree or vanish, but must satisfy the stated
nonnegativity and tail assumptions. Global nonnegativity
eliminates the negative tails in interior Harnack estimates
and allows us to control the positive tail by an interior
value, with constants independent of the exterior data.
If a solution is nonnegative only in its fractional
$p$-harmonicity region, its negative exterior values
contribute to the estimates. For example, the interior
Harnack inequality in \cite[Theorem~1.1]{DKP} takes
the form
\[
 \sup_{B_\rho(a)}u
 \le C\inf_{B_\rho(a)}u
 +C\left(\frac{\rho}{R}\right)^{sp/(p-1)}
       \Tail(u_-,a,R),
 \qquad 0<\rho<R/2,
\]
when $u$ is nonnegative and fractional $p$-harmonic
in $B_R(a)$, under the hypotheses of that theorem.
Here $u_-=\max\{-u,0\}$. Thus local nonnegativity
does not by itself give normalization solely in
terms of a positive interior value.
\end{rem}

\begin{rem}\label{rem:sign-changing-exterior-data}
There is a boundary obstruction to replacing
global nonnegativity by local positivity, already
for $p=2$ and the case of the fractional Laplacian. The construction in
\cite[Theorem~1.14]{DSV} gives a sign-changing function
$u$, vanishing outside a ball $B$, which, after
changing its sign if necessary, is positive and
fractional harmonic in an interior neighborhood
of $\partial B$, but satisfies
\[
 \lim_{B\ni x\to x_0}\frac{u(x)}{d_B(x)^s}=0
 \qquad\text{at a boundary point }x_0,
\]
where $d_B(x)=\dist(x,B^c)$.
The potential in that construction is compactly
supported inside $B$, so the equation is homogeneous
near the boundary. Comparing $u$ with a globally
nonnegative function that is fractional harmonic
in the same boundary patch and comparable to $d_B^s$
there shows that the lower bound in
\eqref{eq:target-bhi} cannot follow from local
positivity alone. A possible extension would impose quantitative
conditions on the negative exterior data, for example
smallness of
\[
 \frac{\Tail(u_-,w,4r)}{u(A)}
 \quad\text{and}\quad
 \frac{\Tail(v_-,w,4r)}{v(A)},
\]
for solutions positive in the comparison patch.
Whether such conditions suffice for a version of
Theorem~\ref{thm:main} requires additional boundary
non-degeneracy and comparison arguments.
No such extension is asserted here.
\end{rem}

\subsection{Relation to the literature}

Boundary Harnack inequalities for non-local equations have an
extensive history in the linear case. For the fractional Laplacian,
Bogdan \cite{Bog97} established boundary comparison in Lipschitz
domains. The theory subsequently developed to include more
general jump operators and weaker geometric assumptions.
In particular, Bogdan, Kumagai, and Kwa\'snicki \cite{BKK15}
proved boundary Harnack inequalities in arbitrary open sets
for a broad class of Markov processes with jumps, under
structural assumptions on the jump kernel and the generator.
Their results illustrate that non-local interactions can permit
boundary comparison without the geometric restrictions familiar
from local equations.

An analytic treatment for non-local elliptic operators in
non-divergence form was developed by Ros-Oton and Serra
\cite{RSBHI}. They consider operators
\[
 Lu(x)=\int_{\R^n}
 \left(\frac{u(x+z)+u(x-z)}{2}-u(x)\right)
 K(x,z)\,\mathrm d z,
\]
where the measurable kernel is even in $z$ and satisfies
\[
 \lambda |z|^{-n-2s}
 \le K(x,z)\le
 \Lambda |z|^{-n-2s}
\]
for fixed constants $0<\lambda\le\Lambda$.
No continuity of the coefficients in $x$ is required.
For globally nonnegative solutions vanishing on a common
exterior portion, they establish boundary comparison in
arbitrary open sets and H\"older continuity of the quotient
up to the boundary in Lipschitz domains.
More recently, Cao and Chen \cite{CC24} obtained a uniform
boundary Harnack principle in arbitrary open subsets of
metric measure spaces under jump measure comparability,
jump tail estimates, and upper bounds on early-exit
probabilities from balls.
Their results include classes of non-local operators in
both divergence and non-divergence form with measurable
coefficients.

The weighted local theory provides another relevant comparison.
Fabes, Kenig, and Serapioni \cite{FKS} establish Harnack and
H\"older estimates for divergence-form equations whose ellipticity
is controlled by an $A_2$ weight.
At a more abstract level, Barlow and Murugan \cite{BM19}
prove a scale-invariant boundary Harnack principle in inner
uniform domains for regular strongly local Dirichlet forms,
assuming the elliptic Harnack inequality and suitable
properties of Green functions.
Their theorem does not require volume doubling of the
reference measure, and their proof uses elliptic estimates
without quantitative heat kernel bounds.
The jump interactions in our linearized forms persist after
the ground-state transformation, so this strongly local
theory does not apply directly.

In the non-local weighted setting, Behn, Diening, Ok, and
Rolfes \cite{BDOR} develop fractional weighted Sobolev spaces
and prove interior Harnack and H\"older estimates for the
corresponding equations with Muckenhoupt weights.
In their quadratic case, the interaction density relative
to Lebesgue measure is comparable to
\[
 J_\omega(x,y)
 =\frac{\omega(x)\omega(y)}
 {|x-y|^{2\eta}\omega(B_{xy})},
 \qquad 0<\eta<1,\qquad \omega\in A_2,
\]
where $B_{xy}$ is the ball with diameter $[x,y]$ and
$\omega(B)=\int_B\omega\,\mathrm d x$.
For $p>2$, the kernels obtained here by linearization can
degenerate in particular directions, even on interior balls
where the natural weight is comparable to a positive constant.
As explained below, this prevents the uniform two-sided
comparison with the displayed weighted kernel.
Moreover, the weighted measure associated with a variable
solution is quantitatively controlled only in its harmonicity
patch, while distant exterior values remain unrestricted
apart from nonnegativity and the tail condition.

A more flexible framework is provided by symmetric non-local
Dirichlet forms on metric measure spaces.
Chen, Kumagai, and Wang \cite{CKW} study elliptic Harnack
inequalities and their implications for H\"older regularity
in this setting, allowing the scale function to depend on
the center.
Our proof follows the quadratic Harnack strategy of
Chen and Yu \cite{CY}.
We give the bounded quadratic argument needed here in
Lemma~\ref{lem:local-form-harnack}, making the solution
class, the treatment of exterior interactions, and the
dependence of the constants explicit. Applying this argument requires structural estimates for
the weighted forms arising from linearization.
A Poincar\'e inequality must control weighted oscillation
by quadratic energy, while cutoff, jump moment, and kernel
averaging estimates must control the contributions introduced
by localization.
Section~\ref{sec:variable} establishes these estimates in the
boundary patch while retaining interactions with all distant
exterior points.
No doubling assumption is imposed on the weighted measure
outside the controlled patch.
Verifying these properties for the degenerate kernels and
proving the resulting local Harnack inequality with constants
independent of the global quotient bound are principal
analytic steps in the paper.

For nonlinear equations, the interior theory is well developed.
Di Castro, Kuusi, and Palatucci \cite{DKP} established Harnack
inequalities for nonlinear non-local equations with measurable
kernels, including the fractional $p$-Laplacian.
Their estimates make explicit the role of the non-local tail.
Boundary regularity in smooth domains was developed by
Iannizzotto, Mosconi, and Squassina \cite{IMS16,IMS20}.
For bounded right-hand sides and homogeneous exterior Dirichlet
data in bounded $C^{1,1}$ domains, these works establish global
H\"older regularity and, in the degenerate range $p\ge2$,
H\"older continuity of the normalized function $u/d_\Om^s$.
The corresponding fine boundary regularity in the singular range
$1<p<2$ was subsequently proved by Iannizzotto and Mosconi
\cite{IM24}.
When combined with a positive lower bound for the normalized
denominator, such estimates yield H\"older regularity of quotients
in the smooth-domain setting.

Boundary comparison for the fractional $p$-Laplacian has also
been studied directly.
Ataei \cite{Ataei} proved a generalized Hopf lemma and a global
boundary Harnack inequality in bounded open sets, within a
framework that allows Wiener regular boundaries.
His comparison theorem treats positive functions satisfying
suitable differential inequalities and assumes that their
exterior data are bounded and mutually comparable.
It gives two-sided bounds for the quotient, with constants
depending on the domain, the exterior-data bounds, and the
normalization.
The local problem considered here allows the two solutions to
have unrelated nonnegative values away from the common zero
patch.
In particular, their distant exterior values need not be
mutually comparable or bounded.

Two further developments are particularly relevant to the
present argument.
Byun, Kim, and Song \cite{BKS} established boundary regularity
for nonlinear non-local equations in sufficiently flat Reifenberg
domains, including boundary H\"older estimates with any exponent
strictly below $s$, provided the flatness constant is chosen
accordingly.
These are estimates for individual solutions.
Giovagnoli, Jesus, and Silvestre \cite{GJS} proved interior
$C^{1,\alpha}$ regularity for fractional $p$-harmonic functions
in the range $2\le p<{2}/{(1-s)}$. This interior differentiability result supplies an essential
input to our analysis of the linearized equations and explains
the parameter range used in Theorem~\ref{thm:main}.

In Theorem~\ref{thm:main} we establish local, scale-invariant
boundary comparison and H\"older continuity up to the
boundary of quotients of positive fractional \(p\)-harmonic
functions in sufficiently flat Reifenberg domains. The equation and the common zero exterior
condition are imposed only in a boundary patch, while the
distant exterior values may be unrelated and unbounded,
subject to global nonnegativity and the natural
\((p-1)\)-tail condition.
After normalization at a corkscrew point, the comparison
and H\"older constants depend only on \(n,s,p\).
The result holds in the natural local energy and tail
class, without a finite global \(p\)-energy assumption.

The principal analytic difficulty is the degeneracy of
the linearized kernels, which depends on increments of
the background solution at pairs of points.
For \(p>2\), even the half-space kernels lack
the uniform pointwise lower bounds available in the
uniformly elliptic linear theory.
The weighted quadratic-form estimates developed below
address this degeneracy while retaining the full
exterior interactions.

\subsection{Convention on constants} For nonnegative quantities \(X\) and \(Y\), we throughout the paper write
\(X\asymp Y\) if \[
C^{-1}Y\le X\le CY\quad \mbox{for some \(C\ge1\)},
\]
with the dependence of \(C\) specified in the statement. Unless otherwise stated, the constant $C$ will depend at most on
$n,s,p$.

\subsection{Outline of the proof of Theorem~\ref{thm:main}}

The central strategy behind the proof of Theorem~\ref{thm:main} is to obtain nonlinear boundary
comparison by integrating estimates for linearized
equations along families of solutions. Given a positive
fractional \(p\)-harmonic solution \(U\), we first prove
boundary estimates for \(F/U\), where \(F\) solves
the equation obtained by linearizing \(\Fsp\) at \(U\)
and satisfies \(0\le F\le M_q U\) globally for some
finite \(M_q\). The estimates are independent of \(M_q\).  We then apply
these estimates to parameter derivatives of ordered variational
families. Integration gives comparison of the endpoint solutions,
and a common fractional \(p\)-harmonic majorant removes the ordering requirement.
The half-space approximation and the weighted energy estimates
provide the analytic ingredients for this argument.

Two regularity inputs play different roles. The capacitary
estimates of Kim, Lee, and Weidner \cite{KLW} provide boundary
continuity and quantitative boundary decay under an exterior
capacity density condition, which follows from the presence of exterior corkscrews.
The qualitative boundary regularity underlying this approach
is characterized by the non-local Wiener criterion in \cite{KLL}.
For interior differentiability, we use the theorem of
Giovagnoli, Jesus, and Silvestre \cite{GJS}, which gives
$C^{1,\alpha_{\mathrm{int}}}$ regularity for some
$\alpha_{\mathrm{int}}>0$ in the range
\begin{equation}\label{eq:range}
 0<s<1,\qquad 2\le p<{2}/{(1-s)}.
\end{equation}
We impose \eqref{eq:range} whenever this differentiability
result is used. The Dirichlet theory and several of the
basic energy and boundary estimates are stated for
$1<p<\infty$.

A normalization used throughout the proof is the control
of both the local supremum and the positive tail by a
single interior reference value.
Proposition~\ref{prop:boundary-normalization} provides
these bounds by combining the positive-tail estimate
of Lemma~\ref{lem:positive-tail} with boundary local
boundedness. The supremum bound is the Carleson estimate.
The tail bound controls the influence of distant exterior
data without requiring a pointwise bound on those data.
This common normalization enters the half-space comparison,
the localization arguments, the weighted kernel estimates,
and the construction of the common majorant.

We first establish boundary comparison in the upper half-space
\[
 \HH=\{x=(x',x_n)\in\R^n:x_n>0\}.
\]
The basic model is
\[
 h(x)=(x_n)_+^s,
 \qquad (x_n)_+=\max\{x_n,0\}.
\]
This function is fractional $p$-harmonic in $\HH$ and
vanishes on $\HH^c$. It provides the profile against which
we measure the boundary behavior of positive solutions.
Theorem~\ref{thm:flat} and Proposition~\ref{prop:flat-bhi}
establish the half-space comparison for every $p\ge2$.
Using the positive-tail normalization,
the flat boundary upper estimate in \cite[Lemma~3.1]{BKS}
and a tangent-ball comparison give two-sided bounds for
$u/x_n^s$ in terms of the same interior value. A cutoff
then permits the local use of the weighted boundary
regularity estimates of \cite{IMS20}.
Lemma~\ref{lem:localized-oscillation} records the resulting
oscillation argument, including both quotient tails.
Under \eqref{eq:range},
Corollary~\ref{cor:flat-gradient-unconditional} also gives
the boundary asymptotics of the gradient.

The purpose of the half-space analysis is to obtain
quantitative information about positive solutions in
sufficiently flat Reifenberg domains.
Theorem~\ref{thm:approximation} shows that, on an interior
ball whose radius is comparable to its distance from the
boundary (a Whitney ball), a solution normalized at the center is close
to a rotated half-space profile. The approximation is proved by contradiction, using
compactness and classification of global solutions
in the half-space. After rescaling and normalization,
the domains approach a half-space and a subsequence
of the solutions converges locally to a nontrivial
limit \(V\).

A difficulty specific to the non-local equation is that
local convergence does not detect exterior mass escaping
to infinity. After a rigid change of coordinates, the limiting
domain is $\HH$. Lemma~\ref{lem:halfspace-compactness} shows that the
limit \(V\) satisfies
\[
 \Fsp V=\lambda\quad\text{weakly in }\HH,
 \qquad
 V=0\quad\text{almost everywhere on }\HH^c,
\]
where \(\lambda\ge0\) is a constant representing the
possible contribution of exterior mass escaping
to infinity.
Comparison with torsion functions in expanding balls
forces \(\lambda=0\). The half-space boundary comparison,
applied on arbitrarily large scales, then shows that
\(V\) is a multiple of \((x_n)_+^s\).
These conclusions are established in
Lemma~\ref{lem:halfspace-classification}.
The interior normalization fixes the multiple, and
local uniform convergence to this profile contradicts
the assumed failure of approximation.

Assume \eqref{eq:range}, and let \(u\) satisfy the
solution hypotheses of Theorem~\ref{thm:main}.
Thus \(u\) is globally nonnegative and nontrivial,
has finite tail, and satisfies
\[
 \Fsp u=0\quad\text{in }D_{4r}(w),
 \qquad
 u=0\quad\text{on }B_{4r}(w)\setminus\Om
\]
in the local weak sense of
Definition~\ref{def:local-zero} below.
For sufficiently small Reifenberg flatness, combining
the profile approximation discussed above with interior differentiability
gives, by Corollary~\ref{cor:reifenberg-gradient},
\begin{equation}\label{eq:funs}
 C^{-1}\frac{u(x)}{d_\Om(x)}
 \le |\nabla u(x)|
 \le C\frac{u(x)}{d_\Om(x)}
 \qquad\text{for }x\in D_{r/c_1}(w).
\end{equation}
Here \(C,c_1\ge1\) depend only on \(n,s,p\). The same corollary gives a quantitative lower bound
for increments in a cone of directions. More precisely,
there exist \(\tau\in(0,1)\) and \(C\ge1\), depending
only on \(n,s,p\), such that for every
\(x\in D_{r/c_1}(w)\) there is a unit vector \(\nu_x\)
for which
\begin{equation}\label{eq:funs+}
 |u(y)-u(z)|
 \ge C^{-1}\frac{u(x)}{d_\Om(x)}\,|y-z|
\end{equation}
whenever
\begin{equation}\label{eq:funs++}
 y,z\in B_{\tau d_\Om(x)}(x),
 \qquad
 |(y-z)\cdot\nu_x|\ge\frac12|y-z|.
\end{equation}
The direction \(\nu_x\) may depend on \(x\) and \(u\),
whereas the constants are uniform. This estimate
controls the increments appearing in the linearized
kernel and provides the directional lower bounds
used in the subsequent energy comparisons.

Estimate \eqref{eq:funs} is the fractional counterpart
of the fundamental inequality in the local boundary
Harnack theory for the $p$-Laplacian developed in
\cite{LN,LLN}. In that theory, its lower bound excludes
critical points in the comparison region, while the
two-sided estimate controls the ellipticity weight
of the linearized operator through
\[
 |\nabla u|^{p-2}
 \asymp
 \bigl({u}/{d_\Om}\bigr)^{p-2}.
\]
This control is a key ingredient in the construction
of an $A_2$ extension of the weight and the application
of weighted linear boundary Harnack estimates.
Those linear estimates can then be integrated along
deformations of solutions to obtain nonlinear boundary
comparison.

In the present non-local setting,
\eqref{eq:funs} and the accompanying cone increment
estimates \eqref{eq:funs+}-\eqref{eq:funs++} provide the corresponding starting point
for the analysis of the linearized kernels and their
quadratic forms. These increment estimates are needed because the coefficients
of the linearized equation depend on pairs of points.
Indeed, for $p>2$, linearization of $\Fsp$ about a positive
solution $U$ produces the non-local equation with kernel
\begin{equation}\label{eq:intro-kernel}
 K_U(x,y)=(p-1)
 \frac{|U(x)-U(y)|^{p-2}}{|x-y|^{n+sp}}.
\end{equation}
The directional degeneracy is already present for the
half-space profile $h(x)=(x_n)_+^s$. Fix an interior ball
$B_\rho(a)$ with $2\rho<a_n$. The heights of points in this
ball are comparable to $a_n$, and the mean value theorem gives
\[
 |h(x)-h(y)|\asymp a_n^{s-1}|x_n-y_n|
 \qquad(x,y\in B_\rho(a)).
\]
Consequently, after separating the radial and height factors,
the linearized kernel contains the angular factor $({|x_n-y_n|}/{|x-y|})^{p-2}$. For $p>2$, this factor becomes arbitrarily small as the
direction $x-y$ approaches a direction parallel to
$\partial\HH$. Such small values occur on sets of pairs
of positive measure, even though $h$ is comparable to the
positive constant $a_n^s$ throughout the ball.
Thus the natural weight does not account for the directional
degeneracy, and the two-sided kernel comparison required
in \cite{BDOR} is unavailable even for this basic model.
In particular, gradient non-degeneracy does not supply a
uniform pointwise lower bound for the linearized kernel.

Instead, as we show below, the cone increment estimates
\eqref{eq:funs+}-\eqref{eq:funs++} yield lower bounds
for the associated quadratic forms on Whitney balls.
These bounds allow the quadratic energy to control
the relevant fractional Sobolev seminorms, despite
the possible vanishing of the kernel on individual pairs.
The weighted local theory in \cite{LN,LLN} provides the model
for our strategy, while the passage from directional increment
bounds to coercivity of the non-local quadratic forms requires
a separate argument.

The next step is to study the quotient \(F/U\) on
\(S=\{U>0\}\), where \(F\) solves the linearized equation
\(\mathcal L_UF=0\) and satisfies \(0\le F\le M_q U\)
globally for some finite \(M_q\).
Here the kernel of \(\mathcal L_U\) is \(K_U\) from
\eqref{eq:intro-kernel}. Since \(U\) solves
\(\Fsp U=0\), homogeneity gives
\begin{equation}\label{eq:funs++1}
 \mathcal L_UU=(p-1)\Fsp U=0
\end{equation}
in the weak sense in the same harmonicity region.
Moreover, the domination condition forces $F=0$ on the
zero set of $U$. Set $q=F/U$ on $S$. For an admissible
test function $\varphi$ compactly supported in this region,
we test the equation for $F$ with $U\varphi$, test the
equation for $U$ with $F\varphi$, and subtract. This gives
\begin{equation}\label{eq:funs++2}
 \frac12\iint_{S\times S}
 (q(x)-q(y))(\varphi(x)-\varphi(y))
 \widetilde K_U(x,y)\,\mathrm d x\,\mathrm d y=0,
 \qquad S=\{U>0\},
\end{equation}
where
\begin{equation}\label{eq:funs++3}
 \widetilde K_U(x,y)=U(x)U(y)K_U(x,y).
\end{equation}
This change of unknown and weak formulation is the
ground-state transformation based on $U$. It converts
comparison of $F$ with $U$ into regularity of their quotient,
with kernel $\widetilde K_U$ relative to Lebesgue measure.

Section~\ref{sec:variable} develops boundary estimates
for the transformed equation \eqref{eq:funs++2}-\eqref{eq:funs++3}. We express its quadratic
form using the weighted measure
\(\mathrm d\mu_U=U\,\mathrm d x\) on \(S\).
Since
\begin{equation}\label{eq:funs++4}
 \widetilde K_U(x,y)\,\mathrm d x\,\mathrm d y
 =
 K_U(x,y)\,\mathrm d\mu_U(x)\,\mathrm d\mu_U(y),
\end{equation}
the kernel relative to \(\mu_U\) is \(K_U\).
For a measurable function \(f:S\to\R\), the associated
quadratic energy is
\begin{equation}\label{eq:funs++5}
 \mathscr E_U(f)
 =
 \frac12\iint_{S\times S}
 |f(x)-f(y)|^2K_U(x,y)
 \,\mathrm d\mu_U(x)\,\mathrm d\mu_U(y).
\end{equation}
We write $\mathscr E_{U,G}(f)$ when both integration
variables are restricted to $S\cap G$.

For a variable positive solution $U$, there is also a
localization issue. The measure $U\,\mathrm d x$ is
quantitatively controlled only in the fractional $p$-harmonicity patch,
while the distant exterior values of $U$ may be arbitrary
nonnegative data with finite tail. The energy
$\mathscr E_U$ nevertheless includes interactions with
all of $S$. We must therefore establish the required
weighted estimates locally while retaining the full
exterior interactions.

To obtain boundary estimates for the quotient \(q\),
we develop local regularity estimates for weak solutions
associated with the quadratic form \(\mathscr E_U\).
A basic ingredient is a weighted Poincar\'e inequality,
which relates variation on a boundary ball to quadratic
energy on a fixed enlargement of that ball. For \(z\in\partial\Om\), write
\(\mathbb B_r(z)=\overline\Om\cap B_r(z)\).
With respect to the measure
\(\mathrm d\mu_U=U\,\mathrm d x\), the weighted average
of \(f\) on this ball is
\[
 f_{\mathbb B_r(z),\mu_U}
 =
 {\biggl (\displaystyle\int_{\Om\cap B_r(z)}
             f(x)U(x)\,\mathrm d x\biggr )}\bigl /
      {\biggl (\displaystyle\int_{\Om\cap B_r(z)}
             U(x)\,\mathrm d x\biggr )}.
\]
We prove the weighted Poincar\'e inequality
\begin{equation}\label{eq:funs++6}
 \int_{\Om\cap B_r(z)}
 |f(x)-f_{\mathbb B_r(z),\mu_U}|^2
 U(x)\,\mathrm d x
 \le
 C\Psi_U(z,r)
 \mathscr E_{U,\mathbb B_{\kappa r}(z)}(f),
\end{equation}
where \(\kappa\ge2\) is fixed and
\[
 \Psi_U(z,r)
 =
 \frac{r^{sp}}{U(A_r(z))^{p-1}}.
\]
The left-hand side measures the deviation of \(f\)
from its weighted average in \(\mathrm L^2(\mu_U)\).
The right-hand side is the quadratic energy of \(f\)
on the enlarged ball, multiplied by the scale factor
\(\Psi_U(z,r)\).
Thus small energy, with this normalization, forces
\(f\) to be close to a constant in the weighted
\(\mathrm L^2\) sense. This supplies quantitative
coercivity even though the kernel may vanish on
individual pairs of points.

The normalization is consistent with multiplication of $U$
by $\lambda>0$ as the energy is multiplied by $\lambda^p$
and $\Psi_U$ by $\lambda^{1-p}$. Both sides of
\eqref{eq:funs++6} are therefore multiplied by $\lambda$,
while the weighted average is unchanged.

The scale \(\Psi_U\) also controls the interactions
described by the transformed kernel. For
\(z\in\partial\Om\), \(x\in D_r(z)\), and radii within
the comparison patch, we prove the jump moment estimate
\begin{equation}\label{eq:funs++7}
 \frac1{r^2}\int_{S\cap B_r(x)}
 |x-y|^2K_U(x,y)\,\mathrm d\mu_U(y)
 +
 \int_{S\setminus B_r(x)}
 K_U(x,y)\,\mathrm d\mu_U(y)
 \le \frac{C}{\Psi_U(z,r)}.
\end{equation}
The first term is the normalized second moment of
interactions at distances smaller than \(r\).
The second is the total kernel mass of interactions
at distances at least \(r\). Both integrals use
the weight \(\mathrm d\mu_U(y)=U(y)\,\mathrm d y\). This estimate also controls the energy introduced
by localization. Let \(0\le\eta\le1\) be a cutoff
with Lipschitz constant at most \(C/r\).
For \(|x-y|<r\), its squared difference is bounded
by \(C|x-y|^2/r^2\), while for \(|x-y|\ge r\)
it is bounded by \(1\). Consequently,
\begin{equation}\label{eq:funs++10}
 \Gamma_U(\eta)(x)
 :=
 \int_S|\eta(x)-\eta(y)|^2
 K_U(x,y)\,\mathrm d\mu_U(y)
 \le \frac{C}{\Psi_U(z,r)}.
\end{equation}
Thus the jump moment estimate in \eqref{eq:funs++7}
bounds both distant kernel interactions and the energy
cost of cutoff functions. Together with the volume
estimates and the Poincar\'e inequality, it supplies
the controls needed in the weighted Harnack and
oscillation arguments.

To apply these arguments at boundary points, we also
prove that the boundary within the comparison patch
has zero capacity with respect to the weighted
quadratic form \(\mathscr E_U\). We adjoin the boundary to the weighted space and
assign it zero \(\mu_U\)-measure. The zero-capacity
argument shows that the weak equation for bounded
quotients remains valid for admissible test functions
whose supports meet the boundary. The weighted Harnack
and oscillation estimates can therefore be applied
on relative balls centered at boundary points.
This allows regularity estimates for the transformed
equation to yield regularity of \(q=F/U\) up to
the boundary.

The resulting statement,
Theorem~\ref{thm:variable-linear-bhi}, applies when
\(F\) solves \(\mathcal L_UF=0\) and satisfies
\(0\le F\le M_q U\) globally for some finite \(M_q\).
It gives
\begin{equation}\label{eq:funs++11uu}
 C^{-1}\frac{F(A)}{U(A)}
 \le \frac{F(x)}{U(x)}
 \le C\frac{F(A)}{U(A)}
\end{equation}
throughout the comparison region, where \(A\) is
the fixed interior reference point, together with
H\"older estimates for \(F/U\) with the same
normalization. The displayed inequality is the
boundary comparison estimate for \(F\) and \(U\).
The comparison and H\"older constants are independent
of \(M_q\). This uniformity is essential for the
subsequent deformation argument, where the available
domination bound depends on the parameter and may
become arbitrarily large.

We next explain how to recover nonlinear comparison
from these linear estimates. To do this we use a deformation of solutions
analogous to that in \cite{LN07} for the local $p$-Laplace
equation. Let $U_0\le U_1$ be globally ordered variational
solutions in a common comparison domain $D$, satisfying the
required positivity and boundary hypotheses. For $0\le t\le1$,
let $U_t$ be the variational solution of $\Fsp U_t=0$ in $D$
with exterior data
\begin{equation}\label{eq:funs++11}
 g_t=(1-t)U_0+tU_1
 \quad\text{on }D^c.
\end{equation}
Comparison and homogeneity give
\begin{equation}\label{eq:funs++12}
 U_\tau\le U_t\le\frac{t}{\tau}U_\tau,
 \qquad 0<\tau\le t\le1.
\end{equation}
Proposition~\ref{prop:ordered-derivative} constructs
weak parameter derivatives \(F_t=\partial_tU_t\)
satisfying
\begin{equation}\label{eq:funs++13}
 \mathcal L_{U_t}F_t=0
 \quad\text{in }D,
 \qquad
 0\le F_t\le\frac{U_t}{t}
 \quad\text{on }\R^n
\end{equation}
for almost every \(t\in(0,1)\).
The bounded-quotient theorem therefore applies to
\(F_t/U_t\) with \(M_q=1/t\), and its constants remain
uniform as \(t\downarrow0\).

Theorem~\ref{ass:paths} justifies pointwise integration
in the comparison patch and gives
\begin{equation}\label{eq:funs++14}
 \log\frac{U_1(x)}{U_0(x)}
 =
 \int_0^1\frac{F_t(x)}{U_t(x)}\,\mathrm d t.
\end{equation}
At the interior reference point \(A\), this identity
shows that the normalizing values \(F_t(A)/U_t(A)\)
are integrable in \(t\). We can consequently integrate
the linear comparison \eqref{eq:funs++11uu} and H\"older estimates  over the
whole parameter interval. This proves the ordered
boundary comparison and logarithmic H\"older estimate
in Proposition~\ref{prop:ordered-completion}.
The argument requires no \(C^1\) dependence in the
variational energy space.

To compare solutions with unrelated exterior values, we first
work in the finite-energy class. For two such solutions $u,v$
satisfying the remaining hypotheses of Theorem~\ref{thm:main},
we construct a common fractional $p$-harmonic majorant $W$.
It is the variational solution in a smaller boundary patch
with exterior data $\max\{u,v\}$. Comparison gives $W\ge u,v$.
The essential additional step is to control the
size of \(W\) at the reference point. After normalizing
\(u(A)=v(A)=1\), Proposition~\ref{prop:max-lift} gives
\begin{equation}\label{eq:funs++14a}
 1\le W(A)\le K_0
\end{equation}
with a uniform constant \(K_0\).
Its proof uses a supersolution that is constant on
a surrounding ball and agrees with the prescribed
data outside that ball. The required height of the
constant is controlled by the local supremum and
positive-tail estimates for the original solutions. We then apply ordered comparison to the pairs
\((u,2W)\) and \((v,2W)\) and use the identity
\begin{equation}\label{eq:funs++15}
 \log\frac{u}{v}
 =
 \log\frac{2W}{v}-\log\frac{2W}{u}.
\end{equation}
The bound for \(W(A)\) controls the reference values
in both ordered comparisons. We thereby obtain boundary comparison and H\"older
estimates for two finite-energy solutions without
requiring comparability of their distant exterior
values. These are the estimates of the main theorem,
proved at this stage under an additional energy
assumption. A final approximation argument then
removes this assumption.

To formulate the finite-energy version precisely,
let \(D\subset\R^n\) be a bounded open set.
We denote by \(V^{s,p}(D)\) the space of measurable
functions \(f:\R^n\to\R\) such that
\(f|_D\in\mathrm L^p(D)\) and
\[
 \iint_{D\times\R^n}
 \frac{|f(x)-f(y)|^p}{|x-y|^{n+sp}}
 \,\mathrm d y\,\mathrm d x<\infty.
\]
This condition requires finite \(p\)-energy for
interactions within \(D\) and between \(D\) and
its complement, without prescribing the exterior
values. The following theorem states the result
under this additional assumption.

\begin{thm}\label{thm:intro-finite}
Let \(u,v,\Om,w,r\) satisfy the hypotheses of
Theorem~\ref{thm:main}, and assume in addition that
\[
 u,v\in V^{s,p}(D_{2r}(w)).
\]
For sufficiently small \(\delta\), depending only on
\(n,s,p\), estimates \eqref{eq:target-bhi} and
\eqref{eq:target-holder} hold with constants depending
only on \(n,s,p\).
In particular, these constants are independent of
the energy norms of \(u\) and \(v\).
\end{thm}

Theorem~\ref{thm:intro-finite} is proved in
Section~\ref{sec:completion}.
The independence of its constants from the energy
norms is essential for the passage to
Theorem~\ref{thm:main}.
Lemma~\ref{lem:natural-tail-approximation} constructs
finite-energy approximants by truncating the distant
exterior data and taking fractional $p$-harmonic replacements.
These approximants preserve the weak exterior zero
condition and converge both locally uniformly and
in the natural tail norm. The comparison estimates
therefore pass to the limit without requiring a
uniform bound on the \(p\)-energies of the approximants.

\subsection{Organization of the paper}

Section~\ref{sec:dirichlet} introduces the geometric assumptions
and develops the Dirichlet theory.
Section~\ref{sec:basic} establishes the non-local counterparts
of the basic estimates in Section~2.1 of \cite{LN}.
Section~\ref{sec:flat} proves boundary comparison in a half-space
and derives gradient estimates from weighted boundary regularity.
Section~\ref{sec:transfer} establishes relative approximation
by half-space profiles in Reifenberg domains and the resulting
gradient bounds.
Section~\ref{sec:linearization} develops the linearized equations
and analyzes the exact half-space model.
Section~\ref{sec:variable} establishes structural estimates
for variable background solutions and proves comparison
and H\"older estimates for bounded quotients.
Section~\ref{sec:completion} combines the ordered deformation
argument with the maximum-lift estimate to prove finite-energy
comparison, then uses approximation to pass to the natural
tail class.
Finally, Section~\ref{sec:concluding} discusses weaker geometric
assumptions and extensions to more general non-local operators
of $p$-Laplace type.

\section{Geometry and the exterior Dirichlet problem}
\label{sec:dirichlet}

This section establishes the geometric and functional framework
for the boundary comparison arguments.
We first recall Reifenberg flatness and the geometric consequences
used below, including corkscrew points and exterior capacity density.
We then introduce the energy and tail spaces and specify the weak
formulation of the fractional $p$-Laplace equation.
Since the operator is non-local, Dirichlet data are prescribed
on the complement of the domain, and the formulation must retain
interactions with the entire exterior.

We distinguish the variational Dirichlet problem, formulated
in a finite-energy class, from the local weak formulation
with zero exterior values on a boundary patch used in
Theorem~\ref{thm:main}.
We establish existence, uniqueness, and comparison for the
variational problem and discuss continuity up to the boundary
under exterior capacity density.
Finally, we construct finite-energy approximations of solutions
in the natural tail class.
These approximations preserve the local exterior zero condition
and provide the convergence needed to pass the boundary comparison
estimates to the class of solutions in the main theorem.
\subsection{Domains and capacity}

\begin{defn}\label{def:reifenberg}
Let \(\delta\in(0,1/8)\) and \(r_0>0\).
An open set \(\Om\subsetneq\R^n\) is
\((\delta,r_0)\)-Reifenberg flat if, for every \(w\in\partial\Om\)
and \(0<r\le r_0\), there is an orthonormal coordinate system centered
at \(w\) such that
\begin{equation}\label{eq:reifenberg}
 B_r\cap\{x_n>\delta r\}
 \subset \Om\cap B_r
 \subset B_r\cap\{x_n>-\delta r\}.
\end{equation}
\end{defn}

\begin{rem}
This definition includes the separation condition. The coordinate
system may vary with \(w\) and \(r\). No surface measure assumption
is imposed.
\end{rem}

\begin{defn}\label{def:nta}
Let \(M_{\mathrm{geom}}\ge2\) and \(r_0>0\).
A domain \(\Om\) is NTA up to scale \(r_0\) if both \(\Om\) and
\(\R^n\setminus\overline\Om\) have corkscrew balls of radius \(r/M_{\mathrm{geom}}\)
inside \(B_r(w)\), for \(w\in\partial\Om\) and \(0<r<r_0\),
and \(\Om\) has the following local uniformity property.
For every \(w\in\partial\Om\), \(0<r<r_0\), and
\(x,y\in D_r(w)\), there is a rectifiable curve
\(\gamma\subset\Om\) joining \(x\) to \(y\), of length at most
\(M_{\mathrm{geom}}|x-y|\), such that
\[
 \min\{\operatorname{length}(\gamma_{x,z}),
        \operatorname{length}(\gamma_{z,y})\}
 \le M_{\mathrm{geom}}d_\Om(z)\qquad(z\in\gamma).
\]
Here \(\gamma_{x,z}\) and \(\gamma_{z,y}\) are the two subarcs.
This is the equivalent local formulation of the Harnack chain
condition used in \cite[Definition~1.3]{LN}.
\end{defn}

\begin{rem}\label{rem:localized-chains}
For sufficiently small \(\delta\), Reifenberg flat domains are NTA
with controlled constants, after a fixed reduction of the scale.
We use this standard geometric fact in the form recorded in \cite{LN}.
For the fixed changes of radii at the end of the proof, we use the
following explicit chain containment. Let \(w\in\partial\Om\), \(0<\theta<1/8\),
\(3r<r_0\), and \(\delta\le\theta/16\). If
\(x,y\in D_{3r/2}(w)\) and
\(d_\Om(x),d_\Om(y)\ge\theta r\), they can be joined by at most
\(C/\theta\) overlapping balls of radius \(\theta r/32\),
whose doubles lie in \(D_{2r}(w)\).
Indeed, in the coordinates of \eqref{eq:reifenberg} at scale
\(3r\), both points have height at least \(\theta r-3\delta r\).
Their joining segment lies in \(B_{3r/2}(w)\) and has distance
at least \(\theta r-6\delta r\ge\theta r/2\) from \(\Om^c\).
Covering that segment gives the stated chain. Thus all interior
Harnack comparisons along this chain remain inside the indicated
harmonicity patch.
\end{rem}

For a compact set \(E\subset B_{2r}(w)\), define
\begin{equation}\label{eq:capacity}
 \capsp(E,B_{2r}(w))
 =\inf\left\{
 \iint_{\R^{2n}}\frac{|\varphi(x)-\varphi(y)|^p}{|x-y|^{n+sp}}
 \dxdy\,\middle|\,
 \varphi\in C_0^\infty(B_{2r}(w)),\ \varphi\ge1\text{ on }E
 \right\}.
\end{equation}
We put
\begin{equation}\label{eq:theta}
 \Theta_w(r)=
 \left(\frac{\capsp(\overline B_r(w)\setminus\Om,B_{2r}(w))}
 {r^{n-sp}}\right)^{1/(p-1)}.
\end{equation}

\begin{lem}\label{lem:capacity}
Suppose that \(\Om\) has exterior corkscrew balls with constant \(M_{\mathrm{geom}}\).
Then there exists \(\theta_0=\theta_0(n,s,p,M_{\mathrm{geom}})>0\) such that
\(\Theta_w(r)\ge\theta_0\) at every allowed point and scale.
In particular, this holds uniformly for sufficiently flat
Reifenberg domains.
\end{lem}

\begin{proof}
An exterior corkscrew ball gives
\(|B_r(w)\setminus\Om|\ge c(n,M_{\mathrm{geom}})r^n\).
For every admissible $\varphi$ in \eqref{eq:capacity},
the fractional Poincar\'e inequality for functions supported
in a ball, see \cite[Lemma~2.4]{BLP14}, gives
\[
 \iint_{\R^{2n}}\frac{|\varphi(x)-\varphi(y)|^p}{|x-y|^{n+sp}}
 \dxdy
 \ge c r^{-sp}\int_{B_{2r}(w)}|\varphi|^p\d x
 \ge c(n,s,p,M_{\mathrm{geom}})r^{n-sp}.
\]
Taking the infimum proves the claim. The measure lower bound
also follows directly from \eqref{eq:reifenberg}.
\end{proof}

\subsection{Energy spaces and weak conditions}

Throughout this subsection, let $0<s<1$ and $1<p<\infty$.
For an open set $D\subset\R^n$ and a measurable function
$u:D\to\R$, we define the Gagliardo seminorm
\[
 [u]_{W^{s,p}(D)}
 =
 \left(
 \iint_{D\times D}
 \frac{|u(x)-u(y)|^p}{|x-y|^{n+sp}}
 \,\mathrm d y\,\mathrm d x
 \right)^{1/p}.
\]
The fractional Sobolev space $W^{s,p}(D)$ is
\[
 \Ws(D)
 =
 \left\{
 u\in\mathrm L^p(D):
 [u]_{W^{s,p}(D)}<\infty
 \right\},
\]
equipped with the norm
\[
 \|u\|_{W^{s,p}(D)}
 =
 \left(
 \|u\|_{\mathrm L^p(D)}^p
 +
 [u]_{W^{s,p}(D)}^p
 \right)^{1/p}.
\]
We write $u\in\Ws_{\mathrm{loc}}(D)$ if
$u|_{D'}\in\Ws(D')$ for every open set $D'\Subset D$.
For functions defined on all of $\R^n$, we abbreviate
\[
 [u]_{s,p}=[u]_{W^{s,p}(\R^n)}.
\]
The space encoding zero exterior Dirichlet values is
\[
 \Wzero(D)
 =
 \overline{C_0^\infty(D)}^{\,W^{s,p}(\R^n)},
\]
where functions in $C_0^\infty(D)$ are extended by zero
outside $D$.
In particular, every element of $\Wzero(D)$ vanishes
almost everywhere on $D^c=\R^n\setminus D$.
For arbitrary open sets, we retain this closure definition
as part of the weak boundary condition. For a bounded open set $D$, we put
\[
 Q_D=\R^{2n}\setminus(D^c\times D^c).
\]
Thus $Q_D$ consists of pairs for which at least one point
belongs to $D$.
We also recall the variational energy space
\begin{equation}\label{eq:V-space}
 V^{s,p}(D)
 =
 \left\{
 \begin{array}{l|l}
 u:\R^n\to\R \text{ measurable}
 &
 \displaystyle
 \int_D |u(x)|^p\,\mathrm d x
 +
 \int_D\int_{\R^n}
 \frac{|u(x)-u(y)|^p}{|x-y|^{n+sp}}
 \,\mathrm d y\,\mathrm d x<\infty
 \end{array}
 \right\}.
\end{equation}
By symmetry, finiteness of the interaction integral in
\eqref{eq:V-space} is equivalent to finiteness of the
corresponding integral over $Q_D$.
Thus $\Vs(D)$ includes interactions between $D$ and its
complement, whereas the seminorm defining $\Ws(D)$
involves only pairs in $D\times D$.
Membership in $\Vs(D)$ does not prescribe the exterior values.
We equip this space with the norm
\[
 \|u\|_{\Vs(D)}
 =\left(\int_D|u|^p\d x+
 \iint_{Q_D}\frac{|u(x)-u(y)|^p}{|x-y|^{n+sp}}\dxdy
 \right)^{1/p}.
\]
This is equivalent to the norm obtained from the integral over
$D\times\R^n$ in \eqref{eq:V-space}.
The natural tail space is
\[
 \mathrm L^{p-1}_{sp}(\R^n)
 =
 \left\{
 u\in\mathrm L^{p-1}_{\mathrm{loc}}(\R^n):
 \int_{\R^n}
 \frac{|u(y)|^{p-1}}{(1+|y|)^{n+sp}}
 \,\mathrm d y<\infty
 \right\}.
\]
Functions in this space have finite tails
$\Tail(u,z,r)$ for every $z\in\R^n$ and $r>0$. Finally, we write
\[
 \Jp(a)=|a|^{p-2}a\quad\text{for }a\ne0,
 \qquad
 \Jp(0)=0.
\]

\begin{defn}\label{def:weak}
Let \(u\in\Ws_{\mathrm{loc}}(D)\cap \mathrm L^{p-1}_{sp}(\R^n)\).
The function \(u\) is a weak solution of \(\Fsp u=0\) in \(D\) if
\begin{equation}\label{eq:weak}
 \frac12\iint_{\R^{2n}}
 \frac{\Jp(u(x)-u(y))(\varphi(x)-\varphi(y))}
 {|x-y|^{n+sp}}\dxdy=0
 \qquad\text{for }\varphi\in C_0^\infty(D).
\end{equation}
It is a weak supersolution if the left side is nonnegative for
nonnegative \(\varphi\), and a weak subsolution if it is nonpositive.
\end{defn}

\begin{rem}The factor \(1/2\) is consistent with \eqref{eq:operator}.
For \(u\in\Vs(D)\), the weak formulation extends by density to
\(\varphi\in\Wzero(D)\).
\end{rem}

\begin{defn}\label{def:local-zero}
Let $w\in\partial\Om$, $R>0$, and $D_R(w)=\Om\cap B_R(w)$. A function \(u\in\Ws(B_R(w))\cap \mathrm L^{p-1}_{sp}(\R^n)\)
satisfies the local homogeneous exterior condition if
\(u=0\) almost everywhere on \(B_R(w)\setminus\Om\) and
\[
 \eta u\in\Wzero(D_R(w))
 \qquad\text{for every }\eta\in C_0^\infty(B_R(w)).
\]
A local weak solution with zero exterior values is such a function
that satisfies \eqref{eq:weak} in \(D_R(w)\).
\end{defn}

The approximation requirement in Definition~\ref{def:local-zero}
retains information on sets of positive capacity and zero Lebesgue
measure. It must not be replaced on arbitrary domains by an almost
everywhere condition alone. For nonzero localized data we use the
weak exterior condition of \cite[Section~3]{KLW}.

\subsection{Variational solutions}

\begin{thm}\label{thm:dirichlet}
Let \(D\subset\R^n\) be bounded and open, \(1<p<\infty\),
and \(g\in\Vs(D)\).
Then there is a unique \(u\in g+\Wzero(D)\) such that
\(\Fsp u=0\) weakly in \(D\).
It is the unique minimizer of
\begin{equation}\label{eq:energy}
 \mathcal E_D(v)=\frac1{2p}
 \iint_{Q_D}\frac{|v(x)-v(y)|^p}{|x-y|^{n+sp}}\dxdy
 \quad\text{over }g+\Wzero(D).
\end{equation}
\end{thm}

\begin{proof}
The competitor \(g\) has finite energy. For \(v=g+z\),
\(z\in\Wzero(D)\), the elementary power inequality gives
\[
 [z]_{s,p}^p
 \le C_p\iint_{Q_D}
 \frac{|v(x)-v(y)|^p+|g(x)-g(y)|^p}{|x-y|^{n+sp}}\dxdy.
\]
Here the increments of \(z\) vanish on \(D^c\times D^c\).
The fractional Poincar\'e inequality controls
\(\|z\|_{\mathrm L^p(\R^n)}\) by \([z]_{s,p}\).
For a minimizing sequence \(v_j=g+z_j\), the corrections
\(z_j\) are therefore bounded in \(\Ws(\R^n)\).
After passing to a subsequence, they converge weakly to
\(z\in\Wzero(D)\). The increment map on \(Q_D\) is bounded
from this space into the corresponding weighted \(L^p\) space.
Weak lower semicontinuity of the energy thus shows that
\(u=g+z\) is a minimizer. Differentiation along \(u+t\varphi\), followed by H\"older's inequality,
gives \eqref{eq:weak} for \(\varphi\in\Wzero(D)\).
Conversely, the supporting hyperplane inequality for the convex
function \(|a|^p\) shows that every weak solution in the affine class
minimizes \(\mathcal E_D\). Finally, if there were two minimizers, strict convexity would force their
increments to agree almost everywhere on \(Q_D\).
Their difference vanishes on \(D^c\), so it also vanishes in \(D\).
This proves uniqueness.
\end{proof}

\begin{prop}\label{prop:comparison}
Let \(u,v\in\Vs(D)\) be respectively a weak subsolution and a weak
supersolution in \(D\). If \((u-v)_+\in\Wzero(D)\), then \(u\le v\)
almost everywhere in \(\R^n\).
\end{prop}

\begin{proof}
Subtract the weak inequalities and test with \((u-v)_+\).
For almost every pair \((x,y)\), the differences
\[
 \Jp(u(x)-u(y))-\Jp(v(x)-v(y))
 \quad\text{and}\quad
 (u-v)_+(x)-(u-v)_+(y)
\]
have the same sign whenever the second is nonzero.
Their product is therefore nonnegative. Its integral is
nonpositive, so it vanishes.
If \((u-v)_+\) were positive on a set of positive measure, pairs
between that set and \(D^c\) would give a strictly positive
contribution. Thus \((u-v)_+=0\).
\end{proof}

The usual truncation and order ideal properties of \(\Wzero(D)\)
justify the comparison test for ordered weak exterior values.
These are also the approximation properties used in
\cite[Section~3]{KLW}.

\begin{cor}\label{cor:maximum}
Let \(u_i\) be the solutions associated with bounded data
\(g_i\in\Vs(D)\), \(i=1,2\).
Then
\[
 \|u_1-u_2\|_{\mathrm L^\infty(\R^n)}
 \le\|g_1-g_2\|_{\mathrm L^\infty(\R^n)}.
\]
In particular,
\(\inf_{\R^n}g\le u\le\sup_{\R^n}g\), and nonnegative data give a
nonnegative solution.
\end{cor}

\begin{proof}
Put \(a=\|g_1-g_2\|_\infty\).
The functions \(u_2+a\) and \(u_2-a\) solve the same homogeneous
equation as \(u_2\). Their weak exterior values bound those of \(u_1\).
Apply Proposition~\ref{prop:comparison} twice.
Comparison with constants proves the last assertion.
\end{proof}

\begin{rem}\label{rem:spaces-and-solutions}
The space \(\Ws(D)\) measures fractional energy within \(D\),
whereas \(\Vs(D)\) also includes interactions with \(D^c\).
For bounded \(D\),
\[
 \Ws(\R^n)\subset\Vs(D),
 \qquad
 u\in\Vs(D)\ \Longrightarrow\ u|_D\in\Ws(D).
\]
Membership in \(\Vs(D)\) neither prescribes exterior values
nor requires energy on \(D^c\times D^c\) or global
\(\mathrm L^p\) integrability. Thus constants belong to
\(\Vs(D)\), although nonzero constants do not belong to
\(\Ws(\R^n)\). The condition \(u-g\in\Wzero(D)\)
specifies the prescribed data through the closure definition
of \(\Wzero(D)\). On arbitrary open sets, it contains more
boundary information than equality almost everywhere on \(D^c\).
\end{rem}

\begin{rem}\label{rem:weak-and-variational-solutions}
By Theorem~\ref{thm:dirichlet}, a variational solution is a
weak solution in the prescribed finite-energy class
\(g+\Wzero(D)\). A local weak solution requires only
\[
 u\in\Ws_{\mathrm{loc}}(D)
       \cap\mathrm L^{p-1}_{sp}(\R^n)
\]
and the equation against compactly supported test functions.
Local fractional energy controls nearby interactions,
while the \((p-1)\)-tail controls distant interactions,
reflecting the growth of \(\Jp\).
No finite \(p\)-energy on \(Q_D\) is required.
For boundary estimates we additionally impose
Definition~\ref{def:local-zero}.
Lemma~\ref{lem:natural-tail-approximation} provides finite-energy
approximations in this local class, allowing comparison
estimates independent of the approximating energies to pass
to the limit.
\end{rem}

\subsection{Continuous exterior data}

\begin{thm}\label{thm:continuous}
Let \(D\) be a bounded NTA domain and let \(g\in C_b(\R^n)\).
Then there is a unique bounded function
\(u\in\Ws_{\mathrm{loc}}(D)\cap C(\R^n)\) such that
\(u=g\) on \(D^c\) and \(\Fsp u=0\) weakly in \(D\).
If, in addition, \(g\in\Vs(D)\), then \(u\) agrees
with the variational solution in \(g+\Wzero(D)\).
\end{thm}

\begin{proof}
Choose bounded smooth functions \(g_j\) with
\(\|g_j-g\|_\infty\to0\). Smooth approximation with a partition of
unity gives such a sequence even when \(g\) is not uniformly
continuous on \(\R^n\).
Each \(g_j\) belongs to \(\Vs(D)\).
Let \(u_j\) be the variational solution from
Theorem~\ref{thm:dirichlet}.

By Lemma~\ref{lem:capacity}, \(D\) satisfies an exterior capacity
density condition. The boundary estimates in \cite{KLW}, together with
interior continuity, give \(u_j\in C(\overline D)\) with boundary
values \(g_j\). Set \(u_j=g_j\) on \(D^c\).
Corollary~\ref{cor:maximum} makes \(u_j\) uniformly Cauchy on
\(\R^n\), with a uniform bound. Its limit \(u\) is continuous and
takes the prescribed exterior values.

The local Caccioppoli estimate, proved in
Lemma~\ref{lem:caccioppoli} below, gives uniform energy bounds
for \(u_j\) on compact subsets of \(D\). On a fixed interior ball, the increments of \(u_j\) converge almost
everywhere and their \(p\)-energies are bounded.
Consequently their nonlinear fluxes converge weakly in the
corresponding \(\mathrm L^{p/(p-1)}\) space. The interactions outside a
larger interior ball pass to the limit by the uniform bound and
the integrable tail of the kernel.
This proves \eqref{eq:weak} for \(u\).

For uniqueness, let \(u,v\) be two bounded continuous solutions.
For every \(\varepsilon>0\), the function
\((u-v-\varepsilon)_+\) has compact support in \(D\), because
\(u=v\) on \(D^c\) and both are continuous.
Local energy regularity makes it an admissible test in the
difference of their equations. The proof of
Proposition~\ref{prop:comparison} gives \(u-v\le\varepsilon\).
Letting \(\varepsilon\downarrow0\) and exchanging \(u,v\) proves
uniqueness.

If \(g\in\Vs(D)\), the variational solution attains its continuous
data by the capacitary boundary estimate and is bounded by
Corollary~\ref{cor:maximum}. The uniqueness just proved identifies
the two solutions.
\end{proof}

\begin{rem}
To identify the continuous and variational solutions for
\(g\in C_b(\R^n)\cap\Vs(D)\), we use the boundary
continuity estimate in \cite[Theorem~1.4]{KLW}, which
accommodates a general modulus of continuity for \(g\).
No H\"older continuity of \(g\) is required.
Theorem~\ref{thm:continuous} concerns bounded continuous
exterior data and does not cover unbounded data in the
tail space. The next subsection instead constructs
finite-energy approximations of a given local weak
solution in the natural tail class, as needed for
the boundary comparison argument.
\end{rem}

\subsection{Approximation in the natural tail class}

\begin{lem}\label{lem:natural-tail-approximation}
Let \(p\ge2\), \(0<s<1\), \(R>0\), and let \(D\) be a nonempty
open subset of \(B_R(z)\). Suppose that
\[
 u\in\Ws(B_{4R/3}(z))\cap \mathrm L^{p-1}_{sp}(\R^n),\qquad
 u\ge0\quad\hbox{on }\R^n,\qquad
 \esssup_{B_R(z)}u<\infty,
\]
and that \(\Fsp u=0\) weakly in \(D\).
Then there are \(g_m\in\Vs(D)\) and variational solutions
\(u_m\in g_m+\Wzero(D)\) such that
\begin{equation}\label{eq:tail-approximation-order}
 0\le u_m\le u_{m+1}\le u\quad\hbox{on }\R^n,
 \qquad g_m=u\quad\hbox{near }\overline D.
\end{equation}
Moreover,
\begin{align}
 (u-u_m)\mathbf1_D&\longrightarrow0\qquad
       \hbox{in }\Ws(\R^n),\label{eq:tail-approximation-energy}\\
 u_m&\longrightarrow u\qquad\hbox{locally uniformly in }D,\label{eq:tail-approximation-local}\\
 \int_{\R^n}\frac{|u-u_m|^{p-1}}{(1+|y|)^{n+sp}}\d y
       &\longrightarrow0.\label{eq:tail-approximation-tail}
\end{align}
If \(D=D_R(z)\) and \(u\) has the local homogeneous exterior
condition on \(B_R(z)\), then every \(u_m\) has that condition.
\end{lem}

\begin{proof}
Choose \(0\le\chi\le1\), with
\(\chi\in C_0^\infty(B_{5R/4}(z))\) and
\(\chi=1\) on \(B_{7R/6}(z)\), and define
\[
 g_m=\chi u+(1-\chi)\min\{u,m\}.
\]
In the error estimate below take \(m\ge\esssup_{B_R(z)}u\),
which discards at most finitely many terms.
The truncation and smooth multiplication properties of \(\Ws\)
give \(g_m\in\Ws(B_{4R/3}(z))\).
Outside that ball \(0\le g_m\le m\), and its distance from
\(D\) is at least \(R/3\). These facts prove
\(g_m\in\Vs(D)\). Theorem~\ref{thm:dirichlet} supplies \(u_m\).

The functions \(g_m,u\) agree on \(B_{7R/6}(z)\).
Subtracting their weak forms shows that
\(\Fsp g_m=e_m\) in \(D\), where
\[
 e_m(x)=\int_{\R^n\setminus B_{7R/6}(z)}
 \frac{J_p(u(x)-g_m(y))-J_p(u(x)-u(y))}
 {|x-y|^{n+sp}}\d y.
\]
This subtraction is legitimate in the natural tail class because
the modified pairs are separated from \(D\).
The integrand is nonnegative and vanishes when \(u(y)\le m\).
Otherwise \(u(x)\le m\le g_m(y)\le u(y)\), and its numerator
is at most \(u(y)^{p-1}\). Since
\(|x-y|\ge |y-z|/7\) on the displayed region,
\begin{equation}\label{eq:tail-approximation-error}
 0\le e_m(x)\le C\epsilon_m,\qquad
 \epsilon_m=
 \int_{\{u>m\}\setminus B_{7R/6}(z)}
       \frac{u(y)^{p-1}}{|y-z|^{n+sp}}\d y
 \longrightarrow0.
\end{equation}
The convergence follows by dominated convergence.
As \(g_m\in\Vs(D)\) and \(e_m\in \mathrm L^\infty(D)\), the weak
equation extends by density to every test in \(\Wzero(D)\).

Comparison gives \(0\le u_m\le g_m\le u\).
Monotonicity of \(g_m\) and comparison also give
\(u_m\le u_{m+1}\). Put
\[
 z_m=g_m-u_m=(u-u_m)\mathbf1_D\in\Wzero(D).
\]
Subtract the equations and test with \(z_m\).
The inequality
\((J_p(a)-J_p(b))(a-b)\ge c_p|a-b|^p\) yields
\begin{equation}\label{eq:tail-approximation-stability}
 c_p[z_m]_{s,p}^p
 \le\int_D e_m z_m\d x
 \le C\epsilon_m\|z_m\|_{\mathrm L^1(D)}
 \le C_R\epsilon_m[z_m]_{s,p}.
\end{equation}
The last step uses H\"older and fractional Poincar\'e.
This proves \eqref{eq:tail-approximation-energy}, including
convergence in \(\mathrm L^p(\R^n)\).

The monotone limit equals \(u\) almost everywhere in \(D\)
by that convergence, and outside \(D\) because \(g_m\uparrow u\).
The bound \(0\le u_m\le u\) and dominated convergence prove
\eqref{eq:tail-approximation-tail}.
On every compact subset of \(D\), local boundedness and the
tails of \(u_m\) are uniformly controlled by those of \(u\).
Proposition~\ref{prop:interior} therefore gives equicontinuity.
The almost everywhere limit is the continuous representative of
\(u\), which proves \eqref{eq:tail-approximation-local} for
the full sequence.

Finally, suppose that \(D=D_R(z)\). For
\(\eta\in C_0^\infty(B_R(z))\),
\[
 \eta u_m=\eta u+\eta(u_m-g_m)\in\Wzero(D_R(z)).
\]
The first term belongs to this space by the assumed weak zero
condition, and the second by smooth multiplication in \(\Wzero(D)\).
Also \(u_m=0\) almost everywhere on \(B_R(z)\setminus\Om\),
and \(u_m\in\Ws(B_R(z))\). This verifies
Definition~\ref{def:local-zero}.
\end{proof}

\begin{rem} Note that in Lemma \ref{lem:natural-tail-approximation} no uniform bound on the full \(p\)-energy of the data \(g_m\)
is asserted or needed. The stability estimate uses only the vanishing
\((p-1)\)-tail in \eqref{eq:tail-approximation-error}.
\end{rem}

\section{Basic estimates and control of localization}\label{sec:basic}

This section collects the energy, interior, and boundary estimates
used in the subsequent comparison arguments. We begin with a
Caccioppoli inequality and keep explicit track of the non-local
tails that arise when restricting the equation to a smaller region.
Such localization requires control of both the local energy and
the interactions with the exterior, since distant values continue
to influence the equation.

The estimates provide the quantitative bounds needed for the
half-space analysis and the approximation arguments in Reifenberg
flat domains. Unless otherwise stated, they hold for
\(0<s<1\) and \(1<p<\infty\), without the additional restriction
required by the interior \(C^{1,\alpha}\) regularity theorem.

\subsection{Energy and interior estimates}

\begin{lem}\label{lem:caccioppoli}
Let \(1<p<\infty\), \(B_{2r}=B_{2r}(z)\), and
\(u\in\Ws(B_{2r})\cap \mathrm L^{p-1}_{sp}(\R^n)\).
Suppose either that \(u\) is a weak solution in \(B_{2r}\), or that it is a local weak solution in \(\Om\cap B_{2r}\)
with zero exterior values on \(B_{2r}\setminus\Om\), in the
sense of Definition~\ref{def:local-zero}.
Let \(0\le\eta\le1\), \(\eta\in C_0^\infty(B_{3r/2})\),
\(\eta=1\) on \(B_r\), and \(|\nabla\eta|\le C/r\).
Then
\begin{equation}\label{eq:caccioppoli}
 [\eta u]_{s,p}^p
 \le \frac C{r^{sp}}\int_{B_{2r}}|u|^p\d x
 +\frac C{r^{sp}}(\Tail(u,z,2r))^{p-1}
             \int_{B_{2r}}|u|\d x.
\end{equation}
The constant depends only on \(n,s,p\) and the cutoff bound.
\end{lem}

\begin{proof}
Use \(\eta^p u\) as a test, with truncation if needed.
In the boundary case this test is admissible by
Definition~\ref{def:local-zero}.
For \(a,b\in\R\) and \(0\le\xi,\zeta\le1\), the elementary
cutoff inequality gives
\[
 \Jp(a-b)(a\xi^p-b\zeta^p)
 \ge c_p|a\xi-b\zeta|^p
       -C_p(|a|^p+|b|^p)|\xi-\zeta|^p.
\]
Integrating on \(B_{2r}\times B_{2r}\) bounds the error by
\(Cr^{-sp}\int_{B_{2r}}|u|^p\d x\).
For \(x\in\supp\eta\) and \(y\notin B_{2r}\),
\(|x-y|\ge c|y-z|\), and
\[
 |\Jp(u(x)-u(y))u(x)|
 \le C_p\bigl(|u(x)|^p+|u(x)|\,|u(y)|^{p-1}\bigr).
\]
The first term is absorbed into the same local error.
The second gives the tail term in \eqref{eq:caccioppoli}.
The remaining part of the energy of \(\eta u\), with one point
outside \(B_{2r}\), is also bounded by the local error.
Approximation and truncation remove the initial test restrictions.
\end{proof}

We use the local boundedness and Harnack estimates in \cite{DKP}
and the H\"older estimate in \cite[Theorem~1.2]{DKP16}.
Although formulated there in global energy spaces, their proofs
use compactly supported cutoff and truncation test functions.
For these test functions, local \(p\)-energy controls nearby interactions,
while the \((p-1)\)-tail controls the separated exterior interactions.
The same estimates therefore hold in the local energy and tail
class used here, with the tail convention of \cite{DKP}.
For \(0<\beta<1\) and \(E\subset\R^n\), we write
\[
 [u]_{C^\beta(E)}
 =
 \sup_{\substack{x,y\in E\\x\ne y}}
 \frac{|u(x)-u(y)|}{|x-y|^\beta}.
\]

\begin{prop}\label{prop:interior}
Let \(1<p<\infty\), and let \(u\) solve \(\Fsp u=0\) in \(B_{2r}(z)\).
Then there are \(C\ge1\) and \(\beta\in(0,1)\), depending only on \(n,s,p\),
such that
\[
 \|u\|_{\mathrm L^\infty(B_r(z))}
 \le C\biggl(
 \biggl(\avgint_{B_{2r}(z)}|u|^p\d x\biggr)^{1/p}
 +\Tail(u,z,2r)\biggr).
\]
If \(u\ge0\) in \(\R^n\), then
\begin{equation}\label{eq:harnack}
 \sup_{B_r(z)}u\le C\inf_{B_r(z)}u.
\end{equation}
Moreover,
\[
 r^\beta[u]_{C^\beta(B_r(z))}
 \le C\bigl(\|u\|_{\mathrm L^\infty(B_{2r}(z))}
                 +\Tail(u,z,2r)\bigr).
\]
\end{prop}

The inequalities in Proposition~\ref{prop:interior} are understood first for the continuous
representative supplied by interior regularity.
In particular, a nontrivial globally nonnegative solution is
strictly positive in every connected region where it solves the
equation. If it vanished identically in an interior ball, testing
there would also exclude nonzero nonnegative exterior values.

\begin{prop}\label{prop:C1}
Assume \eqref{eq:range}. Let \(u\) solve \(\Fsp u=0\) in
\(B_{2r}(z)\).
Then there are \(\sigma\in(0,1)\) and \(C\ge1\), depending only on
\(n,s,p\), such that
\begin{equation}\label{eq:C1}
 r\|\nabla u\|_{\mathrm L^\infty(B_r(z))}
 +r^{1+\sigma}[\nabla u]_{C^\sigma(B_r(z))}
 \le C\bigl(\|u\|_{\mathrm L^\infty(B_{2r}(z))}
                   +\Tail(u,z,2r)\bigr).
\end{equation}
\end{prop}

\begin{proof}
Apply \cite[Theorem~1.1]{GJS} to \(U(\xi)=u(z+r\xi)\).
After rescaling to the unit ball, the two tail conventions
at radius \(2\) differ by the fixed factor
\(2^{sp/(p-1)}\), which is absorbed into the constant.
Scaling back gives \eqref{eq:C1}.
\end{proof}

\begin{rem} In contrast to the local estimate in \cite[Lemma~2.5]{LN} valid for the (local) $p$-Laplace equation,
the right side of \eqref{eq:C1} includes an exterior contribution.
Any estimate of \(\nabla u\) purely in terms of \(u(x)/d_\Om(x)\)
must account for that contribution.
\end{rem}

\subsection{Boundary decay}

For \(w\in\partial\Om\), put
\[
 \mathcal A_u(w,R)=R^{-n/p}\|u\|_{\mathrm L^p(D_R(w))}
                         +\Tail(u,w,R).
\]

\begin{prop}\label{prop:boundary-decay}
Let \(1<p<\infty\).
Suppose that \(\Theta_w(t)\ge\theta_0>0\) for \(0<t\le R\).
Let \(u\in\Vs(D_R(w))\) solve \(\Fsp u=0\) in \(D_R(w)\),
with zero exterior values on \(B_R(w)\setminus\Om\) in the
sense of Definition~\ref{def:local-zero}. Then there are \(C\ge1\) and \(\beta>0\), depending only on
\(n,s,p,\theta_0\), such that
\begin{equation}\label{eq:boundary-decay}
 \esssup_{D_\rho(w)}|u|
 \le C\left(\frac{\rho}{R}\right)^\beta
             \mathcal A_u(w,R),
 \qquad 0<\rho<R/4.
\end{equation}
\end{prop}

\begin{proof}
We first verify the weak exterior condition used in \cite{KLW}.
Fix \(R'<R\), and put \(T=B_{R'}(w)\setminus\Om\).
In \cite[Definition~3.1]{KLW}, zero values on \(T\) mean membership
in the closure of \(\operatorname{Lip}_c(\R^n\setminus T)\)
in \(\Vs(D_{R'}(w))\). Choose
\(\eta\in C_0^\infty(B_R(w))\) equal to one on a neighborhood
of \(\overline B_{R'}(w)\). By Definition~\ref{def:local-zero},
\(\eta u\) is a \(\Ws(\R^n)\)-limit of smooth functions compactly
supported in \(D_R(w)\). These functions avoid \(T\), and their
convergence also holds in \(\Vs(D_{R'}(w))\).
The remainder \((1-\eta)u\) vanishes on a neighborhood of
\(\overline B_{R'}(w)\). Moreover, \(u\in\Vs(D_R(w))\) implies
\[
 \int_{\R^n}\frac{|u(y)|^p}{(1+|y|)^{n+sp}}\,\d y<\infty,
\]
by fixing an interior point at which the defining energy integral
is finite. Truncation and smooth approximation of the remainder
away from \(\overline B_{R'}(w)\) therefore converge in
\(\Vs(D_{R'}(w))\): the separation from this ball bounds the
interaction kernel by a constant times the displayed weight.
Thus \(u\) has zero values on \(T\) in the sense of \cite{KLW}.

Apply the zero-data case of \cite[Corollary~1.5]{KLW} at radius
\(R'\), and let \(R'\uparrow R\). The local norm and tail in
\(\mathcal A_u(w,R')\) converge to those at radius \(R\), giving
\[
 \esssup_{D_\rho(w)}|u|
 \le C\mathcal A_u(w,R)
 \exp\left(-c\int_{4\rho}^R\Theta_w(t)\frac{\d t}{t}\right).
\]
The lower bound for \(\Theta_w\) gives
\((4\rho/R)^{c\theta_0}\). Decreasing the exponent if necessary
proves \eqref{eq:boundary-decay}.
\end{proof}

\begin{rem} Lemma~\ref{lem:capacity} makes the constants in Proposition~\ref{prop:boundary-decay} uniform in bounded
NTA domains and in sufficiently flat Reifenberg domains.
The assumption \(u\in\Vs(D_R(w))\) supplies the finite energy
required by the quoted capacitary estimate and the proof above
verifies compatibility of its weak exterior condition with ours. Lemma~\ref{lem:natural-tail-approximation}
provides the approximation used below to recover boundary comparison
in the natural tail class.
\end{rem}

\subsection{Scaling and exterior contributions}

For \(a,r>0\), \(z\in\R^n\), and
\(U(\xi)=a^{-1}u(z+r\xi)\), changes of variables in the
weak formulation and in \eqref{eq:tail} give
\[
 \Fsp U(\xi)=a^{1-p}r^{sp}\Fsp u(z+r\xi),\qquad
 \Tail(U,\xi,\rho)=a^{-1}\Tail(u,z+r\xi,r\rho).
\]
In particular, normalization by \(a=r^s\) multiplies the
right side of the equation by \(r^s\).

\begin{lem}\label{lem:cutoff}
Let \(p\ge2\), let \(u\in \mathrm L^{p-1}_{sp}(\R^n)\) be bounded in
\(B_R(z)\), and suppose that \(u\) solves the homogeneous
equation in an open subset of \(B_R(z)\).
Take \(0\le\chi\le1\), with \(\chi=1\) on \(B_R(z)\) and
\(\chi=0\) outside \(B_{2R}(z)\), and fix \(0<\theta<1\).
Then \(v=\chi u\) satisfies
\(\Fsp v=f_\chi\) in that subset intersected with \(B_{\theta R}(z)\),
where
\begin{align}\label{eq:cutoff-error}
 f_\chi(x)
 &=\int_{\R^n\setminus B_R(z)}
 \frac{\Jp(u(x)-\chi(y)u(y))-\Jp(u(x)-u(y))}
 {|x-y|^{n+sp}}\d y,\\
 |f_\chi(x)|
 &\le C R^{-sp}
       \bigl(|u(x)|^{p-1}+(\Tail(u,z,R))^{p-1}\bigr).
 \notag
\end{align}
Here \(C\) depends only on \(n,s,p,\theta\).
The assertions are distributional if the individual principal
values have not been justified.
\end{lem}

\begin{proof}
The functions \(u,v\) agree in \(B_R(z)\), so their local
contributions cancel upon subtraction.
For \(x\in B_{\theta R}(z)\) and \(y\notin B_R(z)\),
\(|x-y|\ge(1-\theta)|y-z|\).
The inequality
\[
 |\Jp(a-b)|+|\Jp(a-c)|
 \le C_p\bigl(|a|^{p-1}+|b|^{p-1}+|c|^{p-1}\bigr)
\]
and \(|\chi u|\le|u|\) give the stated bound.
Absolute convergence of the exterior integral also justifies
the subtraction in the weak formulation.
\end{proof}

\begin{rem} By Lemma \ref{lem:cutoff}, a cutoff therefore produces a right side even when the original
equation is homogeneous. In particular, boundary results for a
global zero-exterior problem must be localized with this error
accounted for.
\end{rem}

\subsection{The Carleson estimate and control of the tail}

For the local \(p\)-Laplacian in NTA domains, nonnegative
solutions vanishing on a boundary portion satisfy the Carleson
estimate \cite[Lemma~2.4]{LN}, which in our notation takes
the form
\begin{equation}\label{eq:carleson-target}
 \sup_{D_{r/c}(w)}u\le C u(A_{r/c}(w)).
\end{equation}
This estimate follows from boundary H\"older continuity
and the interior Harnack inequality through a Harnack
chain argument.

Proposition~\ref{prop:boundary-normalization} below establishes
the corresponding estimate for the fractional \(p\)-Laplacian
when \(p\ge2\) in sufficiently flat Reifenberg domains.
It also controls the positive tails at interior scales
by the same reference value. The proof combines an interior
positive-tail estimate with boundary local boundedness
and boundary H\"older estimates.

We first establish the tail estimate in a form that requires
no assumptions on the boundary geometry. It applies in any
interior ball and will subsequently be used at corkscrew
points.

\begin{lem}
\label{lem:positive-tail}
Let \(0<s<1\), \(p\ge2\), \(a\in\R^n\), and \(\rho>0\).
Suppose that \(u\ge0\) almost everywhere in \(\R^n\), $u\in\Ws(B_{2\rho}(a))
       \cap\mathrm L^{p-1}_{sp}(\R^n)$, and that \(u\) is weakly fractional \(p\)-harmonic in
\(B_{2\rho}(a)\). Then
\[
 \Tail(u,a,\rho)\le C u(a),
\]
where \(u(a)\) is evaluated using the continuous interior
representative and \(C=C(n,s,p)\).
Moreover,
\[
 \rho^{sp}\int_{\R^n}
 \frac{u(y)^{p-1}}{(\rho+|y-a|)^{n+sp}}
 \,\mathrm d y
 \le C u(a)^{p-1}.
\]
\end{lem}

\begin{proof}
We give the argument of \cite[Lemma~4.2]{DKP} in the local
energy class used here. Write \(B=B_\rho(a)\) and
\(k=\sup_B u\), which is finite by interior local boundedness.
Choose \(\eta\in C_0^\infty(B_{3\rho/4}(a))\) such that
\(0\le\eta\le1\), \(\eta=1\) on \(B_{\rho/2}(a)\), and
\(|\nabla\eta|\le C/\rho\).

The function \((u-2k)\eta^p\) is an admissible test function
by local Sobolev approximation. Splitting the weak equation
into interactions within \(B\) and interactions with \(B^c\)
gives \(I+J=0\), where
\[
 I=\frac12\int_B\int_B
 \frac{\Jp(u(x)-u(y))
 \bigl((u(x)-2k)\eta(x)^p-(u(y)-2k)\eta(y)^p\bigr)}
 {|x-y|^{n+sp}}
 \,\mathrm d y\,\mathrm d x
\]
and
\[
 J=\int_B\int_{B^c}
 \frac{\Jp(u(x)-u(y))(u(x)-2k)\eta(x)^p}
 {|x-y|^{n+sp}}
 \,\mathrm d y\,\mathrm d x.
\]
These integrals are well defined because \(u\) has local
\(p\)-energy and a finite exterior \((p-1)\)-tail. Since \(0\le u\le k\) on \(B\), Young's inequality gives
\[
 I\ge
 -Ck^p\int_B\int_B
 \frac{|\eta(x)-\eta(y)|^p}{|x-y|^{n+sp}}
 \,\mathrm d y\,\mathrm d x
 \ge-Ck^p\rho^{n-sp}.
\]
For \(0\le b\le k\) and \(t\ge0\), we also have
\[
 \Jp(b-t)(b-2k)\ge ck t^{p-1}-Ck^p.
\]
Applying this inequality to \(J\), the separation of
\(\supp\eta\) from \(B^c\) bounds the negative contribution.
For the positive contribution, we use \(\eta=1\) on
\(B_{\rho/2}(a)\) and
\[
 |x-y|\le\frac32|y-a|
 \qquad
 \text{for }x\in B_{\rho/2}(a),\quad y\in B^c.
\]
Consequently,
\[
 J\ge
 ck\rho^{n-sp}(\Tail(u,a,\rho))^{p-1}
 -Ck^p\rho^{n-sp}.
\]
Hence,
\[
0=I+J\geq -Ck^p\rho^{n-sp}+ ck\rho^{n-sp}(\Tail(u,a,\rho))^{p-1}
 -Ck^p\rho^{n-sp},\]
and if \(k>0\), then
\[
 \Tail(u,a,\rho)\le Ck.
\]
The interior Harnack inequality gives \(k\le Cu(a)\). If \(k=0\), then \(u=0\) in \(B\). Testing the equation
with a nonzero nonnegative function supported in \(B\),
and using global nonnegativity, shows that \(u=0\)
almost everywhere in \(\R^n\). The tail estimate therefore
also holds in this case.

Finally, split the weighted integral into \(B\) and \(B^c\).
On \(B\), interior Harnack gives \(u\le Cu(a)\), and hence
\[
 \rho^{sp}\int_B
 \frac{u(y)^{p-1}}{(\rho+|y-a|)^{n+sp}}
 \,\mathrm d y
 \le C u(a)^{p-1}.
\]
On \(B^c\), the same integral is bounded by
\((\Tail(u,a,\rho))^{p-1}\). This proves the second assertion.
\end{proof}

In particular, if \(u\) satisfies the hypotheses of
Lemma~\ref{lem:positive-tail} with \(a=e_n\) and
\(\rho=1/4\), then
\begin{equation}\label{eq:normalized-mass}
 \int_{\R^n}
 \frac{u(y)^{p-1}}{(1+|y|)^{n+sp}}
 \,\mathrm d y
 \le C u(e_n)^{p-1}.
\end{equation}
Indeed,
\(\frac14+|y-e_n|\le\frac54(1+|y|)\), so this follows
from the weighted estimate in the lemma.
No half-space assumption is needed for this conclusion.

We now apply the interior tail estimate at a corkscrew point.

\begin{prop}\label{prop:boundary-normalization}
Let \(0<s<1\), \(p\ge2\), and \(0<\beta<s\).
Then there exists \(\delta_0=\delta_0(n,s,p,\beta)>0\) such
that the following holds. Let \(\Om\) be \((\delta,r_0)\)-Reifenberg flat with
\(0<\delta\le\delta_0\). Suppose that
\(z\in\partial\Om\), \(0<8r<r_0\), and that \(U\) is a
nontrivial weak solution of \(\Fsp U=0\) in \(D_{8r}(z)\), nonnegative
on \(\R^n\), with finite tail and zero exterior values
on \(B_{8r}(z)\setminus\Om\) in the sense of
Definition~\ref{def:local-zero}.
Set \(a_r=U(A_r(z))\). Then
\begin{align}
 r^{-n}\int_{B_{4r}(z)}U^{p-1}\,\mathrm d y
 +r^{sp}\int_{\R^n\setminus B_{4r}(z)}
 \frac{U(y)^{p-1}}{|y-z|^{n+sp}}\,\mathrm d y
 &\le C a_r^{p-1},
 \label{eq:boundary-normalized-mass}\\
 \|U\|_{\mathrm L^\infty(B_{2r}(z))}
 +r^\beta[U]_{C^\beta(B_{2r}(z))}
 &\le C a_r.
 \label{eq:boundary-normalized-holder}
\end{align}
In particular,
\begin{equation}\label{eq:boundary-normalized-decay}
 U(x)\le C a_r\left(\frac{d_\Om(x)}r\right)^\beta
 \quad\text{for }x\in D_r(z),
 \qquad
 \int_{D_r(z)}U\,\mathrm d x\asymp a_r r^n.
\end{equation}
The constants in these estimates depend only on
\(n,s,p,\beta\) and the fixed corkscrew constants.
\end{prop}

\begin{proof}
Write \(A=A_r(z)\). By the corkscrew condition, $|A-z|<r$, $d_\Om(A)\ge\kappa r$, for a fixed \(\kappa\in(0,1]\).
Taking \(\rho=\kappa r/4\), we have $B_{2\rho}(A)\subset D_{2r}(z)$. Lemma~\ref{lem:positive-tail} therefore gives
\[
 \int_{\R^n\setminus B_\rho(A)}
 \frac{U(y)^{p-1}}{|y-A|^{n+sp}}\,\mathrm d y
 \le C\rho^{-sp}a_r^{p-1}
 \le Cr^{-sp}a_r^{p-1}.
\]
Interior Harnack also gives \(U\le Ca_r\) on \(B_\rho(A)\). For \(y\in B_{4r}(z)\), we have \(|y-A|\le5r\).
Thus the preceding tail estimate controls the integral
of \(U^{p-1}\) over \(B_{4r}(z)\setminus B_\rho(A)\).
The integral over \(B_\rho(A)\) is controlled by Harnack.
For \(y\notin B_{4r}(z)\),
\[
 \frac34|y-z|\le |y-A|\le\frac54|y-z|,
\]
so the same tail estimate controls the exterior integral
centered at \(z\). This proves
\eqref{eq:boundary-normalized-mass}. Reifenberg flatness gives
\(|D_{4r}(z)|\asymp r^n\). Consequently,
\[
 \left(
 \frac{1}{|D_{4r}(z)|}
 \int_{D_{4r}(z)}U^{p-1}\,\mathrm d y
 \right)^{1/(p-1)}
 +\Tail(U,z,4r)
 \le Ca_r.
\]
This is the normalization appearing in
\cite[Lemma~2.6 and Theorem~1.8]{BKS}.
Applying those estimates after translation and rescaling
gives \eqref{eq:boundary-normalized-holder}, provided
\(\delta_0\) is chosen sufficiently small depending on
\(n,s,p,\beta\).
The localized truncation test functions used in these estimates are
admissible by Definition~\ref{def:local-zero}.
Indeed, \(\eta U\in\Wzero(D_{8r}(z))\) for every
\(\eta\in C_0^\infty(B_{8r}(z))\).
Further cutoffs, Lipschitz truncations vanishing at zero,
and Sobolev approximation preserve this test function class.
Local energy and tail integrability justify passage to the limit
in the weak pairing. The exterior values remain in the tail term
throughout.

The continuous representative of \(U\) vanishes on
\(\partial\Om\cap B_{2r}(z)\).
For \(x\in D_r(z)\), choose a nearest boundary point
\(\zeta\in\partial\Om\). Since
\(d_\Om(x)\le|x-z|<r\), we have \(\zeta\in B_{2r}(z)\).
Hence
\[
 U(x)=|U(x)-U(\zeta)|
 \le [U]_{C^\beta(B_{2r}(z))}d_\Om(x)^\beta,
\]
which proves the decay estimate. The upper bound on the integral follows from
\eqref{eq:boundary-normalized-holder}.
For the lower bound, set $E=B_\rho(A)\cap B_r(z)$. Then \(E\subset D_r(z)\) and
\(|E|\ge c\rho^n\ge cr^n\).
Interior Harnack gives \(U\ge ca_r\) on \(B_\rho(A)\), so
\[
 \int_{D_r(z)}U\,\mathrm d x
 \ge\int_E U\,\mathrm d x
 \ge ca_r r^n.
\]
This completes the proof.
\end{proof}

\begin{rem} The boundary measure in \cite[Lemma~2.6]{LN}, $p$-harmonic measure,
has a different non-local counterpart.
If \(u\ge0\) belongs to
\(\mathrm L^{p-1}_{sp}(\R^n)\) and vanishes on an open
exterior set, then at points \(x\) of that set with positive
distance from the essential support of \(u\),
\[
 -\Fsp u(x)=
 \int_{\R^n}
 \frac{u(y)^{p-1}}{|x-y|^{n+sp}}\,\mathrm d y.
\]
Thus the exterior interaction extends into the complement
of \(\Om\). In contrast to \cite{LN}, in this paper no boundary-supported measure identity is used
in the subsequent arguments.
\end{rem}

\section{The half-space and boundary gradient estimates}
\label{sec:flat}

This section develops the flat boundary estimates used later
to approximate solutions in Reifenberg domains. The explicit
profile \((x_n)_+^s\) provides the reference behavior near
the boundary of the half-space. Using the normalization
established in Section~\ref{sec:basic}, we obtain two-sided
bounds and H\"older estimates for \(u/x_n^s\) near the flat
boundary. This is analogous to the corresponding local theory in \cite{LN} where the initial step consists of
obtaining two-sided
bounds and H\"older estimates for \(u/x_n\). The argument combines boundary barriers with
localized weighted estimates, keeping track of the exterior
tails so that distant nonnegative data are controlled by
a single interior reference value.

The half-space boundary comparison holds for every \(p\ge2\).
We then combine the quotient estimates with interior
\(C^{1,\alpha}\) regularity to obtain boundary gradient
asymptotics and quantitative non-degeneracy.
This step uses \cite{GJS} and therefore requires
\(2\le p<2/(1-s)\).
These gradient estimates will subsequently be transferred
to sufficiently flat Reifenberg domains.

\subsection{The half-space profile}

Put \(\HH=\{x\in\R^n\mid x_n>0\}\).

\begin{prop}\label{prop:profile}
Let \(0<s<1\) and \(1<p<\infty\).
The function \(h(x)=(x_n)_+^s\) belongs to
\(\Ws_{\mathrm{loc}}(\R^n)\cap \mathrm L^{p-1}_{sp}(\R^n)\) and satisfies
\(\Fsp h=0\) in \(\HH\). Moreover,
\begin{equation}\label{eq:profile-gradient}
 \nabla h(x)=s x_n^{s-1}e_n,\qquad
 |\nabla h(x)|=s\frac{h(x)}{x_n}.
\end{equation}
\end{prop}

\begin{proof}
Integrating first in the tangential variables reduces the equation,
up to a positive dimensional constant, to the one-dimensional
equation for \(t_+^s\).
By scaling it is enough to evaluate the operator at \(t=1\).
The contribution from the negative half-line is
\[
 \int_{-\infty}^0\frac1{|1-t|^{1+sp}}\d t=\frac1{sp}.
\]
In the integral over \((1,\infty)\), make the change of variables
\(t=1/\tau\). Combining it with the integral over \((0,1)\) gives
\[
 \operatorname{PV}\int_0^\infty
 \frac{|1-t^s|^{p-2}(1-t^s)}{|1-t|^{1+sp}}\d t
 =
 \int_0^1
 \frac{(1-t^s)^{p-1}(1-t^{s-1})}{(1-t)^{1+sp}}\d t.
\]
The difference between the transformed truncations tends to zero.
Indeed, near \(1\) their widths differ by \(O(\varepsilon^2)\),
whereas the integrand before cancellation is
\(O(\varepsilon^{p-2-sp})\).
The resulting error is \(O(\varepsilon^{p(1-s)})\). For
\(\Phi(t)=(1-t^s)^p/(1-t)^{sp}\), direct differentiation gives
\[
 \Phi'(t)=sp
 \frac{(1-t^s)^{p-1}(1-t^{s-1})}{(1-t)^{1+sp}}.
\]
Since \(\Phi(0)=1\) and \(\Phi(1)=0\), the positive half-line
contribution is \(-1/(sp)\).
This proves the pointwise equation. Tangential integration can
equivalently be carried out in the symmetrized weak formulation,
which gives the same conclusion. For completeness, the local one-dimensional energy across \(0\)
is finite because
\[
 \int_0^1 t^{sp}\int_{-1}^0
          (t-\tau)^{-1-sp}\d\tau\,\mathrm d t<\infty.
\]
On the positive side, set \(\tau=tz\), \(0<z<1\).
The resulting integral in \(z\) has integrand
\((1-z^s)^p/(1-z)^{1+sp}\), which is integrable.
Tangential integration proves the local energy assertion in
\(\R^n\).
At infinity, the tail is bounded by a constant times
\(\int_1^\infty\rho^{-s-1}\d\rho\).
Finally, differentiation proves \eqref{eq:profile-gradient}.
\end{proof}

The half-space profile \(h(x)=(x_n)_+^s\) does not belong to
\(\Vs(D)\) for any nonempty bounded open set \(D\).
Indeed, its exterior \(p\)-energy diverges logarithmically on
any fixed cone where \(y_n\asymp|y|\) as \(|y|\to\infty\).
It nevertheless belongs to the local energy and tail class used in
Definition~\ref{def:weak}.
This distinction is another reason to keep the energy formulation
and the more general tail formulation separate.

\subsection{A normalized theorem in a flat boundary patch}

The proof follows the boundary comparison strategy of \cite{LN,LLN},
with \(x_n^s\) replacing \(x_n\). The reflection step in the local
argument has to be replaced. Indeed, if
\(\widetilde h(x)=\operatorname{sgn}(x_n)|x_n|^s\), then
Proposition~\ref{prop:profile} and subtraction of the exterior
integrals give
\[
 \Fsp\widetilde h(x)
 =\int_{\{y_n<0\}}
 \frac{(x_n^s+|y_n|^s)^{p-1}-x_n^{s(p-1)}}
 {|x-y|^{n+sp}}\d y>0,\qquad x\in\HH.
\]
Thus odd reflection does not preserve the homogeneous equation.
We instead use flat barriers and the weighted estimates of
Iannizzotto, Mosconi, and Squassina \cite{IMS20}. Throughout this section, we write \(B_r=B_r(0)\).

\begin{thm}\label{thm:flat}
Let \(0<s<1\) and \(p\ge2\).
Suppose that
\(u\in\Ws(B_4)\cap \mathrm L^{p-1}_{sp}(\R^n)\) is nontrivial,
nonnegative in \(\R^n\), and satisfies
\[
 \Fsp u=0\quad\text{in }\HH\cap B_4,\qquad
 u=0\quad\text{on }B_4\setminus\HH
\]
in the sense of Definition~\ref{def:local-zero}.
Then there exist \(C_F\ge1\) and \(\beta_F\in(0,1)\), depending only
on \(n,s,p\), such that \(q_u=u/x_n^s\) extends continuously to
\(\overline{\HH\cap B_1}\) and
\begin{equation}\label{eq:flat-assumption}
 C_F^{-1}u(e_n)\le q_u\le C_Fu(e_n),\qquad
 [q_u]_{C^{\beta_F}(\overline{\HH\cap B_1})}
 \le C_Fu(e_n).
\end{equation}
Moreover,
\begin{equation}\label{eq:flat-normalization}
 \|u\|_{\mathrm L^\infty(B_2)}+\Tail(u,0,2)\le C_Fu(e_n).
\end{equation}
\end{thm}

The proof of Theorem \ref{thm:flat} uses the two-sided profile bounds in
Lemma~\ref{lem:flat-two-sided}, and the localized oscillation
estimate in Lemma~\ref{lem:localized-oscillation}, both stated and proved below.
All comparisons below retain the values on the entire complement
of the comparison domain.

\begin{lem}\label{lem:flat-two-sided}
Under the hypotheses of Theorem~\ref{thm:flat}, we have
\[
 u(x)\le C u(e_n)x_n^s\quad\text{in }\HH\cap B_2,
 \qquad
 u(x)\ge c u(e_n)x_n^s\quad\text{in }\HH\cap B_1,
\]
and \eqref{eq:flat-normalization} holds with constants
depending only on \(n,s,p\).
\end{lem}

\begin{proof}
Normalize \(u(e_n)=1\).
The flat estimate \cite[Lemma~3.1]{BKS}, applied with radius \(4\)
and zero right side, gives
\[
 \left\|\frac{u}{x_n^s}\right\|_{\mathrm L^\infty(\HH\cap B_2)}
 \le C\biggl[
 \biggl(\int_{B_4}u^{p-1}\d x\biggr)^{1/(p-1)}
 +\Tail(u,0,4)\biggr].
\]
The distance to \(\partial(\HH\cap B_4)\) equals \(x_n\) on
\(\HH\cap B_2\), so the distance in that lemma is exactly
the one used here. Both terms on the right are bounded by
Lemma~\ref{lem:positive-tail}.
This proves the upper estimate.
The same lemma controls \(\Tail(u,0,2)\), and the upper estimate
controls \(\|u\|_{\mathrm L^\infty(B_2)}\).

We now give a comparison proof for the lower estimate.
Fix \(z=(z',0)\), \(|z'|\le1\), and put
\[
 r_b=1/4,\qquad b=z+r_b e_n,\qquad
 B=B_{r_b}(b),\qquad E=B_{r_b/2}(b).
\]
Interior Harnack chains in \(\HH\cap B_4\) imply
\(u\ge m>0\) on \(E\), where \(m=m(n,s,p)\). Let \(\tau\in\Wzero(B)\) be the torsion function of \(B\), i.e., the unique weak solution of
\begin{align}\label{torsion}
  &\Fsp\tau=1\quad \text{in }B,\qquad
  \tau=0\quad  \text{on }B^c.
\end{align}
The ball torsion
estimate \cite[Lemma~2.2]{IMS20} gives
\[
 cr_b^{s/(p-1)}(\dist(x,B^c))^s
 \le \tau(x)
 \le Cr_b^{s/(p-1)}(\dist(x,B^c))^s
 \qquad(x\in B),
\]
where \(c,C>0\) depend only on \(n,s,p\).
The fixed difference in operator normalization is
absorbed into these constants.
Since the radius of \(B\) is fixed in the present argument,
we obtain
\[
 \|\tau\|_{\mathrm L^\infty(\R^n)}\le C,
 \qquad
 \tau(x)\ge c\,(\dist(x,B^c))^s
 \qquad(x\in B).
\]
Choose \(\psi\in C_0^\infty(B_{3r_b/8}(b))\), with
\(0\le\psi\le1\) and \(\psi=1\) on \(B_{r_b/4}(b)\).
Set \(a_0=m/2\) and \(w=\varepsilon\tau+a_0\psi\).
For \(x\in B\setminus\overline E\), the perturbation is
separated from \(x\), so
\[
 \Fsp w(x)
 =\varepsilon^{p-1}
 +\int_E
 \frac{\Jp(\varepsilon(\tau(x)-\tau(y))-a_0\psi(y))
       -\Jp(\varepsilon(\tau(x)-\tau(y)))}
 {|x-y|^{n+sp}}\d y
\]
in the distributional sense. The scalar inequality
\[
 \Jp(A-b_0)-\Jp(A)\le-2^{2-p}b_0^{p-1},
 \qquad A\in\R,\quad b_0\ge0,
\]
implies
\[
 \Fsp w\le\varepsilon^{p-1}-c a_0^{p-1}
 \quad\text{in }B\setminus\overline E.
\]
Choose \(\varepsilon>0\), depending only on \(n,s,p\), so that
the last expression is nonpositive and
\(\varepsilon\|\tau\|_\infty+a_0\le m\). Hence $w$ is a subsolution in $B\setminus\overline E$, \(w\le u\) on \(E\), and \(w=0\le u\) on \(B^c\).
Weak comparison in \(B\setminus\overline E\) yields \(w\le u\)
everywhere. To check the energy of the test, choose
\(0\le\zeta\le1\) in \(C_0^\infty(B_4)\), equal to \(1\)
near \(\overline B\). Then
\((w-u)_+=(w-\zeta u)_+\) has finite global energy and
vanishes outside \(B\setminus\overline E\). Since this annulus
is Lipschitz, it belongs to its zero-exterior energy space.
The subtraction argument in Proposition~\ref{prop:comparison}
applies directly, with the local interactions controlled by
the \(p\)-energy in \(B_4\) and the remaining interactions
controlled by the \((p-1)\)-tail. Finally, to complete the argument, note that for \(0<t<r_b/2\), the point \(z+te_n\) has distance \(t\)
from \(B^c\). Consequently,
\[
 u(z+te_n)\ge w(z+te_n)\ge\varepsilon\tau(z+te_n)\ge c t^s.
\]
For points of \(\HH\cap B_1\) with \(x_n\ge r_b/2\),
interior Harnack gives the same lower bound.
Restoring \(u(e_n)\) proves the lemma.
\end{proof}


Fix a bounded smooth domain \(G\subset\HH\), depending only on
\(n\), such that \(\HH\cap B_4\subset G\) and its boundary
agrees with \(\partial\HH\) in \(B_4\).
It can be obtained by rounding a larger half-ball away from \(B_4\).
In particular,
\[
 d_G(x)=x_n\quad\text{for }x\in\HH\cap B_2.
\]
The following localized consequence of the weighted estimates
in \cite{IMS20} specifies exactly where the equation is needed.

\begin{lem}\label{lem:localized-oscillation}
Let \(p\ge2\) and let \(v\in\Wzero(G)\) be nonnegative.
Suppose that
\[
 0\le Q:={v}/{d_G^s}\le1\quad\text{in }G,\qquad
 -1\le\Fsp v\le1\quad\text{in }G\cap B_{5/4}
\]
in the weak sense. Then there are \(\alpha\in(0,s/(p-1))\),
\(R_0\in(0,1/8)\), and \(C\ge1\), depending only on \(n,s,p\),
such that
\[
 \mathop{\mathrm{ess\,osc}}_{G\cap B_r(z)}Q\le C r^\alpha
 \quad\text{for }z\in\partial\HH\cap\overline{B_1},
 \quad 0<r\le R_0.
\]
\end{lem}

\begin{proof}
Fix \(z\) as in the statement.
For \(\ell\in\{1,p-1\}\), use the quotient tails
\[
 \mathcal T_\ell(g,z,R)
 =\left(\int_{G\setminus B_R(z)}
       \frac{|g(y)|^\ell}{|y-z|^{n+s}}\d y\right)^{1/\ell}.
\]
These have kernel exponent \(n+s\), and are different from
\(\Tail\) in \eqref{eq:tail}.
Put \(a=s/(p-1)\), \(R_k=R_0 8^{-k}\), and
\(D_k=G\cap B_{R_k}(z)\).
We construct nested intervals \([m_k,M_k]\subset[0,1]\)
containing the essential range of \(Q\) on \(D_k\), with
\begin{equation}\label{eq:flat-brackets}
 M_k-m_k=w_k,\qquad w_k=8^{-k\alpha}.
\end{equation}
Take \(m_0=0\) and \(M_0=1\). Suppose the construction holds through step \(k\).
Put \(R=R_k/2\), and choose the interior ball
\[
 \widetilde B_k=B_{R_k/8}(z+7R_ke_n/8)
 \subset D_k\setminus B_{3R_k/4}(z).
\]
Apply \cite[Propositions~3.7 and~4.5]{IMS20} to \(v\),
with radius \(R=R_k/2\) and levels \(m_k\) and \(M_k\),
respectively. The required bounds on \(G\cap B_{2R}(z)\)
follow from the inductive brackets. The lower estimate
uses \(\Fsp v\ge-1\), while the upper estimate uses
\(\Fsp v\le1\). Restricting the lower estimate to
\(D_{k+1}\) and absorbing the fixed normalization factors
into the constants, we obtain, for every
\(\varepsilon>0\),
\begin{align*}
 \essinf_{D_{k+1}}(Q-m_k)
 &\ge \sigma\avgint_{\widetilde B_k}(Q-m_k)\d x
       -\varepsilon\|Q-m_k\|_{\mathrm L^\infty(G\cap B_{R_k/2}(z))}\\
 &\quad-C R_k^s\mathcal T_1((m_k-Q)_+,z,R_k)\\
 &\quad-C_\varepsilon R_k^a
       \bigl(1+m_k+\mathcal T_{p-1}((m_k-Q)_+,z,R_k)\bigr),\\
 \essinf_{D_{k+1}}(M_k-Q)
 &\ge \sigma\avgint_{\widetilde B_k}(M_k-Q)\d x
       -\varepsilon\|M_k-Q\|_{\mathrm L^\infty(G\cap B_{R_k/2}(z))}\\
 &\quad-C R_k^s\mathcal T_1((Q-M_k)_+,z,R_k)\\
 &\quad-C_\varepsilon R_k^a
       \bigl(1+M_k+\mathcal T_{p-1}((Q-M_k)_+,z,R_k)\bigr).
\end{align*}
Here \(\sigma\in(0,1]\) and \(C\) depend only on \(n,s,p\),
and \(C_\varepsilon\) also depends on \(\varepsilon\). The two excess averages add to \(w_k\), each supremum
is at most \(w_k\), and the sum of the left sides is
\(w_k-\mathop{\mathrm{ess\,osc}}_{D_{k+1}}Q\).
Choose \(\varepsilon=\sigma/4\), and use
\(0\le m_k\le M_k\le1\). Adding the estimates gives
\begin{equation}\label{eq:flat-oscillation-step}
 \mathop{\mathrm{ess\,osc}}_{D_{k+1}}Q
 \le(1-\vartheta)w_k+C R_k^a
       +C R_k^s\mathcal E_{k,1}
       +C R_k^a\mathcal E_{k,p-1},
\end{equation}
where \(\vartheta=\sigma/2\) and
\begin{equation}\label{Itails}
 \mathcal E_{k,\ell}
 =\mathcal T_\ell((m_k-Q)_+,z,R_k)
  +\mathcal T_\ell((Q-M_k)_+,z,R_k).
\end{equation}
Only the equation in \(G\cap B_{R_k/2}(z)\) is used.
Taking \(R_0\) below a fixed geometric radius makes every
application legitimate. The factor \(2\) in the normalization
of \cite{IMS20} is absorbed into the constants.

We next estimate the two tails in \eqref{Itails} explicitly.
On \(D_i\setminus D_{i+1}\), \(0\le i<k\), nesting gives
\[
 (m_k-Q)_+ +(Q-M_k)_+\le w_i-w_k.
\]
Indeed, at most one summand is nonzero, and
\(m_k-m_i\le w_i-w_k\) and \(M_i-M_k\le w_i-w_k\).
On \(G\setminus D_0\), the global bound \(0\le Q\le1\)
gives the same estimate with \(w_0-w_k\).
Integrating over these annuli and the outer region yields
\begin{equation}\label{eq:flat-quotient-tails}
 R_k^{s/\ell}\mathcal E_{k,\ell}
 \le C w_k S_\ell(\alpha)^{1/\ell},\qquad
 S_\ell(\alpha)=
 \sum_{j=1}^\infty\frac{(8^{j\alpha}-1)^\ell}{8^{js}}.
\end{equation}
For example, the annulus with index \(i\) contributes at most
\(C R_i^{-s}(w_i-w_k)^\ell\) before taking the
\(\ell\)-th root. After multiplication by \(R_k^s\),
this equals
\(Cw_k^\ell(8^{(k-i)\alpha}-1)^\ell8^{-s(k-i)}\).
The outer region has the same bound with \(i=0\).
For \(k=0\), both tail errors vanish. For \(0<\alpha<s/(2(p-1))\), these series converge, and
dominated convergence gives \(S_\ell(\alpha)\to0\) as
\(\alpha\downarrow0\), for both values of \(\ell\).
Choose \(\alpha\) so small that the terms in
\eqref{eq:flat-quotient-tails}, inserted in
\eqref{eq:flat-oscillation-step}, contribute at most
\(\vartheta w_k/4\), and that
\(8^{-\alpha}\ge1-\vartheta/4\).
Since \(\alpha<a\),
\[
 R_k^a/w_k=R_0^a8^{-k(a-\alpha)}\le R_0^a.
\]
We can therefore decrease \(R_0\), independently of \(k,z,v\),
so that the remaining term contributes at most
\(\vartheta w_k/4\).
It follows that
\[
 \mathop{\mathrm{ess\,osc}}_{D_{k+1}}Q
 \le(1-\vartheta/2)w_k\le w_{k+1}.
\]
An interval of length \(w_{k+1}\) containing this essential
range can be chosen inside \([m_k,M_k]\). This proves the
induction. For \(R_{k+1}<r\le R_k\), we use \(G\cap B_r(z)\subset D_k\)
to obtain
\[
 \mathop{\mathrm{ess\,osc}}_{G\cap B_r(z)}Q
 \le w_k\le(8/R_0)^\alpha r^\alpha.
\]
All constants are uniform for the indicated boundary points.
\end{proof}

\begin{rem}
In Lemma \ref{lem:localized-oscillation} the global quotient bound is imposed on the auxiliary
localized function \(v\). In the application below, the
cutoff used in the definition of $v$ is chosen so that the local profile estimate
controls the quotient on its support and the localized
function vanishes near the artificial boundary of \(G\).
The positive-tail estimate controls the change in the
equation caused by the cutoff. A subsequent normalization
then gives both bounds in the hypothesis.
\end{rem}

\begin{proof}[Proof of Theorem~\ref{thm:flat}]
Normalize \(u(e_n)=1\).
The two-sided bounds and \eqref{eq:flat-normalization} follow
from Lemma~\ref{lem:flat-two-sided}.
Choose \(0\le\chi\le1\), with
\(\chi\in C_0^\infty(B_{7/4})\) and \(\chi=1\) on \(B_{3/2}\),
and put \(v=\chi u\).
Definition~\ref{def:local-zero} implies \(v\in\Wzero(G)\).
Since \(d_G=x_n\) on the support of \(v\), the upper estimate gives
\begin{equation}\label{eq:cutoff-quotient-bound}
 0\le v/d_G^s\le C\quad\text{throughout }G.
\end{equation}
Lemma~\ref{lem:cutoff}, with \(R=3/2\) and \(\theta=5/6\),
together with the local upper bound and
\eqref{eq:normalized-mass}, gives
\begin{equation}\label{eq:flat-cutoff-rhs}
 \Fsp v=f_\chi\quad\text{in }\HH\cap B_{5/4},
 \qquad \|f_\chi\|_\infty\le C.
\end{equation}
This identity is justified in the weak formulation.
No bound for \(\Fsp v\) in the transition region of \(\chi\)
is required.
Dividing \(v\) by a sufficiently large constant depending only
on \(n,s,p\), Lemma~\ref{lem:localized-oscillation} applies.
As \(v=u\) on \(B_{3/2}\), it yields
\begin{equation}\label{eq:flat-boundary-oscillation}
 \osc_{\HH\cap B_r(z)}q_u\le C r^\alpha,
 \quad z\in\partial\HH\cap\overline{B_1},\quad 0<r\le R_0.
\end{equation}
The quotient is continuous in \(\HH\cap B_{3/2}\) by interior
regularity, so essential oscillations can be replaced by
ordinary oscillations. The estimate gives a unique continuous
boundary value at every point of the flat patch.

We include the passage from \eqref{eq:flat-boundary-oscillation}
to a H\"older estimate for all pairs of points.
Let \(x\in\HH\cap B_1\) and \(t=x_n\).
The upper bound \(u(y)\le C(y_n)_+^s\) in \(B_2\), together
with \eqref{eq:normalized-mass}, implies
\begin{equation}\label{eq:flat-whitney-tail}
 \|u\|_{\mathrm L^\infty(B_{t/2}(x))}
 +\Tail(u,x,t/2)\le C t^s.
\end{equation}
To verify the tail bound, split the integral at \(B_2\).
In \(B_2\setminus B_{t/2}(x)\), use
\((y_n)_+\le t+|y-x|\le3|y-x|\).
The radial integral is bounded by
\(C\int_{t/2}^{3}\rho^{-s-1}\d\rho\le Ct^{-s}\).
Multiplication by \(t^{sp}\) and the \((p-1)\)-th root
give \(Ct^s\).
On \(B_2^c\), the weighted integral is bounded and
\(t^{sp}\le t^{s(p-1)}\). The interior H\"older estimate therefore gives an
\(\eta\in(0,1)\), depending only on \(n,s,p\), such that
\[
 [u]_{C^\eta(B_{t/8}(x))}\le C t^{s-\eta},
 \qquad
 [q_u]_{C^\eta(B_{t/8}(x))}\le C t^{-\eta}.
\]
For \(x,y\in\HH\cap B_1\), put
\(\delta=|x-y|\) and \(t=\min\{x_n,y_n\}\), and choose the
name \(x\) so that \(x_n=t\).
If \(\delta\ge t/8\) and \(\delta\) is sufficiently small,
both points lie in a boundary ball of radius \(C\delta\).
Equation~\eqref{eq:flat-boundary-oscillation} bounds their
quotient difference by \(C\delta^\alpha\).
If \(\delta<t/8\) and \(t<R_0/4\), the interior estimate and
the boundary oscillation estimate together give
\[
 |q_u(x)-q_u(y)|
 \le C\min\{t^\alpha,(\delta/t)^\eta\}
 \le C\delta^{\alpha\eta/(\alpha+\eta)}.
\]
If \(t\ge R_0/4\), the interior estimate is uniform.
Pairs at a fixed positive separation are controlled by the
uniform bound for \(q_u\).
Thus the asserted estimate holds with $\beta_F={\alpha\eta}/{(\alpha+\eta)}$. It extends to the closure by continuity.
The positive lower bound persists at the flat boundary.
Restoring the factor \(u(e_n)\) completes the proof.
\end{proof}

\begin{prop}\label{prop:flat-bhi}
Let \(u,v\) satisfy the hypotheses of Theorem~\ref{thm:flat}.
Then
\[
 C_F^{-2}\frac{u(e_n)}{v(e_n)}
 \le\frac{u(x)}{v(x)}
 \le C_F^2\frac{u(e_n)}{v(e_n)}
 \qquad\text{for }x\in\HH\cap B_1,
\]
and
\[
 \left|\log\frac{u(x)}{v(x)}-\log\frac{u(y)}{v(y)}\right|
 \le 2C_F^2|x-y|^{\beta_F}
 \qquad\text{for }x,y\in\HH\cap B_1.
\]
\end{prop}

\begin{proof}
Since \(u/v=q_u/q_v\), the first conclusion follows directly from
the bounds in \eqref{eq:flat-assumption}.
For positive numbers \(a,b\),
\(|\log a-\log b|\le|a-b|/\min\{a,b\}\).
Apply this inequality to \(q_u\) and \(q_v\), and add.
\end{proof}

\subsection{Interpolation and gradient approximation}

\begin{lem}\label{lem:interpolation}
Let \(0<\sigma<1\), \(M_{\mathrm{reg}}\ge1\), and \(0<\varepsilon\le1\).
Suppose that \(f\in C^{1,\sigma}(B_2)\),
\(\|f\|_{\mathrm L^\infty(B_2)}\le\varepsilon\), and
\([\nabla f]_{C^\sigma(B_2)}\le M_{\mathrm{reg}}\).
Then
\begin{equation}\label{eq:interpolation}
 \|\nabla f\|_{\mathrm L^\infty(B_1)}
 \le C_\sigma
 \left(M_{\mathrm{reg}}^{1/(1+\sigma)}\varepsilon^{\sigma/(1+\sigma)}
       +\varepsilon\right).
\end{equation}
\end{lem}

\begin{proof}
For \(x\in B_1\), a unit vector \(e\), and \(0<t<1/2\),
the fundamental theorem of calculus gives
\[
 |\nabla f(x)\cdot e|
 \le\frac{|f(x+te)-f(x)|}{t}
       +\frac1t\int_0^t
           |\nabla f(x+\tau e)-\nabla f(x)|\d\tau
 \le\frac{2\varepsilon}{t}+M_{\mathrm{reg}}t^\sigma.
\]
Choose \(t\) to be a fixed small multiple of
\(\min\{1,(\varepsilon/M_{\mathrm{reg}})^{1/(1+\sigma)}\}\).
Taking the supremum in \(e\) proves \eqref{eq:interpolation}.
\end{proof}

Together with Proposition~\ref{prop:C1},
Lemma~\ref{lem:interpolation} converts uniform approximation by
a uniformly \(C^{1,\sigma}\) family into gradient approximation
on smaller balls, provided the local solution bounds and tails
are uniform.

The next result gives a precise gradient consequence of a
weighted boundary estimate. Its hypotheses are explicit, so no
differentiation of a merely H\"older quotient is used.

\begin{prop}\label{prop:flat-gradient}
Assume \eqref{eq:range}.
Suppose that \(u\ge0\) in \(\R^n\), that
\(\Fsp u=0\) in \(\HH\cap B_2\), and that
\(u=0\) on \(B_2\setminus\HH\).
Let
\[
 \|u\|_{\mathrm L^\infty(B_2)}+\Tail(u,0,2)\le M_{\mathrm{reg}},\qquad
 q=\frac{u}{x_n^s}\in C^\beta(\overline{\HH\cap B_1}),
 \qquad m\le q,\quad \|q\|_{C^\beta}\le M_{\mathrm{reg}},
\]
where \(m>0\), \(M_{\mathrm{reg}}\ge1\), and \(0<\beta<1\).
Then there are \(t_0>0\), \(C\ge1\), and \(\gamma>0\), depending
only on \(n,s,p,m,M_{\mathrm{reg}},\beta\), such that, for
\(|z'|<1/4\) and \(0<t<t_0\),
\begin{equation}\label{eq:gradient-asymptotic}
 \left|t^{1-s}\nabla u(z',t)-s q(z',0)e_n\right|
 \le Ct^\gamma.
\end{equation}
Here one may take \(\gamma=\beta\sigma/(1+\sigma)\), where
\(\sigma\) is as in Proposition~\ref{prop:C1}.
After decreasing \(t_0\),
\[
 C^{-1}\frac{u(z',t)}{t}
 \le|\nabla u(z',t)|
 \le C\frac{u(z',t)}{t},
 \qquad
 \partial_n u(z',t)\ge C^{-1}\frac{u(z',t)}{t}.
\]
\end{prop}

\begin{proof}
Fix \(z=(z',0)\), put \(a=q(z)\), and define
\[
 U_t(\xi)=\frac{u(z+t\xi)}{a t^s}.
\]
On \(B_{1/2}(e_n)\),
\[
 |U_t(\xi)-(\xi_n)^s|
 =(\xi_n)^s\left|\frac{q(z+t\xi)-q(z)}{a}\right|
 \le C(m,M_{\mathrm{reg}})t^\beta.
\]
For sufficiently small \(t\), the function \(U_t\) is globally
nonnegative and fractional \(p\)-harmonic in \(B_{1/2}(e_n)\).
Since
\[
 U_t(e_n)=\frac{q(z+te_n)}a\le\frac{M_{\mathrm{reg}}}{m},
\]
interior Harnack and Lemma~\ref{lem:positive-tail}, with radius
\(1/4\), give
\[
 \|U_t\|_{\mathrm L^\infty(B_{1/4}(e_n))}
 +\Tail(U_t,e_n,1/4)\le C(m,M_{\mathrm{reg}},n,s,p).
\]
Proposition~\ref{prop:C1} therefore gives a uniform
\(C^{1,\sigma}\) bound on \(B_{1/8}(e_n)\).
Apply the scaled version of Lemma~\ref{lem:interpolation} to
\(U_t-(\xi_n)^s\) on this ball. At \(\xi=e_n\) it gives
\[
 |\nabla U_t(e_n)-se_n|\le Ct^{\beta\sigma/(1+\sigma)}.
\]
Since \(\nabla U_t(e_n)=a^{-1}t^{1-s}\nabla u(z',t)\),
this proves \eqref{eq:gradient-asymptotic}.
The bounds \(m\le a,q\le M_{\mathrm{reg}}\) yield the final conclusions.
\end{proof}

For \(x_n>0\), set $q_u(x)={u(x)}/{x_n^s}$. We use the same notation for its H\"older continuous
extension to the flat boundary, supplied by the preceding
boundary estimates. In particular,
\[
 q_u(z',0)=\lim_{t\downarrow0}\frac{u(z',t)}{t^s}.
\]

\begin{cor}\label{cor:flat-gradient-unconditional}
Assume \eqref{eq:range}, and let \(u\) satisfy the hypotheses
of Theorem~\ref{thm:flat}.
Then there exist \(t_0\in(0,1/4)\), \(\gamma>0\), and \(C\ge1\),
depending only on \(n,s,p\), such that
\[
 \left|t^{1-s}\nabla u(z',t)
       -s q_u(z',0)e_n\right|
 \le C u(e_n)t^\gamma
 \quad\text{for }|z'|<1/4,\quad 0<t<t_0.
\]
In the same region,
\[
 C^{-1}\frac{u(z',t)}{t}
 \le|\nabla u(z',t)|
 \le C\frac{u(z',t)}{t},\qquad
 \partial_n u(z',t)\ge C^{-1}\frac{u(z',t)}{t}.
\]
\end{cor}

\begin{proof}
Apply Theorem~\ref{thm:flat} to \(u/u(e_n)\).
Its quotient bounds and \eqref{eq:flat-normalization} verify
all hypotheses of Proposition~\ref{prop:flat-gradient} with
constants depending only on \(n,s,p\).
One may take
\(\gamma=\beta_F\sigma/(1+\sigma)\), where \(\sigma\) is the
interior gradient exponent in Proposition~\ref{prop:C1}.
\end{proof}

\section{Approximation in Reifenberg domains}\label{sec:transfer}

We here prove that, in sufficiently flat Reifenberg domains,
positive solutions can be approximated by the half-space
profile on balls whose radii are comparable to the distance
of their centers from the boundary. The approximation is
measured relative to the value of the solution at the center
of the ball. The proof uses a compactness argument that accounts for
exterior mass escaping to infinity along the rescaled
sequence. Although such mass is not detected by local
convergence, it may contribute to the limiting non-local
equation. We show that this contribution is a nonnegative
constant right-hand side. Comparison with torsion functions
in expanding balls forces the constant to vanish, so the
limiting function is fractional \(p\)-harmonic.

Let \(0<s<1\), \(p\ge2\), and let \(\Om\) be a Reifenberg flat
domain. Recall that \(D_R(w)=\Om\cap B_R(w)\) and
\(B_R=B_R(0)\).
Suppose that \(w\in\partial\Om\) and that
\(u\in\Ws_{\mathrm{loc}}(D_{4r}(w))
\cap\mathrm L^{p-1}_{sp}(\R^n)\) is a nontrivial, globally
nonnegative weak solution in \(D_{4r}(w)\).
For \(x\in D_{r/8}(w)\), put \(\rho=d_\Om(x)\).
Then \(0<\rho\le|x-w|<r/8\) and
\(B_\rho(x)\subset D_{r/4}(w)\).
Interior positivity permits the normalization
\[
 U_x(\xi)={(u(x+\rho\xi/4))}/{u(x)}.
\]
The function \(U_x\) is fractional \(p\)-harmonic in \(B_4\)
and satisfies \(U_x(0)=1\).
Interior Harnack on \(B_{\rho/2}(x)\),
Lemma~\ref{lem:positive-tail} with radius \(\rho/4\),
and scaling give
\begin{equation}\label{eq:whitney-normalization}
 \|U_x\|_{\mathrm L^\infty(B_2)}
 +\Tail(U_x,0,2)\le C_{\mathrm W},
 \qquad C_{\mathrm W}=C_{\mathrm W}(n,s,p).
\end{equation}
Here enlarging the tail radius from \(\rho/4\) to \(\rho/2\)
costs only the fixed factor \(2^{sp/(p-1)}\).

The purpose of this section is to prove the following theorem.
\begin{thm}\label{thm:approximation}
Let \(0<s<1\) and \(p\ge2\).
For every \(\varepsilon>0\), there exist
\(\delta_\varepsilon>0\) and \(c_\varepsilon\ge8\), depending
only on \(n,s,p,\varepsilon\), with the following property.
Suppose that \(\Om\) is \((\delta,r_0)\)-Reifenberg flat,
\(\delta\le\delta_\varepsilon\), \(w\in\partial\Om\), and
\(0<r<r_0/8\).
Let \(u\in\Ws(B_{4r}(w))\cap \mathrm L^{p-1}_{sp}(\R^n)\) be
nontrivial and nonnegative in \(\R^n\), and satisfy
\[
 \Fsp u=0\quad\text{in }D_{4r}(w),\qquad
 u=0\quad\text{on }B_{4r}(w)\setminus\Om
\]
in the sense of Definition~\ref{def:local-zero}.
For every \(x\in D_{r/c_\varepsilon}(w)\), put
\(\rho=d_\Om(x)\) and
\[
 U_x(\xi)={(u(x+\rho\xi/4))}/{u(x)}.
\]
Then there is a unit vector \(\nu_x\) such that, with
\(V_x(\xi)=(1+\xi\cdot\nu_x/4)^s\) on \(B_2\),
\begin{equation}\label{eq:whitney-approximation}
 \|U_x-V_x\|_{\mathrm L^\infty(B_2)}\le\varepsilon.
\end{equation}
\end{thm}

The proof of Theorem~\ref{thm:approximation} is given at the end
of this section. It uses Lemma~\ref{lem:halfspace-compactness}
to identify the contribution of distant values to the limiting
equation and Lemma~\ref{lem:halfspace-classification} to classify
its solutions. Both lemmas are stated and proved below. We first record the resulting gradient estimates.

\begin{cor}\label{cor:reifenberg-gradient}
Assume \eqref{eq:range}. Let \(\Om\) be
\((\delta,r_0)\)-Reifenberg flat, \(w\in\partial\Om\), and
\(0<r<r_0/8\). Let \(u\) satisfy the solution hypotheses
of Theorem~\ref{thm:approximation}.
Then there are \(\delta_1>0\), \(c_1\ge8\), \(C\ge1\), and
\(\tau>0\), depending only on \(n,s,p\), such that, for
\(\delta\le\delta_1\),
\begin{equation}\label{eq:fundamental}
 C^{-1}\frac{u(x)}{d_\Om(x)}
 \le|\nabla u(x)|
 \le C\frac{u(x)}{d_\Om(x)}
 \qquad\text{for }x\in D_{r/c_1}(w).
\end{equation}
For each such \(x\), there is a unit vector \(\nu_x\) such
that
\begin{equation}\label{eq:increment-cone}
 |u(y)-u(z)|\ge
 C^{-1}\frac{u(x)}{d_\Om(x)}|y-z|
\end{equation}
whenever \(y,z\in B_{\tau d_\Om(x)}(x)\) and
\(|(y-z)\cdot\nu_x|\ge|y-z|/2\).
\end{cor}

\begin{proof}
Write \(\rho=d_\Om(x)\).
By Theorem~\ref{thm:approximation}, choose
\(\delta\le\delta_\varepsilon\) and
\(x\in D_{r/c_\varepsilon}(w)\) so that
\eqref{eq:whitney-approximation} holds.
Since \(B_{\rho/2}(x)\subset\Om\), \(U_x\) solves the equation
in \(B_2\). Apply Proposition~\ref{prop:C1} on a slightly smaller
ball. Its \(C^{1,\sigma}\) norm there is bounded in terms of
\(n,s,p\) by \eqref{eq:whitney-normalization}.
The functions \(V_x\) have uniformly bounded smooth norms on
\(B_2\), and \(\nabla V_x(0)=s\nu_x/4\).
Lemma~\ref{lem:interpolation}, applied on fixed smaller balls,
therefore gives
\[
 \|\nabla U_x-\nabla V_x\|_{\mathrm L^\infty(B_{1/2})}
 \le C\varepsilon^{\sigma/(1+\sigma)}.
\]
Choose \(\varepsilon\) so small that the right side is less
than \(s/32\), and fix \(\delta_1=\delta_\varepsilon\) and
\(c_1=c_\varepsilon\).
Rescaling proves \eqref{eq:fundamental}. On a sufficiently small fixed ball \(B_{\tau_0}\), the same
interpolation estimate and the smoothness of \(V_x\) give
\[
 |\nabla U_x(\xi)-s\nu_x/4|\le s/16.
\]
Integrating along the segment between two points in that ball
gives the lower increment bound in directions with
\(|(\xi-\zeta)\cdot\nu_x|\ge|\xi-\zeta|/2\).
Rescaling, and setting \(\tau=\tau_0/4\), proves
\eqref{eq:increment-cone}.
\end{proof}

\subsection{Compactness and the contribution from infinity}

\begin{lem}\label{lem:halfspace-compactness}
Let \(0<s<1\), \(p\ge2\), \(R_j\to\infty\),
\(\delta_j\to0\), and \(\eta_j\to0\), with
\(\delta_j,\eta_j\ge0\).
Suppose that \(0\in\partial\Om_j\), that \(\Om_j\) is
\((\delta_j,R_j)\)-Reifenberg flat, and that
\begin{equation}\label{eq:compact-slab}
 B_{R_j}\cap\{x_n>\eta_j\}
 \subset\Om_j\cap B_{R_j}
 \subset B_{R_j}\cap\{x_n>-\eta_j\}.
\end{equation}
Let \(v_j\ge0\) in \(\R^n\) be a local weak solution in
\(\Om_j\cap B_{R_j}\), with zero exterior values on
\(B_{R_j}\setminus\Om_j\) in the sense of
Definition~\ref{def:local-zero}. Suppose
that
\begin{equation}\label{eq:compact-mass}
 \int_{\R^n}\frac{v_j(y)^{p-1}}{(1+|y|)^{n+sp}}\d y\le M_{\mathrm{tail}}.
\end{equation}
Then there are a subsequence, a nonnegative function
\(V\in\Ws_{\mathrm{loc}}(\R^n)\cap \mathrm L^{p-1}_{sp}(\R^n)\),
and \(\lambda\ge0\) such that
\begin{equation}\label{eq:compact-convergence}
 v_j\rightharpoonup V\quad\text{in }\Ws(B_L),\qquad
 v_j\to V\quad\text{in }\mathrm L^p(B_L)
 \quad\text{for every }L<\infty.
\end{equation}
Convergence is also locally uniform in \(\HH\), and
\begin{equation}\label{eq:defect-equation}
 \Fsp V=\lambda\quad\text{weakly in }\HH,\qquad
 V=0\quad\text{almost everywhere in }\HH^c.
\end{equation}
For every \(L>0\) and \(\zeta\in C_0^\infty(B_L)\),
\(\zeta V\in\Wzero(\HH\cap B_L)\).
Along this subsequence,
\begin{equation}\label{eq:defect-mass}
 \lambda=\lim_{R\to\infty}\lim_{j\to\infty}
 \int_{\R^n\setminus B_R}
 \frac{v_j(y)^{p-1}}{|y|^{n+sp}}\d y.
\end{equation}
\end{lem}

\begin{proof}
Fix \(L\ge1\). For all sufficiently large \(j\), the equation
and the geometric hypotheses are available on \(B_{8L}\).
The measure density constants of these domains are uniform.
Estimate \eqref{eq:compact-mass} bounds both
\[
 \left(\avgint_{\Om_j\cap B_{8L}}v_j^{p-1}\d y\right)^{1/(p-1)}
 \quad\text{and}\quad \Tail(v_j,0,8L)
\]
by a constant depending on \(n,s,p,M_{\mathrm{tail}},L\).
The local boundedness estimate of \cite[Lemma~2.6]{BKS},
followed by Lemma~\ref{lem:caccioppoli}, therefore gives
\begin{equation}\label{eq:compact-local-bounds}
 \|v_j\|_{\mathrm L^\infty(B_{2L})}
 +\|v_j\|_{\Ws(B_{2L})}\le C_L.
\end{equation}
The boundedness argument in \cite{BKS} uses local cutoff and
truncation test functions. These are admissible under
Definition~\ref{def:local-zero}, and only measure density in
the displayed balls is needed. In particular, no global
boundedness of \(\Om_j\) is required here.

Local compact Sobolev embedding and a diagonal subsequence give
\eqref{eq:compact-convergence} and almost everywhere convergence
on \(\R^n\). Fatou's lemma gives
\[
 \int_{\R^n}\frac{V(y)^{p-1}}{(1+|y|)^{n+sp}}\d y\le M_{\mathrm{tail}}.
\]
Every compact subset of \(\HH\) lies in \(\Om_j\) for large
\(j\). On such a subset, \eqref{eq:compact-local-bounds},
\eqref{eq:compact-mass}, and the interior H\"older estimate
in Proposition~\ref{prop:interior} give equicontinuity.
Thus convergence is locally uniform in \(\HH\).
The slab inclusions imply \(V=0\) almost everywhere in
\(\HH^c\). For a cutoff \(\zeta\) as in the statement,
\(\zeta V\in\Ws(\R^n)\) and vanishes in \(\HH^c\).
Translation a small distance in the positive \(e_n\) direction,
followed by mollification, approximates \(\zeta V\) by
functions in \(C_0^\infty(\HH\cap B_L)\).
This proves the asserted weak exterior condition.

Put \(a=n+sp\) and, for \(R\ge1\), define
\[
 m_j(R)=\int_{\R^n\setminus B_R}
             \frac{v_j(y)^{p-1}}{|y|^a}\d y.
\]
These quantities are uniformly bounded.
After taking a further subsequence, \(m_j(1)\) converges.
Strong local \(\mathrm L^p\) convergence implies convergence of
\(v_j^{p-1}\) in \(\mathrm L^1\) on every bounded set.
Consequently \(m_j(R)\) has a limit \(m(R)\) for every
\(R\ge1\), and
\[
 m(R)=m(1)-\int_{B_R\setminus B_1}
                   \frac{V(y)^{p-1}}{|y|^a}\d y.
\]
In particular, \(m(R)\) is nonnegative and decreasing.
Let \(\lambda=\lim_{R\to\infty}m(R)\ge0\).
Equivalently,
\begin{equation}\label{eq:mass-splitting}
 m(R)=\lambda+
 \int_{\R^n\setminus B_R}\frac{V(y)^{p-1}}{|y|^a}\d y.
\end{equation}

We next identify the equation satisfied by \(V\).
Take \(\varphi\in C_0^\infty(\HH)\), choose \(L\ge1\)
with \(\supp\varphi\subset B_L\), and let \(R>2L\).
Write
\[
 \mathcal B_R(v,\varphi)=\frac12
 \int_{B_R}\int_{B_R}
 \frac{\Jp(v(x)-v(y))(\varphi(x)-\varphi(y))}
 {|x-y|^a}\dxdy.
\]
On \(B_R\times B_R\), the functions
\(\Jp(v_j(x)-v_j(y))\) converge almost everywhere to
\(\Jp(V(x)-V(y))\). They are bounded in
\(\mathrm L^{p/(p-1)}\) for the measure
\(|x-y|^{-a}\dxdy\), by the uniform local energy bound.
They therefore converge weakly in that space.
Since \(\varphi(x)-\varphi(y)\) belongs to the dual
\(\mathrm L^p\) space on \(B_R\times B_R\),
\begin{equation}\label{eq:local-form-limit}
 \mathcal B_R(v_j,\varphi)\longrightarrow
 \mathcal B_R(V,\varphi).
\end{equation}
For \(0\le A\le C_L\), \(b\ge0\), and \(p>2\),
\[
 |\Jp(A-b)+b^{p-1}|
 \le C\bigl(A^{p-1}+A b^{p-2}\bigr).
\]
For \(p=2\), the left side equals \(A\).
H\"older's inequality gives, when \(p>2\),
\[
 \int_{|y|>R}\frac{v_j(y)^{p-2}}{|y|^a}\d y
 \le m_j(R)^{(p-2)/(p-1)}
       \left(\int_{|y|>R}|y|^{-a}\d y\right)^{1/(p-1)}
 \le C R^{-sp/(p-1)}.
\]
Also, for \(x\in B_L\) and \(|y|>R>2L\),
\[
 \bigl||x-y|^{-a}-|y|^{-a}\bigr|
 \le C L R^{-1}|y|^{-a}.
\]
Together these estimates, including the direct calculation for
\(p=2\), show that
\begin{equation}\label{eq:far-form-error}
 \sup_{x\in B_L}
 \left|\int_{|y|>R}
 \frac{\Jp(v_j(x)-v_j(y))}{|x-y|^a}\d y+m_j(R)\right|
 \le C_L\left(R^{-1}+R^{-sp/(p-1)}\right).
\end{equation}
The supremum can be read as an essential supremum.
Only global nonnegativity and the \((p-1)\)-tail bound enter
this estimate. For large \(j\), \(\varphi\) is an admissible interior test
for \(v_j\), so
\[
 0=\mathcal B_R(v_j,\varphi)
 +\int_{B_L}\varphi(x)\int_{|y|>R}
   \frac{\Jp(v_j(x)-v_j(y))}{|x-y|^a}\d y\,\mathrm d x.
\]
Letting \(j\to\infty\) in this identity by
\eqref{eq:local-form-limit} and \eqref{eq:far-form-error} gives
\[
 \left|\mathcal B_R(V,\varphi)
       -m(R)\int_{\HH}\varphi\d x\right|
 \le C_\varphi\left(R^{-1}+R^{-sp/(p-1)}\right).
\]
The local energy and finite tail of \(V\) imply that
\(\mathcal B_R(V,\varphi)\) tends to its full weak form as
\(R\to\infty\). Since \(m(R)\to\lambda\), we obtain
\[
 \frac12\iint_{\R^{2n}}
 \frac{\Jp(V(x)-V(y))(\varphi(x)-\varphi(y))}
 {|x-y|^{n+sp}}\dxdy
 =\lambda\int_{\HH}\varphi\d x.\qedhere
\]
\end{proof}

\subsection{Classification of the limit}

\begin{lem}\label{lem:halfspace-classification}
Let \(0<s<1\), \(p\ge2\), and \(\lambda\ge0\).
Suppose that
\(V\in\Ws_{\mathrm{loc}}(\R^n)\cap \mathrm L^{p-1}_{sp}(\R^n)\)
is nonnegative in \(\R^n\), vanishes almost everywhere in
\(\HH^c\), and satisfies \(\Fsp V=\lambda\) weakly in
\(\HH\). Then \(\lambda=0\) and
\[
 V(x)=b(x_n)_+^s\quad\text{almost everywhere in }\R^n
\]
for a constant \(b\ge0\).
\end{lem}

\begin{proof}
Suppose first that \(\lambda>0\).
For \(R\ge2\), let
\[
 E_R=B_R((R+1/2)e_n),
\]
and let \(T_R\in\Wzero(E_R)\) solve
\[
 \Fsp T_R=1\quad\text{in }E_R,\qquad
 T_R=0\quad\text{in }E_R^c.
\]
The ball \(E_R\) has closure in \(\HH\).
The torsion estimate of \cite[Lemma~2.2]{IMS20}, see \eqref{torsion} with $B$ replaced by $E_R$, with the
normalization of \(\Fsp\), gives
\begin{equation}\label{eq:expanding-torsion}
 T_R(x)\ge cR^{s/(p-1)}d_{E_R}(x)^s
 \qquad\text{for }x\in E_R.
\end{equation}
Comparison yields
\[
 V\ge\lambda^{1/(p-1)}T_R\quad\text{in }E_R.
\]
This comparison is valid in the local energy and tail class.
Indeed, choose a nonnegative smooth cutoff \(\zeta_R\)
equal to one near \(\overline{E_R}\), with compact support
in \(\HH\). Then
\[
 \psi_R=(\lambda^{1/(p-1)}T_R-\zeta_R V)_+
\]
belongs to \(\Ws(\R^n)\), vanishes outside \(E_R\), and
is bounded by \(\lambda^{1/(p-1)}T_R\).
Since \(E_R\) is a ball, \(\psi_R\in\Wzero(E_R)\).
It is therefore an admissible comparison test, and it equals
\((\lambda^{1/(p-1)}T_R-V)_+\) on all of \(\R^n\).
The exterior integrals require only the finite
\((p-1)\)-tail, and monotonicity of \(\Jp\) proves the
comparison. For \(x\in B_{1/4}(e_n)\), one has
\(d_{E_R}(x)\ge1/4\). Hence
\[
 V(x)\ge c\lambda^{1/(p-1)}R^{s/(p-1)}
 \quad\text{almost everywhere in }B_{1/4}(e_n).
\]
Letting \(R\to\infty\) contradicts
\(V\in \mathrm L^p(B_{1/4}(e_n))\). Thus \(\lambda=0\). If \(V\) is not identically zero, it is strictly positive
in \(\HH\) by interior Harnack and global nonnegativity.
The local weak exterior condition on each half-ball follows
by the translation and mollification argument used in
Lemma~\ref{lem:halfspace-compactness}.
Put \(h(x)=(x_n)_+^s\).
For any \(x,y\in\HH\) and
\(R>\max\{|x|,|y|\}\), apply
Proposition~\ref{prop:flat-bhi} to
\(V(R\,\cdot)\) and \(h(R\,\cdot)\).
Its constants are independent of the values at \(Re_n\),
and it gives
\[
 \left|\log\frac{V(x)}{h(x)}-
       \log\frac{V(y)}{h(y)}\right|
 \le 2C_F^2\left(\frac{|x-y|}{R}\right)^{\beta_F}.
\]
Let \(R\to\infty\). The quotient \(V/h\) is constant
in \(\HH\), as required. If \(V=0\), take \(b=0\).
\end{proof}

\subsection{Proof of Theorem~\ref{thm:approximation}}

To prove Theorem~\ref{thm:approximation},  suppose that the conclusion fails for some \(\varepsilon_0>0\).
For every integer \(j\ge3\), choose a counterexample with
\(\delta_j\le j^{-2}\),
\(x_j\in\Om_j\cap B_{r_j/j^2}(w_j)\), and a solution
\(u_j\) satisfying the stated hypotheses.
Put
\[
 \rho_j=d_{\Om_j}(x_j),\qquad
 z_j\in\partial\Om_j,\qquad |x_j-z_j|=\rho_j.
\]
Then
\[
 \rho_j\le |x_j-w_j|<r_j/j^2,\qquad
 |z_j-w_j|<2r_j/j^2.
\]
Use Reifenberg flatness at \(z_j\) on the scale \(j\rho_j\).
Choose an orthogonal matrix \(Q_j\) whose last column points
to the interior side of the approximating plane, and define
\[
 \widetilde\Om_j=\rho_j^{-1}Q_j^{\mathsf T}(\Om_j-z_j),
 \qquad
 v_j(y)=\frac{u_j(z_j+\rho_jQ_jy)}{u_j(x_j)},
 \qquad
 a_j=\frac{Q_j^{\mathsf T}(x_j-z_j)}{\rho_j}.
\]
The scaled domain is Reifenberg flat on all scales at most
\(j\), and
\begin{equation}\label{eq:blowup-slab}
 B_j\cap\{y_n>j\delta_j\}
 \subset\widetilde\Om_j\cap B_j
 \subset B_j\cap\{y_n>-j\delta_j\}.
\end{equation}
Moreover, \(v_j\) is a local weak solution with zero exterior
values in \(\widetilde\Om_j\cap B_j\). Indeed,
\(B_{j\rho_j}(z_j)\subset B_{4r_j}(w_j)\). We have
\[
 |a_j|=1,\qquad d_{\widetilde\Om_j}(a_j)=1,
 \qquad v_j(a_j)=1.
\]
The point \((a_j',-j\delta_j)\) lies in
\(B_j\setminus\widetilde\Om_j\). Consequently,
\[
 1\le a_{j,n}+j\delta_j,
 \qquad 1-j\delta_j\le a_{j,n}\le1,
\]
so \(a_j\to e_n\).
Since \(B_1(a_j)\subset\widetilde\Om_j\cap B_j\),
Lemma~\ref{lem:positive-tail} with radius \(1/2\), and
interior Harnack, imply
\[
 \Tail(v_j,a_j,1/2)+
 \|v_j\|_{\mathrm L^\infty(B_{1/2}(a_j))}\le C.
\]
As \(|a_j|=1\), this gives the fixed normalization
\begin{equation}\label{eq:blowup-mass}
 \int_{\R^n}\frac{v_j(y)^{p-1}}{(1+|y|)^{n+sp}}\d y\le C.
\end{equation}
Apply Lemma~\ref{lem:halfspace-compactness} with
\(R_j=j\) and \(\eta_j=j\delta_j\).
After passage to a subsequence, the limit \(V\) satisfies
\(V(e_n)=1\), because convergence is uniform near \(e_n\)
and \(a_j\to e_n\).
Lemma~\ref{lem:halfspace-classification} gives
\[
 V(y)=(y_n)_+^s.
\]
Put \(h(y)=(y_n)_+^s\) and \(\nu_j=Q_je_n\).
For \(\xi\in B_2\),
\[
 \frac{u_j(x_j+\rho_j\xi/4)}{u_j(x_j)}
 =v_j(a_j+Q_j^{\mathsf T}\xi/4),\qquad
 (1+\xi\cdot\nu_j/4)^s
 =h(e_n+Q_j^{\mathsf T}\xi/4).
\]
All the arguments in these two expressions lie, for large
\(j\), in one fixed compact subset of \(\HH\).
Uniform convergence to \(h\), and \(a_j\to e_n\), therefore
imply
\[
 \sup_{\xi\in B_2}
 \left|\frac{u_j(x_j+\rho_j\xi/4)}{u_j(x_j)}
             -(1+\xi\cdot\nu_j/4)^s\right|\longrightarrow0.
\]
This contradicts the defining failure with error
\(\varepsilon_0\) for every unit vector. This completes the proof of Theorem~\ref{thm:approximation}.\qed

\begin{rem} The constants in Theorem~\ref{thm:approximation} are obtained
by compactness. The proof gives their dependence on
\(n,s,p,\varepsilon\), but does not give an explicit rate
for \(\delta_\varepsilon\) or \(c_\varepsilon\).
Along the subsequence selected in the compactness argument,
the normalized functions also satisfy
\begin{equation}\label{eq:blowup-tail-tightness}
 \lim_{R\to\infty}\limsup_{j\to\infty}
 \int_{\R^n\setminus B_R}
 \frac{v_j(y)^{p-1}}{|y|^{n+sp}}\d y=0.
\end{equation}
This is a consequence of the equation and nonnegativity.
It is not inferred from the bounded total tail alone.
No growth estimate is assumed in the proof.
\end{rem}

\section{Linearized equations and their kernels}
\label{sec:linearization}

This section develops the linear equations arising from
differences and first variations of fractional \(p\)-harmonic
functions. Linearization at a positive solution \(U\)
produces the kernel
\[
 K_U(x,y)=(p-1)
 \frac{|U(x)-U(y)|^{p-2}}{|x-y|^{n+sp}}.
\]
For \(p>2\), this kernel vanishes for distinct points
with \(U(x)=U(y)\), even where \(\nabla U\) is nonzero.
Its degeneracy therefore depends on increments between
pairs of points, and gradient non-degeneracy does not
imply a uniform pointwise ellipticity bound.

We first derive the exact linear equation satisfied by
the difference \(v-u\), with averaged kernel \(K_{u,v}\).
This identity requires no differentiability of a family
of solutions. It is used later in
Proposition~\ref{prop:ordered-derivative} to obtain energy
bounds for difference quotients and pass to the weak
equation for the parameter derivative along ordered
families of solutions.

We then express these kernels through weights on pairs
and prove energy comparisons on Whitney balls.
These estimates show how directional lower bounds for
increments yield coercivity of the quadratic form,
despite the possible vanishing of the kernel on
individual pairs. They hold for arbitrary functions
of finite energy, not only for differences of nonlinear
solutions. Their single-background case, obtained by
taking \(u=v=U\) and hence \(K_{u,v}=K_U\), supplies
the interior part of the weighted Poincar\'e inequality
in Proposition~\ref{prop:variable-poincare}.

Finally, we introduce the ground-state transformation
and examine the kernel associated with the half-space
profile \((x_n)_+^s\). We identify its scaling and the
weighted structure governing boundary regularity of
the quotient. This model prepares the analysis of
variable positive solutions and their transformed
quadratic forms in Section \ref{sec:variable}.

\subsection{Exact equations for differences}

Throughout this section \(p\ge2\), unless a narrower range is
stated. We write \(\Delta_{xy}u=u(x)-u(y)\).
When \(p=2\), the map \(a\mapsto\Jp(a)\) is the identity,
and its derivative is \(1\) at every \(a\in\R\).
Accordingly, in the linearization formulas we interpret
the factor \(|a|^{p-2}\) as \(1\), also when \(a=0\).

\begin{lem}\label{lem:scalar}
For \(a,b\in\R\), define
\[
 a_p(a,b)=(p-1)\int_0^1|(1-t)a+tb|^{p-2}\d t.
\]
Then
\[
 \Jp(b)-\Jp(a)=a_p(a,b)(b-a).
\]
If \(p>2\), there are positive constants \(c_p,C_p\) such that
\begin{equation}\label{eq:scalar-bounds}
 c_p(|a|+|b|)^{p-2}
 \le a_p(a,b)
 \le C_p(|a|+|b|)^{p-2}.
\end{equation}
For \(p=2\), \(a_2(a,b)=1\).
\end{lem}

\begin{proof}
The identity follows from the fundamental theorem of calculus.
The upper bound is immediate.
For the lower bound, use homogeneity to restrict to
\(|a|+|b|=1\). On this compact set the integral is continuous
and positive. Its minimum is therefore positive.
The case \(a=b=0\) follows directly.
\end{proof}

\begin{prop}\label{prop:difference}
Let \(u,v\in\Vs(D)\) solve \(\Fsp u=\Fsp v=0\) in \(D\).
Put \(f=v-u\) and
\begin{equation}\label{eq:averaged-kernel}
 K_{u,v}(x,y)=
 \frac{a_p(\Delta_{xy}u,\Delta_{xy}v)}{|x-y|^{n+sp}}.
\end{equation}
Then \(K_{u,v}(x,y)=K_{u,v}(y,x)\), and
\begin{equation}\label{eq:difference-equation}
 \frac12\iint_{\R^{2n}}
 (f(x)-f(y))(\varphi(x)-\varphi(y))
 K_{u,v}(x,y)\dxdy=0
 \qquad\text{for }\varphi\in\Wzero(D).
\end{equation}
For \(p>2\),
\[
 K_{u,v}(x,y)\asymp_p
 \frac{(|\Delta_{xy}u|+|\Delta_{xy}v|)^{p-2}}
 {|x-y|^{n+sp}}.
\]
\end{prop}

\begin{proof}
Subtract the weak equations and apply Lemma~\ref{lem:scalar}
to each pair \((x,y)\).
All integrals are finite by H\"older's inequality applied to
the \(p\)-energy increments.
Simultaneously changing the signs of both arguments of \(a_p\)
does not change it, proving symmetry.
The comparison follows from \eqref{eq:scalar-bounds}.
\end{proof}

Equation \eqref{eq:difference-equation} is exact.
It does not require differentiability of a family of solutions.
In general \(u\) and \(v\) themselves do not solve the equation
with kernel \(K_{u,v}\). A linear boundary comparison applied to
\(f\) therefore requires a companion solution for that same
kernel. This is the point that must be addressed in an adaptation
of the ordered-difference method in \cite[Section~6]{LLN}.

\subsection{Pair weights and Whitney balls}

For \(p>2\), $K_{u,v}(x,y)={\widehat w_{u,v}(x,y)}{|x-y|^{-n-2s}}$ where
\begin{align}
 \widehat w_{u,v}(x,y)
 &=(p-1)\int_0^1
 \left|\frac{(1-t)\Delta_{xy}u+t\Delta_{xy}v}
 {|x-y|^s}\right|^{p-2}\d t,\notag\\
 \widehat w_{u,v}(x,y)&\asymp_p
 \left(\frac{|\Delta_{xy}u|+|\Delta_{xy}v|}{|x-y|^s}\right)^{p-2}:=w_{u,v}(x,y).
 \label{eq:pair-weight}
\end{align}
Hence,  $\widehat w_{u,v}(x,y)$ is the exact weight relative to the fractional
Laplacian of order \(2s\). \(w_{u,v}(x,y)\) is a weight on pairs of points and functions, and need not be comparable to a
weight depending on one endpoint alone.

\begin{prop}\label{prop:whitney-pair}
Assume \eqref{eq:range}. Let \(\Om\) be
\((\delta,r_0)\)-Reifenberg flat, let
\(w\in\partial\Om\), and let \(0<r<r_0/8\).
Let \(\delta_1\) and \(c_1\) be the constants in
Corollary~\ref{cor:reifenberg-gradient}, and assume
\(\delta\le\delta_1\).
Suppose that \(u\) and \(v\) satisfy the solution
hypotheses of that corollary for the same
\(\Om,w,r\), and fix $x_0\in D_{r/c_1}(w)$. Then there is a fixed \(\varepsilon>0\), depending
only on \(n,s,p\), for which the following statements
hold.  Put
\[
 d_0=d_\Om(x_0),\qquad a=u(x_0)+v(x_0),\qquad
 B=B_{\varepsilon d_0}(x_0),\qquad
 \Lambda_B=(a/d_0)^{p-2}.
\]
For $p>2$ and all distinct $x,y\in2B$,
\begin{equation}\label{eq:whitney-weight-upper}
 w_{u,v}(x,y)\le C\Lambda_B
 |x-y|^{(1-s)(p-2)}.
\end{equation}
Moreover, there is a unit vector $e_B\in\R^n$ such that
\begin{equation}\label{eq:whitney-weight-lower}
 w_{u,v}(x,y)\ge c\Lambda_B
 |x-y|^{(1-s)(p-2)}
\end{equation}
whenever $x,y\in2B$, $x\ne y$, and
\[
 |(x-y)\cdot e_B|\ge\frac12|x-y|.
\]
Thus the upper bound holds for every pair, while the
lower bound holds for pairs whose connecting direction
lies in a fixed double cone about $e_B$. Furthermore, for every measurable \(f\),
\begin{equation}\label{eq:whitney-pair-energy}
 c\Lambda_B[f]_{\mathrm H^{\sigma_*/2}(B)}^2
 \le\iint_{(2B)\times (2B)}|f(x)-f(y)|^2K_{u,v}(x,y)\dxdy
 \le C\Lambda_B[f]_{\mathrm H^{\sigma_*/2}(2B)}^2,
\end{equation}
where $\sigma_*=sp-p+2$ and
\[
 [f]_{\mathrm H^{\sigma_*/2}(G)}^2
 =\iint_{G\times G}
 \frac{|f(x)-f(y)|^2}{|x-y|^{n+\sigma_*}}\dxdy.
\]
For a measurable function $f$ on $2B$ with $f|_B\in\mathrm L^2(B)$,
the lower energy bound consequently gives
\begin{equation}\label{eq:whitney-pair-poincare}
 \int_B|f-f_B|^2\d x
 \le C\frac{d_0^{sp}}{a^{p-2}}
       \iint_{(2B)^2}|f(x)-f(y)|^2K_{u,v}(x,y)\dxdy.
\end{equation}
The same kernel and energy comparisons hold with
$B=B_\ell(x_0)$ for any $0<\ell\le\varepsilon d_0$,
using the same coefficient $\Lambda_B=(a/d_0)^{p-2}$.
For such a ball, the Poincar\'e factor is
$C\ell^{\sigma_*}/\Lambda_B$.
The energy conclusions also hold for \(p=2\).
\end{prop}

\begin{proof}
We first assume \(p>2\). Choose \(\varepsilon\in(0,1/8)\)
sufficiently small, depending only on \(n,s,p\).
For \(x\in2B\), the Lipschitz continuity of the distance
function gives
\[
 (1-2\varepsilon)d_0
 \le d_\Om(x)\le(1+2\varepsilon)d_0.
\]
Interior Harnack gives \(u+v\le Ca\) on a fixed interior
ball about \(x_0\) of radius comparable to \(d_0\).
Lemma~\ref{lem:positive-tail} controls the tail at this scale.
Applying the rescaled interior estimate of
Proposition~\ref{prop:C1} on a smaller concentric ball containing
\(2B\) therefore yields
\[
 \sup_{2B}\bigl(|\nabla u|+|\nabla v|\bigr)
 \le C\frac{a}{d_0}.
\]
Since \(2B\) is convex, integration along the segment
joining \(x,y\in2B\) gives
\[
 |u(x)-u(y)|+|v(x)-v(y)|
 \le C\frac{a}{d_0}|x-y|.
\]
By the definition in \eqref{eq:pair-weight},
\[
 w_{u,v}(x,y)=
 \left(
 \frac{|u(x)-u(y)|+|v(x)-v(y)|}{|x-y|^s}
 \right)^{p-2}.
\]
Combining these inequalities proves
\eqref{eq:whitney-weight-upper}.

To obtain the lower bound in \eqref{eq:whitney-weight-lower}, interchange \(u\) and \(v\)
if necessary so that \(u(x_0)\ge a/2\).
The gradient lower bound gives
\[
 |\nabla u(x_0)|\ge c\frac{u(x_0)}{d_0}
 \ge c\frac{a}{d_0}.
\]
Proposition~\ref{prop:C1}, applied with the normalization
in \eqref{eq:whitney-normalization}, gives
\[
 [\nabla u]_{C^\sigma(B_{d_0/4}(x_0))}
 \le C\frac{u(x_0)}{d_0^{1+\sigma}},
\]
where $\sigma\in(0,1)$ and $C$ depend only on $n,s,p$.
For $\varepsilon\le1/8$, we have
$2B\subset B_{d_0/4}(x_0)$. Hence, using the gradient
lower bound in Corollary~\ref{cor:reifenberg-gradient},
\[
 \sup_{\zeta\in2B}
 |\nabla u(\zeta)-\nabla u(x_0)|
 \le C\varepsilon^\sigma\frac{u(x_0)}{d_0}
 \le C\varepsilon^\sigma|\nabla u(x_0)|.
\]
Reducing $\varepsilon$, still depending only on $n,s,p$,
we therefore obtain
\[
 \sup_{\zeta\in2B}
 |\nabla u(\zeta)-\nabla u(x_0)|
 \le\frac14|\nabla u(x_0)|.
\]
Set $e_B={\nabla u(x_0)}/{|\nabla u(x_0)|}$. For \(x,y\in2B\) satisfying
\(|(x-y)\cdot e_B|\ge|x-y|/2\), integration along
the segment gives
\begin{align*}
 |u(x)-u(y)|
 &\ge |\nabla u(x_0)\cdot(x-y)|
      -\sup_{\zeta\in2B}
       |\nabla u(\zeta)-\nabla u(x_0)|\,|x-y|\\
 &\ge\frac14|\nabla u(x_0)|\,|x-y|
 \ge c\frac{a}{d_0}|x-y|.
\end{align*}
The preceding comparison for \(w_{u,v}\) now yields \eqref{eq:whitney-weight-lower} on this cone.

Recall that $\sigma_*=sp-p+2=2-p(1-s)$. The range \eqref{eq:range} ensures
\(0<\sigma_*<2\). Inserting the weight estimates into
\eqref{eq:pair-weight}, and using $-n-2s+(1-s)(p-2)=-n-\sigma_*$, we obtain
\[
 K_{u,v}(x,y)
 \le C\Lambda_B|x-y|^{-n-\sigma_*}
 \qquad(x,y\in2B),
\]
together with the reverse inequality for pairs in the
specified cone. Integrating the upper bound proves
the upper inequality in \eqref{eq:whitney-pair-energy}.

For the lower inequality in \eqref{eq:whitney-pair-energy}, define the fixed double cone
\begin{equation}\label{CGG1}
 \Gamma_B=
 \left\{h\in\R^n\setminus\{0\}:
 |h\cdot e_B|\ge|h|/2\right\}
\end{equation}
and the kernel
\begin{equation}\label{CGG2}
 k_B(x,y)=
 \frac{\mathbf 1_{\Gamma_B}(y-x)}
 {|x-y|^{n+\sigma_*}},
 \qquad x,y\in\R^n,\quad x\ne y.
\end{equation}
This kernel is defined on all of \(\R^{2n}\), and
\begin{equation}\label{CGG3}
 K_{u,v}(x,y)\ge c\Lambda_B k_B(x,y)
 \qquad(x,y\in2B).
\end{equation}
Note that $\Gamma_B$ is the double cone which contains directions about both $e_B$ and $-e_B$.
Thus, even when one branch points out of the ball,
the other branch may point into it. Using this we see that there exists \(\mu=\mu(n)>0\) such that, for every
Euclidean ball \(B'\) and every \(x\in B'\),
\begin{equation}\label{CGG4}
 \bigl|B'\cap(x+\Gamma_B)\bigr|\ge\mu|B'|.
\end{equation}
Indeed, translation, rotation, and dilation reduce this
claim to the unit ball and a fixed double cone.
For every vertex in the closed unit ball, the translated
cone intersects the open unit ball in a set of positive
measure. The intersection measure depends continuously
on the vertex, so compactness gives a uniform positive
lower bound.

By \eqref{CGG4}, the globally defined cone kernel \(k_B\)
satisfies the density hypothesis of the coercivity theorem
of Chaker and Silvestre \cite{CS}, as restated in
\cite[Theorem~7.2]{GJS}, with
\(\gamma=\sigma_*/2\in(0,1)\), \(\lambda=1\), and
\(\mu=\mu(n)>0\).
Therefore, translation, rescaling, and \eqref{CGG3}, give
\begin{align}\label{kesta}
 [f]_{\mathrm H^{\sigma_*/2}(B)}^2
 &\le C
 \iint_{(2B)\times(2B)}|f(x)-f(y)|^2k_B(x,y)
 \,\mathrm d y\,\mathrm d x\notag\\
 &\le\frac{C}{\Lambda_B}
 \iint_{(2B)\times(2B)}|f(x)-f(y)|^2K_{u,v}(x,y)
 \,\mathrm d y\,\mathrm d x.
\end{align}
This proves the lower energy bound in \eqref{eq:whitney-pair-energy}. Only the fixed cone
kernel has been defined globally. The comparison with
\(K_{u,v}\) uses pairs in \(2B\) and requires no extension
of the solutions or their gradients.

Finally, the fractional Poincar\'e inequality gives
\[
 \int_B|f-f_B|^2\,\mathrm d x
 \le C(\varepsilon d_0)^{\sigma_*}
       [f]_{\mathrm H^{\sigma_*/2}(B)}^2,\qquad  f_B=\frac1{|B|}\int_B f\,\mathrm d x.
\]
Combining this with the lower energy bound and the identity
\[
 \frac{d_0^{\sigma_*}}{\Lambda_B}
 =\frac{d_0^{\sigma_*+p-2}}{a^{p-2}}
 =\frac{d_0^{sp}}{a^{p-2}}
\]
proves \eqref{eq:whitney-pair-poincare}.

When \(p=2\), we have \(\Lambda_B=1\),
\(\sigma_*=2s\), and
\(K_{u,v}(x,y)=|x-y|^{-n-2s}\).
The energy bounds and the Poincar\'e estimate then follow
directly. The argument for $p>2$ also applies on every smaller
ball $B_\ell(x_0)$, since the gradient bounds and the same fixed
cone remain valid there. Rescaling the last Poincar\'e estimate
gives the factor $C\ell^{\sigma_*}/\Lambda_B$.
\end{proof}

\begin{rem} The weight estimate \eqref{eq:whitney-weight-upper},
its reverse on the specified cone in  \eqref{eq:whitney-weight-lower}, and the energy comparison
\eqref{eq:whitney-pair-energy} all involve the coefficient
\[
 \Lambda_B=a^{p-2}d_0^{-(p-2)}.
\]
The factor \(d_0^{-(p-2)}\) reflects the local gradient
scale \(a/d_0\) and is essential to these bounds. A pointwise lower bound for all pairs in a Whitney ball
is generally unavailable. For example, when \(n\ge2\)
and \(u=v=(x_n)_+^s\), distinct points with
\(x_n=y_n>0\) have zero weight for \(p>2\).
Harnack estimates and gradient non-degeneracy do not
eliminate this directional degeneracy.
The lower bound needed below is therefore the comparison
of quadratic forms in \eqref{eq:whitney-pair-energy}.
\end{rem}

\subsection{The derivative at a single solution}

For a fixed \(u\in\Vs(D)\), define
\begin{align}
 K_u(x,y)&=(p-1)
          \frac{|\Delta_{xy}u|^{p-2}}{|x-y|^{n+sp}},
          \label{eq:Ku}\\
 \Bu(f,\varphi)&=
 \frac12\iint_{\R^{2n}}
 (f(x)-f(y))(\varphi(x)-\varphi(y))
 K_u(x,y)\dxdy.
 \label{eq:bilinear}
\end{align}
We say that \(f\) is \(\mathcal L_u\)-harmonic in \(D\) if
\(\Bu(f,\varphi)=0\) for compactly supported admissible test functions.
The bilinear formulation is primary. Pointwise principal
values are not assumed.

\begin{prop}\label{prop:homogeneity}
If \(u\in\Vs(D)\) solves \(\Fsp u=0\) in \(D\), then
\(\mathcal L_u u=0\) weakly in \(D\).
More generally,
\[
 \Bu(u,\varphi)
 =\frac{p-1}{2}\iint_{\R^{2n}}
 \frac{\Jp(\Delta_{xy}u)\Delta_{xy}\varphi}
 {|x-y|^{n+sp}}\dxdy.
\]
\end{prop}

\begin{proof}
Multiply \(\Delta_{xy}u\) by the coefficient in
\eqref{eq:Ku}. The identity is pointwise in \((x,y)\).
\end{proof}

\begin{prop}\label{prop:parameter}
Let \(0<s<1\), \(p\ge2\), and let \(D\subset\R^n\)
be a bounded open set. Suppose that \(t\mapsto U_t\)
is a \(C^1\) map from an open interval into \(\Vs(D)\)
and that every \(U_t\) solves \(\Fsp U_t=0\) weakly in \(D\).
Then the derivative \(F_t=\partial_tU_t\) belongs to
\(\Vs(D)\), has finite weighted energy
\begin{equation}\label{lulu1}
 \iint_{Q_D}|\Delta_{xy}F_t|^2K_{U_t}(x,y)
 \,\mathrm d y\,\mathrm d x<\infty,
\end{equation}
and satisfies
\begin{equation}\label{lulu2}
 \mathcal B_{U_t}(F_t,\varphi)=0
 \qquad\text{for every }\varphi\in\Wzero(D).
\end{equation}
\end{prop}

\begin{proof}
Equip \(Q_D\) with the measure
\(\mathrm d\nu=|x-y|^{-n-sp}\,\mathrm d y\,\mathrm d x\),
and put \(p'=p/(p-1)\).
Since \(t\mapsto U_t\) is \(C^1\) into \(\Vs(D)\),
its derivative \(F_t\) belongs to \(\Vs(D)\).
The increment map from \(\Vs(D)\) to
\(\mathrm L^p(Q_D,\nu)\) is linear and continuous,
so \(t\mapsto\Delta_{xy}U_t\) is \(C^1\) in this
space, with derivative \(\Delta_{xy}F_t\). For \(p>2\), H\"older's inequality gives \eqref{lulu1}. For \(p=2\), the left-hand side in \eqref{lulu1} is simply
\(\|\Delta_{xy}F_t\|_{\mathrm L^2(Q_D,\nu)}^2\).
The same argument gives finite weighted energy for
every \(\varphi\in\Wzero(D)\), so
\(\mathcal B_{U_t}(F_t,\varphi)\) is well defined. For \(p\ge2\), the map \(A\mapsto\Jp(A)\) is continuously
differentiable from \(\mathrm L^p(Q_D,\nu)\) to
\(\mathrm L^{p'}(Q_D,\nu)\), with derivative
\((p-1)|A|^{p-2}B\) in direction \(B\).
For \(p>2\), continuity of this derivative follows from
continuity of \(A\mapsto|A|^{p-2}\) into
\(\mathrm L^{p/(p-2)}(Q_D,\nu)\) and H\"older's
inequality. For \(p=2\), the map is linear.
Differentiating the weak equation against a fixed
\(\varphi\in\Wzero(D)\) therefore gives
\[
 \mathcal B_{U_t}(F_t,\varphi)
 =
 \frac{p-1}{2}\iint_{Q_D}|\Delta_{xy}U_t|^{p-2}
 \Delta_{xy}F_t\,\Delta_{xy}\varphi\,\mathrm d\nu
 =0,
\]
which is the asserted identity.\end{proof}

Our argument does not require \(C^1\) dependence in
\(\Vs(D)\). The following energy estimate is used instead
in the differentiation argument of
Proposition~\ref{prop:ordered-derivative}.

\begin{prop}\label{prop:path-holder}
Let \(g_0,g_1\in\Vs(D)\), and set
\[
 g_t=(1-t)g_0+tg_1,\qquad 0\le t\le1.
\]
For each \(t\), let \(U_t\) be the unique variational
solution supplied by Theorem~\ref{thm:dirichlet}, satisfying
\[
 \Fsp U_t=0\quad\text{weakly in }D,
 \qquad
 U_t-g_t\in\Wzero(D).
\]
If \(p\ge2\), then
\begin{equation}\label{eq:path-holder}
 \left(\iint_{Q_D}
 \frac{|\Delta_{xy}(U_t-U_\tau)|^p}{|x-y|^{n+sp}}
 \dxdy\right)^{1/p}
 \le C|t-\tau|^{1/(p-1)}.
\end{equation}
The constant depends on \(p\) and the \(\Vs(D)\) norms of
\(g_0,g_1\). The same estimate holds for the \(\mathrm L^p(D)\) norm,
with an additional dependence on a containing ball for \(D\).
\end{prop}

\begin{proof}
Again, equip \(Q_D\) with the measure
\(\mathrm d\nu=|x-y|^{-n-sp}\,\mathrm d y\,\mathrm d x\), and write $\Delta f(x,y)=f(x)-f(y)$ on $Q_D$. All pair-space norms below are taken with respect to \(\nu\).
Set
\[
 M_{\mathrm{en}}=1+\|\Delta g_0\|_{\mathrm L^p(\nu)}
     +\|\Delta g_1\|_{\mathrm L^p(\nu)}.
\]
Since \(g_t\) is an admissible competitor for \(U_t\),
the minimizing property and the triangle inequality give
\[
 \|\Delta U_t\|_{\mathrm L^p(\nu)}
 \le\|\Delta g_t\|_{\mathrm L^p(\nu)}
 \le M_{\mathrm{en}}
 \qquad(0\le t\le1).
\]
Fix \(t,\tau\in[0,1]\), and put
\[Z=U_t-U_\tau,\qquad G=g_t-g_\tau=(t-\tau)(g_1-g_0).\] The prescribed variational exterior values imply
\(Z-G\in\Wzero(D)\). Subtracting the weak equations
and testing with \(Z-G\), we obtain
\[
 \iint_{Q_D}
 \bigl(\Jp(\Delta U_t)-\Jp(\Delta U_\tau)\bigr)
 \Delta Z\,\mathrm d\nu
 =
 \iint_{Q_D}
 \bigl(\Jp(\Delta U_t)-\Jp(\Delta U_\tau)\bigr)
 \Delta G\,\mathrm d\nu.
\]
For \(p\ge2\), the scalar inequalities
\[
 \bigl(\Jp(A)-\Jp(B)\bigr)(A-B)
 \ge c_p|A-B|^p
\]
and
\[
 |\Jp(A)-\Jp(B)|
 \le C_p(|A|+|B|)^{p-2}|A-B|
\]
hold for all \(A,B\in\R\), with the usual convention
when \(p=2\). Applying these inequalities and then
H\"older's inequality gives
\begin{align*}
 c_p\|\Delta Z\|_{\mathrm L^p(\nu)}^p
 &\le C_p
 \bigl(
 \|\Delta U_t\|_{\mathrm L^p(\nu)}
 +\|\Delta U_\tau\|_{\mathrm L^p(\nu)}
 \bigr)^{p-2}
 \|\Delta Z\|_{\mathrm L^p(\nu)}
 \|\Delta G\|_{\mathrm L^p(\nu)}\\
 &\le C_p M_{\mathrm{en}}^{p-2}
 \|\Delta Z\|_{\mathrm L^p(\nu)}
 \|\Delta G\|_{\mathrm L^p(\nu)}.
\end{align*}
For \(p>2\), the H\"older exponents used here are
\(p/(p-2),p,p\). For \(p=2\), the estimate is simply
Cauchy-Schwarz. If \(\|\Delta Z\|_{\mathrm L^p(\nu)}=0\), the desired
energy estimate is immediate. Otherwise, division gives
\[
 \|\Delta Z\|_{\mathrm L^p(\nu)}^{p-1}
 \le C_p M_{\mathrm{en}}^{p-2}|t-\tau|
       \|\Delta(g_1-g_0)\|_{\mathrm L^p(\nu)}.
\]
Taking the power \(1/(p-1)\) proves
\eqref{eq:path-holder}, with the stated dependence
of the constant. For the \(\mathrm L^p(D)\) estimate, choose a ball \(B_R(z_0)\)
containing \(D\). Since \(Z-G\in\Wzero(D)\), the
fractional Poincar\'e inequality yields
\begin{align*}
 \|Z\|_{\mathrm L^p(D)}
 &\le \|Z-G\|_{\mathrm L^p(D)}
       +\|G\|_{\mathrm L^p(D)}\\
 &\le CR^s
 \bigl(
 \|\Delta Z\|_{\mathrm L^p(\nu)}
 +\|\Delta G\|_{\mathrm L^p(\nu)}
 \bigr)
 +\|G\|_{\mathrm L^p(D)}.
\end{align*}
Here \(C\) depends only on \(n,s,p\).
Both norms of \(G\) are bounded by a constant times
\(|t-\tau|\). Since $|t-\tau|\le |t-\tau|^{1/(p-1)}$ for $0\le t,\tau\le1$, the energy estimate proves the final assertion.
\end{proof}

\begin{rem}
For \(p>2\), the exponent \(1/(p-1)\) in
\eqref{eq:path-holder} is below \(1\).
This estimate alone does not give absolute continuity of the
parameter dependence. For ordered data, comparison with scalar
multiples gives a stronger pointwise estimate
\eqref{eq:ordered-parameter-bound}. This estimate, rather than
monotonicity alone, supplies the required absolute continuity.
\end{rem}

\subsection{Directional coercivity and boundary scaling}

The segment estimates in the proof of
Proposition~\ref{prop:whitney-pair} give the following
general kernel bounds. Their interior order differs
from the boundary scaling of the half-space model below.

\begin{prop}\label{prop:cone-kernel}
Assume \eqref{eq:range}, and put
\begin{equation}\label{eq:sigma-order}
 \sigma_* = sp-p+2=2-p(1-s)\in(0,2).
\end{equation}
Let \(B\) be a ball, \(a>0\), and \(e\) a unit vector.
If \(u\in C^1(B)\) satisfies
\[
 |\nabla u-ae|\le a/4\quad\text{in }B,
\]
then
\[
 K_u(x,y)\le C_p a^{p-2}|x-y|^{-n-\sigma_*}
 \qquad(x,y\in B),
\]
and the reverse inequality holds with a constant \(c_p>0\)
whenever \(|(x-y)\cdot e|\ge|x-y|/2\).
\end{prop}

\begin{proof}
The segment argument gives
\(|u(x)-u(y)|\le5a|x-y|/4\) for all pairs and
\(|u(x)-u(y)|\ge a|x-y|/4\) in the indicated cone.
Insert these bounds in \eqref{eq:Ku}, using
\(-n-sp+p-2=-n-\sigma_*\).
\end{proof}

The bounds in Proposition \ref{prop:cone-kernel} give interior coercivity of the type used in
\cite[Section~7]{GJS}. Near the boundary, the gradient size,
available ball, and preferred direction depend on position, and the boundary scaling must therefore be considered separately.

\begin{prop}\label{prop:model-kernel}
Let \(p\ge2\), \(0<s<1\), and \(h=(x_n)_+^s\).
Then
\begin{equation}\label{eq:model-kernel}
 K_h(x,y)=(p-1)
 \frac{|(x_n)_+^s-(y_n)_+^s|^{p-2}}{|x-y|^{n+sp}}
\end{equation}
and
\begin{equation}\label{eq:model-scaling}
 K_h(rx,ry)=r^{-n-2s}K_h(x,y),\qquad r>0.
\end{equation}
For \(x\in\HH\),
\begin{equation}\label{eq:killing}
 \int_{\HH^c}K_h(x,y)\d y
 =\kappa_{n,s,p}x_n^{-2s},
\end{equation}
where \(\kappa_{n,s,p}>0\).
For \(p>2\), \(K_h(x,y)=0\) whenever \(x_n=y_n>0\).
\end{prop}

\begin{proof}
Formula \eqref{eq:model-kernel} follows from \eqref{eq:Ku}.
The numerator scales with degree \(s(p-2)\), so the full
degree is \(s(p-2)-n-sp=-n-2s\).
For \(y\in\HH^c\), the numerator equals
\(x_n^{s(p-2)}\).
Translation in the tangential variables and scaling by
\(x_n\) show that
\[
 \int_{\HH^c}|x-y|^{-n-sp}\d y
 =x_n^{-sp}\int_{\HH^c}|e_n-y|^{-n-sp}\d y.
\]
The last integral is finite and positive.
This proves \eqref{eq:killing}. The final assertion is immediate.
\end{proof}

The dilation law in \eqref{eq:model-scaling} reflects boundary
scaling of order \(2s\), whereas the small-scale interior order
at a point where the background gradient does not vanish is
\(\sigma_*=sp-p+2\). Both scaling properties must be accounted
for in the boundary theory.

Even when $f=0$ on $\HH^c$, the exterior interaction
contributes to the equation. Indeed, let \(h(x)=(x_n)_+^s\), and suppose that
\(f=0\) on \(\HH^c\). Then, by
Proposition~\ref{prop:model-kernel},
\[
 \mathcal L_h f(x)
 =
 \operatorname{PV}\int_\HH
 (f(x)-f(y))K_h(x,y)\,\mathrm d y
 +\kappa_{n,s,p}x_n^{-2s}f(x),
\]
whenever the principal value is well defined.
The same decomposition holds in the weak formulation
for functions in the energy class.  Indeed, symmetry of \(K_h\) and the
zero exterior values give, for every
\(\varphi\in C_0^\infty(\HH)\), extended by zero to \(\HH^c\),
\begin{align*}
 \frac12\iint_{\R^{2n}}
 \Delta_{xy}f\,\Delta_{xy}\varphi\,K_h(x,y)
 \,\mathrm d y\,\mathrm d x
 &=
 \frac12\iint_{\HH\times\HH}
 \Delta_{xy}f\,\Delta_{xy}\varphi\,K_h(x,y)
 \,\mathrm d y\,\mathrm d x+\kappa_{n,s,p}\int_\HH
 x_n^{-2s}f(x)\varphi(x)\,\mathrm d x.
\end{align*}
Hence, when the equation is written using only interior
pairs, the exterior interaction appears as a zeroth order
term with positive coefficient
\(\kappa_{n,s,p}x_n^{-2s}\).  The zeroth order term is essential. In particular,
$\mathcal L_hh=0$ expresses cancellation between
the interior and exterior contributions.
In the ground-state transformation below, we write
$f=hq$ and combine the weak equations for $f$ and $h$.
Their exterior contributions then cancel, yielding
an equation for $q$ with interaction kernel
$h(x)h(y)K_h(x,y)$ on $\HH\times\HH$.
Thus the absence of an exterior term in the quotient
equation follows from the transformation.

\subsection{An identity for quotients}

There is a useful exact identity for the quotient by a positive
solution of the linearized equation.
It gives the transformed equation used to prove boundary
comparison in the next section.

\begin{prop}\label{prop:ground-state}
Let \(D\subset\R^n\) be an open set, and let
\(K:\R^n\times\R^n\to[0,\infty]\) be a measurable
symmetric kernel. Suppose that \(U>0\) in \(D\), \(U\ge0\) in \(\R^n\),
and that both \(U\) and \(f\) solve the linear equation
with kernel \(K\) weakly in \(D\). Assume that $f=0$ almost everywhere on $\{U=0\}$,
and define
\[
 S=\{U>0\},\qquad q={f}/{U}\quad\text{on }S.
\]
Let $\phi$ be compactly supported in $D$, and suppose
that $U\phi$ and $f\phi$ are admissible test functions
in the weak equations for $f$ and $U$, respectively.
Then
\begin{equation}\label{eq:quotient-form}
 \frac12\iint_{S\times S}
 U(x)U(y)(q(x)-q(y))(\phi(x)-\phi(y))
 K(x,y)\dxdy=0.
\end{equation}
\end{prop}

\begin{proof}
Subtract the weak equation for \(U\), tested with \(f\phi\),
from the weak equation for \(f\), tested with \(U\phi\).
The pointwise identity
\begin{align*}
 &(f(x)-f(y))(U(x)\phi(x)-U(y)\phi(y))-(U(x)-U(y))(f(x)\phi(x)-f(y)\phi(y))\\
 &=U(x)U(y)(q(x)-q(y))(\phi(x)-\phi(y))
\end{align*}
proves \eqref{eq:quotient-form}.
Values of \(q\) on \(\{U=0\}\) are immaterial.
\end{proof}

For \(K=K_U\), Proposition~\ref{prop:homogeneity} shows
that \(U\) satisfies \(\mathcal L_U U=0\).
The corresponding ground-state transform has kernel
\[
 \widetilde K_U(x,y)=
 (p-1)\frac{U(x)U(y)|U(x)-U(y)|^{p-2}}
 {|x-y|^{n+sp}}.
\]
For the half-space profile \(U(x)=(x_n)_+^s\), this kernel
is homogeneous of degree \(-n\) under simultaneous dilation
of its two variables:
\[
 \widetilde K_U(\lambda x,\lambda y)
 =\lambda^{-n}\widetilde K_U(x,y),
 \qquad \lambda>0.
\]
The scaling of the associated operator also depends on
the reference measure, which must therefore be specified.

We will now prove a boundary oscillation theorem for this
half-space model. Before doing so, we identify the exterior
contribution that vanishes under the additional zero-set
hypothesis in Proposition~\ref{prop:ground-state}.

\begin{prop}\label{prop:quotient-source} Let \(D\subset\R^n\) be an open set. Let  $U>0$ in $D$, $U\ge0$ in $\R^n$, and suppose
that $\mathcal L_U U=\mathcal L_U F=0$ weakly in $D$.
Put
\[
 S=\{U>0\},\qquad Z=\{U=0\},\qquad q={F}/{U}
 \quad\text{on }S.
\]
Let $\phi$ be compactly supported in $D$, and suppose
that $U\phi$ and $F\phi$ are admissible test functions
in the weak equations for $F$ and $U$, respectively.
Assume also that the integrals below are absolutely
convergent. Then
\begin{align}
 &\frac12\iint_{S\times S}
 U(x)U(y)(q(x)-q(y))(\phi(x)-\phi(y))
 K_U(x,y)\dxdy \notag\\
 &\qquad=(p-1)\int_S U(x)^{p-1}\phi(x)
       \left(\int_Z\frac{F(y)}{|x-y|^{n+sp}}\d y\right)\d x.
 \label{eq:quotient-source}
\end{align}
In particular, the right-hand side is nonnegative
when $F\ge0$ on $Z$ and $\phi\ge0$ in $D$.
\end{prop}

\begin{proof}
Make the same subtraction as in
Proposition~\ref{prop:ground-state}. For \(x\in S\) and
\(y\in Z\), its integrand equals
\(-U(x)F(y)\phi(x)K_U(x,y)\).
The two orientations of such pairs cancel the factor \(1/2\).
Since \(K_U(x,y)=(p-1)U(x)^{p-2}|x-y|^{-n-sp}\) on
these pairs, moving this contribution to the right gives
\eqref{eq:quotient-source}.
Pairs in \(Z\times Z\) make no contribution.
\end{proof}

The source in \eqref{eq:quotient-source} vanishes for
parameter derivatives of positive Dirichlet families
with nonnegative, affinely interpolated exterior data. Indeed, suppose that \(U_0,U_1\ge0\), that \(U_t>0\)
in \(D\) for \(0<t<1\), and that
\[
 U_t=(1-t)U_0+tU_1
 \quad\text{almost everywhere on }D^c.
\]
At any \(t\in(0,1)\) for which the parameter derivative
\(F_t=\partial_tU_t\) exists, its exterior values are
\[
 F_t=U_1-U_0
 \quad\text{almost everywhere on }D^c.
\]
On \(D^c\cap\{U_t=0\}\), nonnegativity and
\(0<t<1\) imply \(U_0=U_1=0\), and hence \(F_t=0\).
Since \(U_t>0\) in \(D\), it follows that \(F_t\)
vanishes almost everywhere on the entire zero set
of \(U_t\). Thus the right-hand side of
\eqref{eq:quotient-source} is zero for \(U=U_t\)
and \(F=F_t\).

For a general solution \(F\) of the linearized equation,
a local exterior zero condition need not imply that
\(F\) vanishes on the entire zero set of the background
\(U\). Such a local condition therefore does not, by
itself, eliminate the right-hand side of
\eqref{eq:quotient-source}.

\subsection{The weighted half-space form}

Assume \eqref{eq:range}, set \(X_h=\overline{\HH}\), and use the
Euclidean metric on \(X_h\). Write
\(\mathbb B_r(x)=X_h\cap B_r(x)\). Define
\begin{align}
 \d\mu_h(x)&=x_n^s\d x,
 &V_h(x,r)&=\mu_h(\mathbb B_r(x)),\notag\\
 \Psi_h(x,r)&=r^{\sigma_*}(x_n+r)^{s-\sigma_*},
 &\beta_0&=\min\{s,\sigma_*\},
 \label{eq:model-measure-scale}
\end{align}
where \(\mu_h\) gives zero mass to \(\partial\HH\).  The triple \((X_h,d,\mu_h)\), where \(d(x,y)=|x-y|\),
is a complete metric measure space. The measure \(\mu_h\) is Radon and has full support
on \(X_h\), and every closed metric ball is compact. For measurable \(f,g\) of finite energy, put
\begin{equation}\label{eq:model-weighted-form}
 \mathscr E_h(f,g)=\frac12\iint_{\HH\times\HH}
 (f(x)-f(y))(g(x)-g(y))K_h(x,y)
 \d\mu_h(x)\d\mu_h(y).
\end{equation}
By Proposition~\ref{prop:ground-state}, if
\(\mathcal L_h f=0\) in \(\HH\), \(f=0\) on \(\HH^c\),
and the products required there are admissible test functions, then
\(q=f/h\) satisfies \(\mathscr E_h(q,\psi)=0\).
Indeed, \(\mathcal L_hh=0\), and the transformed interaction
measure is
\[
 h(x)h(y)K_h(x,y)\,\mathrm d x\,\mathrm d y
 =K_h(x,y)\,\mathrm d\mu_h(x)\,\mathrm d\mu_h(y).
\]
Thus the kernel relative to \(\mu_h\) is \(K_h\).
For \(G\subset X_h\), we write \(\mathscr E_{h,G}\)
when both integration variables are restricted to \(G\).

\begin{lem}\label{lem:model-weighted-estimates}
There are positive constants \(c,C\), depending only on
\(n,s,p\), such that, for \(x\in X_h\) and \(r>0\),
\begin{equation}\label{eq:model-volume}
 c r^n(x_n+r)^s\le V_h(x,r)\le C r^n(x_n+r)^s.
\end{equation}
The measure \(\mu_h\) is doubling and reverse doubling.
For \(0<r\le R\),
\begin{equation}\label{eq:model-scale-growth}
 \left(\frac Rr\right)^{\min\{s,\sigma_*\}}
 \le\frac{\Psi_h(x,R)}{\Psi_h(x,r)}
 \le\left(\frac Rr\right)^{\max\{s,\sigma_*\}}.
\end{equation}
Moreover, \(\Psi_h(x,r)\asymp\Psi_h(z,r)\) if
\(|x-z|\le r\). For distinct \(x,y\in\HH\), with
\(\ell=|x-y|\),
\begin{equation}\label{eq:model-jump-upper}
 K_h(x,y)\le
 \frac{C}{V_h(x,\ell)\Psi_h(x,\ell)}.
\end{equation}
For \(x\in\HH\) and \(r>0\),
\begin{equation}\label{eq:model-jump-moment}
\frac 1{r^2}\int_{\HH\cap \mathbb B_r(x)}{|x-y|^2}
 K_h(x,y)\d\mu_h(y)+\int_{\HH\setminus\mathbb B_r(x)}
 K_h(x,y)\d\mu_h(y)\le\frac{C}{\Psi_h(x,r)}.
\end{equation}
If \(f\in \mathrm L^2(\mathbb B_{4r}(x),\mu_h)\) has finite energy there, then
\begin{equation}\label{eq:model-poincare}
 \int_{\mathbb B_r(x)}|f-f_{x,r}|^2\d\mu_h
 \le C\Psi_h(x,r)
       \mathscr E_{h,\mathbb B_{4r}(x)}(f,f),
 \qquad
 f_{x,r}=\frac{1}{V_h(x,r)}
             \int_{\mathbb B_r(x)}f\d\mu_h.
\end{equation}
Finally, if \(x,y\in\HH\) and \(0<r\le|x-y|/2\), then
\begin{equation}\label{eq:model-ujs}
 K_h(x,y)\le\frac{C}{V_h(x,r)}
           \int_{\mathbb B_r(x)}K_h(z,y)\d\mu_h(z).
\end{equation}
\end{lem}

\begin{proof}
Put \(t=x_n\). The upper volume bound follows from
\(z_n\le t+r\) on \(\mathbb B_r(x)\).
The ball \(B_{r/8}(x+(r/2)e_n)\) is contained in
\(\mathbb B_r(x)\), and \(z_n\asymp t+r\) there.
It gives the lower bound. Doubling and reverse doubling follow,
the latter with lower exponent \(n\).
The logarithmic derivative of \(\Psi_h(x,r)\) with respect to
\(r\) is
\[
 \sigma_*+(s-\sigma_*)r/{(t+r)}.
\]
It lies between \(s\) and \(\sigma_*\), proving
\eqref{eq:model-scale-growth}. The comparison for nearby centers
follows from \(x_n+r\asymp z_n+r\).

For \(x,y\in\HH\),
\[
 |x_n^s-y_n^s|\le C_s\ell(t+\ell)^{s-1}.
\]
When \(\ell<t/2\), this follows from the mean value theorem.
Otherwise it follows from the \(s\)-H\"older continuity of the
power function. Hence, putting
\(a=(1-s)(p-2)=2s-\sigma_*\),
\[
 K_h(x,y)\le C\ell^{-n-\sigma_*}(t+\ell)^{-a}
       \asymp\frac{1}{V_h(x,\ell)\Psi_h(x,\ell)}.
\]
This also covers \(p=2\) with the stated convention for the
zeroth power.
By volume doubling, integration over an annulus of radius
\(\rho\) is bounded by \(C/\Psi_h(x,\rho)\).
Summing outer annuli uses \(\beta_0>0\).
Summing inner annuli with the factor \((\rho/r)^2\) uses
\(\max\{s,\sigma_*\}<2\). This proves
\eqref{eq:model-jump-moment}.

To prove \eqref{eq:model-poincare}, fix a center \(x_0\in X_h\),
write \(t=(x_0)_n\), and set
\[
 B=\mathbb B_r(x_0),\qquad
 G=B_{r/4}(x_0+3re_n).
\]
For \(x\in B\cap\HH\) and \(y\in G\), we have
\(G\subset\mathbb B_{4r}(x_0)\), \(|x-y|\asymp r\), and
\[
 y_n^s\asymp(t+r)^s,\qquad
 y_n^s-x_n^s\ge c_s r(t+r)^{s-1}.
\]
Consequently,
\[
 K_h(x,y)y_n^s
 \ge c r^{-n-\sigma_*}(t+r)^{\sigma_*-s}.
\]
Let \(f_G=|G|^{-1}\int_G f\d y\), with the ordinary
Lebesgue mean on \(G\). Jensen's inequality gives
\[
 \mathscr E_{h,\mathbb B_{4r}(x_0)}(f,f)
 \ge\frac{c}{\Psi_h(x_0,r)}
      \int_B|f-f_G|^2\d\mu_h.
\]
The weighted mean on \(B\) minimizes the last squared distance
over all constants, proving the required inequality.
Truncation gives the assertion whenever the indicated energy
is finite.

It remains to prove \eqref{eq:model-ujs}. The case \(p=2\)
follows directly from \(|z-y|\asymp|x-y|\) on
\(\mathbb B_r(x)\). Suppose that \(p>2\), and set
\(\Delta=r(t+r)^{s-1}\).
Consider the two balls
\[
 G_1=B_{r/32}(x+(r/4)e_n),\qquad
 G_2=B_{r/32}(x+(3r/4)e_n).
\]
Each has \(\mu_h\)-measure at least \(cV_h(x,r)\).
The ranges of \(z_n^s\) on these balls are separated by
\(c_s\Delta\), and both lie within \(C_s\Delta\) of \(t^s\).
For any real number \(b\), either \(|t^s-b|\ge2C_s\Delta\),
in which case both balls have
\(|z_n^s-b|\ge|t^s-b|/2\), or one of the two balls has
\(|z_n^s-b|\ge c_s\Delta\ge c'_s|t^s-b|\).
Increasing \(C_s\), if necessary, makes these alternatives
exhaustive. Thus
\[
 \frac1{V_h(x,r)}\int_{\mathbb B_r(x)}
       |z_n^s-b|^{p-2}\d\mu_h(z)
 \ge c|t^s-b|^{p-2}.
\]
Take \(b=y_n^s\), and use \(|z-y|\asymp|x-y|\).
This proves \eqref{eq:model-ujs}.
\end{proof}

Consider the integral form $\mathscr E_h$ defined in
\eqref{eq:model-weighted-form}, with maximal finite-energy
domain
\[
 \mathscr D_h^{\max}
 =
 \left\{
 f\in\mathrm L^2(X_h,\mu_h):
 \mathscr E_h(f,f)<\infty
 \right\}.
\]
Equip this space with the form norm
\begin{equation}\label{form}
 \|f\|_{\mathscr E_h,1}
 =
 \bigl(\|f\|_{\mathrm L^2(X_h,\mu_h)}^2+
              \mathscr E_h(f,f)\bigr)^{1/2}.
\end{equation}
The space $\mathscr D_h^{\max}$ is complete in this
norm. Indeed, a Cauchy sequence converges in
$\mathrm L^2(X_h,\mu_h)$, and Fatou's lemma, applied
along an almost everywhere convergent subsequence,
gives convergence in energy as well.
Thus $\mathscr E_h$ is closed on $\mathscr D_h^{\max}$. The moment estimate \eqref{eq:model-jump-moment}
shows that compactly supported Lipschitz functions
on $X_h$ have finite energy.
Let $\mathscr D_h$ be their closure in the norm
\eqref{form}.  Compact support means that the support is a compact
subset of \(X_h=\overline{\HH}\), and need not lie in \(\HH\). Thus these functions may be nonzero
on the flat boundary. In particular, this definition imposes no homogeneous
boundary condition on the quotient. The restriction of $\mathscr E_h$ to $\mathscr D_h$
is closed, symmetric, and nonnegative.
For every normal contraction $T:\R\to\R$, meaning
$T(0)=0$ and $|T(a)-T(b)|\le|a-b|$, contraction
of differences gives
\[
 \mathscr E_h(T\circ f,T\circ f)
 \le \mathscr E_h(f,f).
\]
The Lipschitz core is stable under such contractions,
and closure yields the Markov property on
$\mathscr D_h$.
Moreover, compactly supported Lipschitz functions
are dense in $\mathscr D_h$ in the form norm and
uniformly dense in $C_0(X_h)$, the continuous
functions vanishing at infinity.
Consequently, \((\mathscr E_h,\mathscr D_h)\) is a
regular symmetric Dirichlet form on
\(\mathrm L^2(X_h,\mu_h)\), with the compactly supported
Lipschitz functions on \(X_h\) as a core, see
\cite[Section~1.1]{FOT}.

To include quotients with an additive constant,
we also use
\[
 \mathscr D_h+\R
 =\{f+c:f\in\mathscr D_h,\ c\in\R\}.
\]
Since the energy depends only on differences,
the bilinear form extends to this class by
\[
 \mathscr E_h(f+c,g+d)=\mathscr E_h(f,g),
 \qquad f,g\in\mathscr D_h,\quad c,d\in\R.
\]
This extension does not require the shifted functions
to belong to $\mathrm L^2(X_h,\mu_h)$.

For a compact set $K\subset X_h$, possibly intersecting
$\partial\HH$, define
\begin{equation}\label{eq:model-capacity}
 \operatorname{Cap}_h(K)
 =
 \inf\left\{
 \|\eta\|_{\mathrm L^2(X_h,\mu_h)}^2
 +\mathscr E_h(\eta,\eta):
 \eta\in\mathscr D_h,\quad
 \eta\ge1\ \mu_h\text{-almost everywhere near }K
 \right\}.
\end{equation}
Here ``near $K$'' means on a relatively open
neighborhood of $K$ in $X_h$.
This is the $1$-capacity associated with
$(\mathscr E_h,\mathscr D_h)$, which we call
$\mathscr E_h$-capacity below.
Thus $K$ has zero capacity precisely when there exist
$\eta_j\in\mathscr D_h$, with $\eta_j\ge1$
$\mu_h$-almost everywhere on neighborhoods of $K$,
such that
\[
 \|\eta_j\|_{\mathrm L^2(X_h,\mu_h)}^2
 +\mathscr E_h(\eta_j,\eta_j)\longrightarrow0.
\]
Replacing \(\eta_j\) by
\(T\circ\eta_j\), where \(T(t)=\min\{1,\max\{0,t\}\}\),
does not increase either the \(\mathrm L^2\)-norm or
the energy, by the Markov property.
Thus these functions may be chosen to satisfy
\(0\le\eta_j\le1\), with \(\eta_j=1\)
\(\mu_h\)-almost everywhere on the corresponding
neighborhoods of \(K\).

The next lemma shows that every compact subset of
$\partial\HH$ has zero $\mathscr E_h$-capacity.
The resulting cutoff functions allow compactly
supported Lipschitz test functions on $X_h$ to be approximated
in the form norm by test functions supported inside $\HH$.
The quotient equation therefore remains valid for
test functions whose supports meet the flat boundary.
This permits the subsequent weighted Harnack and
oscillation arguments on relative balls centered
at boundary points of $X_h=\overline{\HH}$.

\begin{lem}\label{lem:model-boundary-polar}
Every compact subset of \(\partial\HH\) has zero
\(\mathscr E_h\)-capacity. Furthermore, if \(q\in\mathscr D_h+\R\) and
\[
 \mathscr E_h(q,\phi)=0
\]
for every compactly supported Lipschitz \(\phi\) in
\(\HH\cap B_R(w)\), where \(w\in\partial\HH\), then the
same identity holds for every such \(\phi\) on \(X_h\) with
support in \(\mathbb B_R(w)\).
\end{lem}

\begin{proof}
Fix a compact boundary set, choose a compactly supported
Lipschitz function \(\chi\) which is one near that set, and
let $\zeta_\varepsilon(x)=\chi(x)\theta(x_n/\varepsilon)$, where \(\theta\) is Lipschitz, \(0\le\theta\le1\),
\(\theta=1\) on \([0,1]\), and
\(\theta=0\) on \([2,\infty)\).
The support of \(\zeta_\varepsilon\) has \(\mu_h\)-measure
at most \(C\varepsilon^{s+1}\), and
\[
 |\zeta_\varepsilon(x)-\zeta_\varepsilon(y)|
 \le C\min\{1,|x-y|/\varepsilon\}.
\]
In the energy integral, a nonzero difference has at least one
endpoint in this support. At that endpoint \(x_n\le2\varepsilon\),
so \eqref{eq:model-jump-moment} gives
\[
 \mathscr E_h(\zeta_\varepsilon,\zeta_\varepsilon)
       \le C\varepsilon,\qquad
 \|\zeta_\varepsilon\|_{\mathrm L^2(\mu_h)}^2
       \le C\varepsilon^{s+1}.
\]
This proves the capacity assertion. For the second assertion, let \(\phi\) be a Lipschitz function
on \(X_h\) whose support is a compact subset of
\(\mathbb B_R(w)\). Choose \(\chi\) with compact support in
\(\mathbb B_R(w)\), equal to one near
\(\partial\HH\cap\operatorname{supp}\phi\), and use the
cutoffs \(\zeta_\varepsilon\) constructed above.
The function $\phi_\varepsilon=\phi(1-\zeta_\varepsilon)$ is Lipschitz and has compact support in
\(\HH\cap B_R(w)\). Hence the assumed interior weak
identity gives $\mathscr E_h(q,\phi_\varepsilon)=0$. The product estimate for the quadratic form gives
\begin{equation}\label{Mestauu}
 \mathscr E_h(\phi\zeta_\varepsilon,
             \phi\zeta_\varepsilon)
 \le
 2\|\phi\|_\infty^2
 \mathscr E_h(\zeta_\varepsilon,\zeta_\varepsilon)
 +\int_{X_h}\zeta_\varepsilon^2
       \Gamma_h(\phi)\d\mu_h,
\end{equation}
where
\[
 \Gamma_h(\phi)(x)
 =\int_{X_h}|\phi(x)-\phi(y)|^2
       K_h(x,y)\d\mu_h(y).
\]
By \eqref{eq:model-jump-moment}, this function is bounded
on the fixed compact support of \(\chi\).
The right-hand side in \eqref{Mestauu} therefore tends to zero.
Moreover, $\|\phi\zeta_\varepsilon\|_{\mathrm L^2(X_h,\mu_h)}
 \le
 \|\phi\|_\infty
 \|\zeta_\varepsilon\|_{\mathrm L^2(X_h,\mu_h)}
 \longrightarrow0$.  Thus \(\phi_\varepsilon\to\phi\) in the form norm. Write \(q=q_0+c\), where \(q_0\in\mathscr D_h\) and
\(c\in\R\). Since the energy depends only on differences, $\mathscr E_h(q,q)=\mathscr E_h(q_0,q_0)<\infty$. Using the interior weak identity, bilinearity, and
Cauchy-Schwarz, we obtain \[
|\mathscr E_h(q,\phi)|
 =|\mathscr E_h(q,\phi-\phi_\varepsilon)|=|\mathscr E_h(q,\phi\zeta_\varepsilon)|\le
 \mathscr E_h(q,q)^{1/2}
 \mathscr E_h(\phi\zeta_\varepsilon,
             \phi\zeta_\varepsilon)^{1/2}
 \longrightarrow0.
\] Consequently, \(\mathscr E_h(q,\phi)=0\), proving the
second assertion.
\end{proof}

\begin{rem}\label{rem:model-local-form-domain}
For a relatively compact open set \(O\subset X_h\), define
\(\mathscr D_{h,0}(O)\) as the closure, in the form norm
\eqref{form}, of Lipschitz functions on \(X_h\) compactly
supported in \(O\).
For any such cutoff \(\eta\), the moment estimate bounds
\(\Gamma_h(\eta)\) on a fixed compact neighborhood of its
support, while \eqref{eq:model-jump-upper} bounds it outside
that neighborhood.
Thus \(\Gamma_h(\eta)\) is bounded on \(X_h\), and the
product estimate makes multiplication by \(\eta\)
continuous in the form norm. If \(f\in\mathscr D_h\) vanishes outside a compact subset
of \(O\), choose \(\eta=1\) near that subset.
Multiplying Lipschitz-core approximants of \(f\) by \(\eta\)
gives \(f\in\mathscr D_{h,0}(O)\).
Similarly, \(\eta v\in\mathscr D_{h,0}(O)\) for every
\(v\in\mathscr D_h+\R\) and every Lipschitz cutoff
\(\eta\) compactly supported in \(O\). For \(q\) as in Lemma~\ref{lem:model-boundary-polar}, take
\(O=\mathbb B_R(w)\) and approximate
\(\phi\in\mathscr D_{h,0}(O)\) in the form norm by
Lipschitz functions \(\phi_j\) compactly supported in \(O\).
The lemma gives \(\mathscr E_h(q,\phi_j)=0\), and
Cauchy-Schwarz yields
\[
 |\mathscr E_h(q,\phi)|
 =|\mathscr E_h(q,\phi-\phi_j)|\longrightarrow0.
\]
Hence the weak identity holds for every
\(\phi\in\mathscr D_{h,0}(O)\).
\end{rem}

\begin{rem}
Lemma~\ref{lem:model-boundary-polar} concerns the weak
equation on the relative ball
\(\mathbb B_R(w)=X_h\cap B_R(w)\).
Tests may meet the flat boundary, but \(q\) is not extended
into the lower half-space, and no equation is asserted on
the full Euclidean ball \(B_R(w)\).
All interactions remain those of \(\mathscr E_h\) on \(X_h\).
\end{rem}

We next record an approximation that permits us to apply
the bounded-function Harnack argument without imposing a
global bound in the model theorem.

\begin{lem}\label{lem:model-harmonic-approximation}
Let \(z\in X_h\), \(a>0\), and \(O=\mathbb B_a(z)\). Suppose that
\(q\in\mathscr D_h+\R\) is nonnegative on \(X_h\) and
\[
 \mathscr E_h(q,\phi)=0
 \qquad(\phi\in\mathscr D_{h,0}(O)).
\]
For all sufficiently large integers \(N\), there are
\(q_N\in\mathscr D_h+\R\), with \(0\le q_N\le N\), such that
\[
 q_N-(q\wedge N)\in\mathscr D_{h,0}(O),\qquad
 \mathscr E_h(q_N,\phi)=0
 \quad(\phi\in\mathscr D_{h,0}(O)),
\]
and
\[
 \|q_N-q\|_{\mathrm L^2(X_h,\mu_h)}^2
 +\mathscr E_h(q_N-q,q_N-q)\longrightarrow0.
\]
\end{lem}

\begin{proof}
Reverse doubling gives a fixed \(K>1\) such that
\(V_h(z,Ka)\ge2V_h(z,a)\).
Every \(v\in\mathscr D_{h,0}(O)\) vanishes almost everywhere
outside \(O\). Hence Cauchy-Schwarz gives
\[
 \frac12\|v\|_{\mathrm L^2(X_h,\mu_h)}^2
 \le\int_{\mathbb B_{Ka}(z)}
       |v-v_{z,Ka}|^2\,\mathrm d\mu_h
 \le C\Psi_h(z,a)\mathscr E_h(v,v),
\]
where the second inequality follows from
\eqref{eq:model-poincare} and scale comparison.
All norms in this proof use \(\mu_h\).
Thus the energy norm is equivalent to the form norm on
\(\mathscr D_{h,0}(O)\).

Write \(q=c+f\), with \(f\in\mathscr D_h\).
Since \(\mu_h(X_h)=\infty\) and \(q\ge0\), we have \(c\ge0\).
Take \(N>c\), and put \(g_N=q\wedge N\).
The Markov property gives \(g_N\in\mathscr D_h+\R\), and
\(e_N=g_N-q=-(q-N)_+\) belongs to \(\mathscr D_h\).
Since \((q-N)_+\le|f|\) and \[
|(q(x)-N)_+-(q(y)-N)_+|\le|q(x)-q(y)|,
\] dominated convergence gives \(\|e_N\|_2^2+
\mathscr E_h(e_N,e_N)\to0\).
Coercivity and the Hilbert-space representation theorem
supply a unique \(v_N\in\mathscr D_{h,0}(O)\) satisfying
\[
 \mathscr E_h(v_N,\phi)=-\mathscr E_h(g_N,\phi)
 \qquad(\phi\in\mathscr D_{h,0}(O)).
\]
Set \(q_N=g_N+v_N\). It is the unique minimizer of
\(\mathscr E_h\) in \(g_N+\mathscr D_{h,0}(O)\).

To justify clipping in this affine class, let
\(T(t)=\min\{N,\max\{0,t\}\}\), and approximate \(v_N\)
in the form norm by Lipschitz functions \(v_j\) compactly
supported in \(O\).
Since \(T(g_N)=g_N\), the functions
\(T(g_N+v_j)-g_N\) belong to \(\mathscr D_h\) and vanish
outside \(\supp v_j\). Remark~\ref{rem:model-local-form-domain}
therefore places them in \(\mathscr D_{h,0}(O)\).
They are bounded in the form norm, by contraction of
differences, and converge in \(\mathrm L^2\) to
\(T(q_N)-g_N\). Weak compactness and closedness give
\(T(q_N)-g_N\in\mathscr D_{h,0}(O)\).
Clipping cannot increase the energy, so uniqueness of the
minimizer implies \(T(q_N)=q_N\), proving \(0\le q_N\le N\).

Finally, harmonicity of \(q\) gives \[
\mathscr E_h(v_N,v_N)=-\mathscr E_h(e_N,v_N),\quad\mathscr E_h(v_N,v_N)\le\mathscr E_h(e_N,e_N).
\] Moreover, orthogonality gives
\[
 \mathscr E_h(q_N-q,q_N-q)
 =\mathscr E_h(e_N,e_N)-\mathscr E_h(v_N,v_N)
 \le\mathscr E_h(e_N,e_N).
\]
Coercivity implies \(\|v_N\|_2\to0\), and
\(q_N-q=e_N+v_N\) converges to zero in the form norm.
\end{proof}

We now pass from the structural estimates to regularity of the
quotient. We use the abstract local Harnack inequality and
H\"older estimate stated in Lemma~\ref{lem:local-form-harnack}
and Corollary~\ref{cor:local-form-holder} in
Section~\ref{sec:variable}, with proofs in
Appendix~\ref{app:local-form-harnack}.
These results are formulated for general quadratic non-local
forms and use only the structural and energy hypotheses
specified there. Their proofs do not rely on the present
half-space theorem.
The estimates established above verify these hypotheses
for \(\mathscr E_h\), while the boundary test-space
approximation permits their application on relative balls
meeting the flat boundary.
This first gives the conclusions for bounded quotients. Lemma~\ref{lem:model-harmonic-approximation} then removes
the global boundedness restriction.

\begin{thm}\label{thm:linear-flat-model}
Assume \eqref{eq:range}. Let \(w\in\partial\HH\), \(r>0\),
and let \(q\in\mathscr D_h+\R\) be nonnegative on \(X_h\).
Suppose that
\[
 \mathscr E_h(q,\phi)=0
\]
for every compactly supported Lipschitz test function in
\(\HH\cap B_{4r}(w)\).
Then \(q\) has a H\"older continuous representative on
\(\mathbb B_{2r}(w)\). With \(A=w+re_n\),
\begin{align}
 C^{-1}q(A)&\le q(x)\le Cq(A),\qquad x\in\mathbb B_r(w),
       \label{eq:linear-flat-comparison}\\
 |q(x)-q(y)|&\le Cq(A)
       \left(\frac{|x-y|}{r}\right)^\alpha,\qquad x,y\in\mathbb B_r(w),
       \label{eq:linear-flat-holder}
\end{align}
where \(C\ge1\) and \(\alpha\in(0,1)\) depend only on
\(n,s,p\).
\end{thm}
\begin{proof}
Lemma~\ref{lem:model-boundary-polar} and
Remark~\ref{rem:model-local-form-domain} extend the weak
identity to all test functions in
\(\mathscr D_{h,0}(\mathbb B_{4r}(w))\). In the following we first assume that \(q\) is bounded and verify the
hypotheses of Lemma~\ref{lem:local-form-harnack} with
\(X=P=X_h\), \(J=K_h\), and \(\mu=\mu_h\).
Closed metric balls are compact, and \(\mu_h\) is a
Radon measure with full support.
Remark~\ref{rem:model-local-form-domain} supplies the
local form-domain property, while
Lemma~\ref{lem:model-weighted-estimates} supplies the
volume and scale comparisons, Poincar\'e inequality,
full jump moment bound, and kernel averaging.
Thus the Harnack inequality of
Lemma~\ref{lem:local-form-harnack} applies, with constants
depending only on \(n,s,p\). Consequently, since \(q\ge0\) on \(X_h\), its negative tail vanishes and chaining relative balls in the convex set \(X_h\) gives
\begin{equation}\label{eq:model-essential-comparison}
 Q_0:=\esssup_{\mathbb B_{3r}(w)}q
 \le C\essinf_{\mathbb B_{3r}(w)}q.
\end{equation}
The balls in the chain have radii comparable to \(r\),
with a fixed sufficiently small factor, and the balls on
which harmonicity is used lie in \(\mathbb B_{4r}(w)\).
Structural estimates hold globally, so their further
enlargements cause no restriction.
If \(Q_0=0\), the zero representative gives both conclusions.
Assume \(Q_0>0\).

For each \(z\in\mathbb B_{2r}(w)\), apply
Corollary~\ref{cor:local-form-holder} with radius \(r/4\).
The required local upper bound is at most \(Q_0\), since
\(\mathbb B_{r/2}(z)\subset\mathbb B_{3r}(w)\).
The corollary gives local H\"older representatives with
constants independent of the global bound on \(q\).
These representatives agree on overlaps by full support
of \(\mu_h\). They therefore define a representative on
\(\mathbb B_{2r}(w)\), with
\[
 |q(x)-q(y)|\le CQ_0
       \left(\frac{|x-y|}{r}\right)^\alpha
 \qquad(x,y\in\mathbb B_{2r}(w)).
\]
For \(|x-y|<r/4\), this follows by centering the local
estimate at \(x\). For larger distances, use \(0\le q\le Q_0\).
Continuity and \eqref{eq:model-essential-comparison} give
\(Q_0\le Cq(A)\) and the comparison estimate on
\(\mathbb B_r(w)\). This proves both assertions for bounded \(q\).

For a general nonnegative \(q\in\mathscr D_h+\R\),
apply Lemma~\ref{lem:model-harmonic-approximation} with
\(O=\mathbb B_{4r}(w)\). The preceding bounded argument
applies to each \(q_N\), with constants independent of \(N\).
In particular, \eqref{eq:model-essential-comparison} gives
\[
 Q_N:=\esssup_{\mathbb B_{3r}(w)}q_N
 \le C V_h(w,3r)^{-1/2}
          \|q_N\|_{\mathrm L^2(\mathbb B_{3r}(w),\mu_h)}.
\]
The right side is uniformly bounded by the local
\(\mathrm L^2\) convergence to \(q\).
The oscillation argument therefore gives uniform H\"older
bounds on \(\mathbb B_{2r}(w)\).
Compactness on smaller closed relative balls yields a
subsequence converging locally uniformly there.
The \(\mathrm L^2\) limit and full support of \(\mu_h\)
identify this continuous limit as a representative of \(q\).
In particular, \(q_N(A)\to q(A)\) along the subsequence.
Passing to the limit in both estimates for \(q_N\) proves
\eqref{eq:linear-flat-comparison} and
\eqref{eq:linear-flat-holder} in the stated class.
\end{proof}

\begin{rem}
The boundary values in \eqref{eq:linear-flat-comparison}
refer to the H\"older representative of \(q\).
If \(q(A)>0\), they are strictly positive on
\(\partial\HH\cap B_r(w)\), although \(f=hq\) vanishes
there. No homogeneous boundary condition is imposed on
\(q\). For \(f=hq\), the theorem estimates \(f/h\) whenever
Proposition~\ref{prop:ground-state} supplies its equation
and \(q\) belongs to the stated class \(\mathscr D_h+\R\).
\end{rem}

\begin{rem}
The explicit model illustrates the weighted regularity
argument with globally controlled geometry.
Section~\ref{sec:variable} proves the structural estimates
for variable backgrounds directly from their nonlinear
equation, with local geometric control and full exterior
interactions. This linear model is distinct from the
nonlinear half-space estimates of Section~\ref{sec:flat},
which enter the profile approximation in
Section~\ref{sec:transfer}.
\end{rem}

\section{Weighted estimates for variable backgrounds}
\label{sec:variable}

Theorem~\ref{thm:linear-flat-model} establishes quotient
regularity for the explicit half-space profile. We now
replace its volume, scale, and kernel calculations by
uniform estimates for a variable positive solution \(U\)
in a Reifenberg boundary patch. The weighted geometry
is controlled only in the fractional \(p\)-harmonicity
patch, while the form retains all distant interactions.
In particular, the measure \(U\,\mathrm d x\) need not
be globally doubling.

We establish local volume and scale comparisons,
Poincar\'e and jump moment estimates, and kernel averaging.
The capacity argument then justifies energy test functions whose
supports meet the boundary. These results permit the
application of Lemma~\ref{lem:local-form-harnack} and
Corollary~\ref{cor:local-form-holder} to the quotient
equation. Their proofs in
Appendix~\ref{app:local-form-harnack} use the quadratic
iteration strategy of \cite{CY} and only the stated
structural hypotheses. In particular, they do not depend on the flat
model theorem. The resulting comparison and H\"older
constants are independent of the global bound on
\(F/U\), as required for the ordered deformation argument.

\subsection{Preliminaries}\label{subprelim}

We here specify the setting and the order in which the
parameters are chosen.
Throughout this section we assume \eqref{eq:range},
and \(\Om\) is \((\delta,r_0)\)-Reifenberg flat.
First impose \(\delta\le\delta_{\mathrm{geom}}\), where
\(\delta_{\mathrm{geom}}>0\) is a fixed geometric threshold
ensuring a uniform exterior corkscrew condition.
Thus there is a fixed \(c>0\) such that, for every
\(z\in\partial\Om\) and \(0<\rho<r_0\), some point \(a\)
satisfies $B_{c\rho}(a)\subset B_\rho(z)\setminus\overline\Om$. In particular, every boundary ball contains a smaller
ball of comparable radius disjoint from the boundary.
This uniform boundary porosity yields the boundary-layer
estimate used in Lemma~\ref{lem:variable-boundary-capacity},
with an exponent \(\kappa_0>0\) controlling the decay
of the volume of thin neighborhoods of the boundary.
We fix \(\kappa_0\) at this stage.
Subsequent reductions of the flatness threshold preserve
that estimate with the same exponent and uniform constants.

Next choose \(\beta\in(0,s)\) sufficiently close to \(s\)
that
\[
 \kappa_0-p(s-\beta)>0,
 \qquad
 sp-\beta(p-1)<2.
\]
Such a choice is possible because the left-hand sides
tend to \(\kappa_0>0\) and \(s<2\), respectively,
as \(\beta\uparrow s\).
The first inequality provides the positive exponent
needed in the boundary capacity estimate.
The second keeps the relevant upper scale exponent
below \(2\), as required in the jump moment and
Lipschitz cutoff estimates.
We use this fixed \(\beta\) in every boundary decay
estimate below.
We then decrease the flatness threshold so that
Proposition~\ref{prop:boundary-normalization} and
Corollary~\ref{cor:reifenberg-gradient} apply.
These choices depend only on \(n,s,p\) and the fixed
geometric constants.

The structural estimates below have constants and require
enlargement factors depending only on \(n,s,p\) and the
fixed geometric constants.
These constants also determine the enlargements needed
in the local Harnack argument.
We first choose \(\kappa_*\ge8\) large enough to accommodate
all enlargements used in these arguments, including the
fixed sequences of enlargements and changes of center.
We then choose \(L\ge32\) sufficiently large for the
comparison balls used below, and finally choose
\(\Lambda\) larger than a sufficiently large fixed
multiple of \(L\).
All three parameters depend only on \(n,s,p\) and the
fixed geometric constants.

Let \(w\in\partial\Om\), \(R>0\), and \(\Lambda R<r_0\).
Suppose that \(U\ge0\) on \(\R^n\) is nontrivial,
belongs to \(\mathrm L^{p-1}_{sp}(\R^n)\), and satisfies
\[
 \Fsp U=0\quad\text{in }D_{\Lambda R}(w),
 \qquad
 U=0\quad\text{on }B_{\Lambda R}(w)\setminus\Om
\]
in the local weak sense of
Definition~\ref{def:local-zero}.
Here, as before, \(D_r(z)=\Om\cap B_r(z)\).
The choice of \(\Lambda\) ensures that
Proposition~\ref{prop:boundary-normalization} supplies
the normalized boundary and positive-tail estimates
on the boundary balls considered below.
It also ensures that
Corollary~\ref{cor:reifenberg-gradient} supplies the
two-sided gradient bounds and cone increment estimates
throughout \(D_{LR}(w)\).
We refer to \(B_{LR}(w)\) as the controlled patch.

For \(x\in\overline\Om\cap B_{LR}(w)\), write
\[
 \mathbb B_\rho(x)=\overline\Om\cap B_\rho(x).
\]
We call this relative ball allowed if \(\rho>0\) and $|x-w|+\kappa_*\rho<LR$, or, equivalently,
\[
 \overline{B_{\kappa_*\rho}(x)}\subset B_{LR}(w).
\]
Thus an allowed ball has sufficient room for the fixed
enlargements and auxiliary constructions used in the
local estimates.
The parameter \(L\) is chosen so that the comparison
balls in the subsequent Harnack argument satisfy
this condition.
This restriction concerns the location of the balls.
The full form and the tail estimates retain distant
interactions, including those outside \(B_{\Lambda R}(w)\).
The constants in the estimates are independent of
the distant exterior values of \(U\).

We now introduce the measure and quadratic form for
the quotient equation. We use the positive continuous
representative of \(U\) in \(D_{\Lambda R}(w)\) and set
\(U=0\) on \(B_{\Lambda R}(w)\setminus\Om\). Set
\[
 S=\{U>0\},
 \qquad
 \mathrm d\mu(x)=U(x)\,\mathrm d x
 \quad\text{on }S,
\]
and define
\begin{equation}\label{eq:variable-form}
 \mathscr E_U(f,g)
 =
 \frac12\iint_{S\times S}
 (f(x)-f(y))(g(x)-g(y))K_U(x,y)
 \,\mathrm d\mu(x)\,\mathrm d\mu(y),
\end{equation}
whenever the integral is absolutely convergent.
Thus \(K_U\) is the jump kernel relative to \(\mu\),
while the interaction density relative to Lebesgue
measure is \(U(x)U(y)K_U(x,y)\).
We write \(\mathscr E_U(f)=\mathscr E_U(f,f)\), interpreting the
diagonal energy as a nonnegative extended integral that may
equal \(+\infty\).
A subscript \(G\) indicates that both integration
variables are restricted to \(S\cap G\).

Positivity of \(U\) in \(D_{\Lambda R}(w)\) and its
local exterior zero values imply
\[
 S\cap B_{\Lambda R}(w)
 =
 \Om\cap B_{\Lambda R}(w).
\]
Within this patch we adjoin the boundary to the local
space and assign the adjoined boundary zero
\(\mu\)-measure.
For an allowed relative ball \(\mathbb B_r(x)\),
integration therefore means
\[
 \int_{\mathbb B_r(x)}f\,\mathrm d\mu
 =
 \int_{\Om\cap B_r(x)}f(y)U(y)\,\mathrm d y.
\]
Assigning the boundary zero measure does not by itself
justify test functions whose supports meet it.
For the bounded quotient solutions considered below,
the capacity and cutoff arguments will show that the
weak equation remains valid for such test functions, provided
their supports lie within the controlled patch $B_{LR}(w)$.
Thus the quotient equation can be studied on relative
balls centered at boundary points.

The set \(S\) in \eqref{eq:variable-form} remains global.
It contains every point where \(U\) is positive,
including points outside the fractional \(p\)-harmonicity
patch.
Consequently, the full form retains all distant
interactions, including those with exterior values
that are not controlled pointwise.

\subsection{Volumes, scale functions, and Poincar\'e inequalities}

For \(z\in\partial\Om\) and \(t>0\) with
\(B_{8t}(z)\subset B_{\Lambda R}(w)\), write
\(a_t(z)=U(A_t(z))\).
For an allowed boundary ball \(\mathbb B_r(z)\), put
\begin{equation}\label{eq:boundary-scale}
 \Psi(z,r)={r^{sp}}{a_r(z)^{1-p}}.
\end{equation}
Different choices of corkscrew points with the fixed
geometric constants give comparable values of \(a_r(z)\),
by an interior Harnack chain of bounded length.
Thus the scale is independent of that choice up to
uniform multiplicative constants. Its role comes from
the ratio of weighted mass to quadratic energy.
Writing \(a_r=a_r(z)\), boundary normalization gives
\(\mu(\mathbb B_r(z))\asymp a_r r^n\).
For the two interior sets constructed in
Lemma~\ref{lem:two-value-balls}, both their distance
and their separation in \(U\)-values are comparable
to \(r\) and \(a_r\), respectively, and each has
weighted measure comparable to \(a_r r^n\).
The kernel \(K_U\) on these interactions is therefore
comparable to \(a_r^{p-2}r^{-n-sp}\).
A difference of size one between the sets has energy
scale \(a_r^p r^{n-sp}\), and
\[
 \frac{\hbox{weighted mass}}{\hbox{energy scale}}
 \asymp\frac{a_r r^n}{a_r^p r^{n-sp}}
 =\Psi(z,r).
\]

For interior balls whose radii are much smaller than
the distance to the boundary,
Proposition~\ref{prop:whitney-pair}, applied with
\(u=v=U\), instead gives a quadratic-form comparison
of fractional order $\sigma_*=sp-p+2$. Together with the gradient scale \(U(x)/d_\Om(x)\),
this leads to the interior expression stated
in \eqref{eq:interior-scale} in the following lemma.

\begin{lem}\label{lem:variable-volume-scale}
Assume \eqref{eq:range}, and fix the flatness threshold
and the parameters \(\kappa_*,L,\Lambda\) as in subsection \ref{subprelim}. Let \(\Om\) be \((\delta,r_0)\)-Reifenberg flat with
\(\delta\) below this threshold.
Let \(w\in\partial\Om\), \(R>0\), and \(\Lambda R<r_0\).
Suppose that
\(U\in\mathrm L^{p-1}_{sp}(\R^n)\) is nontrivial,
globally nonnegative, and satisfies
\[
 \Fsp U=0\quad\text{in }D_{\Lambda R}(w),
 \qquad
 U=0\quad\text{on }B_{\Lambda R}(w)\setminus\Om
\]
in the sense of Definition~\ref{def:local-zero}.
Let \(\mathrm d\mu=U\,\mathrm d x\) on \(S=\{U>0\}\),
with zero measure assigned to the adjoined boundary. Then, for every allowed boundary ball \(\mathbb B_r(z)\)
and every allowed interior ball \(B_r(x)\) with
\(2r<d_\Om(x)\), respectively,
\begin{equation}\label{eq:variable-volume}
 \mu(\mathbb B_r(z))\asymp a_r(z)r^n,
 \qquad
 \mu(B_r(x))\asymp U(x)r^n.
\end{equation}
Furthermore, there are exponents \(0<d_-\le d_+<\infty\) such that
\[
 c\left(\frac{R_1}{r}\right)^{d_-}
 \le
 \frac{\mu(\mathbb B_{R_1}(x))}
      {\mu(\mathbb B_r(x))}
 \le
 C\left(\frac{R_1}{r}\right)^{d_+}
\]
whenever \(x\in\overline\Om\), \(0<r\le R_1\), and $|x-w|+\kappa_*R_1<LR$. In particular, \(\mu\) satisfies volume doubling and
reverse volume doubling within this range of centers
and radii. Also, after replacing \eqref{eq:boundary-scale} by a uniformly
comparable function, the boundary scale extends to a
scale function on all allowed relative balls.
On allowed boundary balls it satisfies
\[
 \Psi(z,r)\asymp{r^{sp}}{a_r(z)^{1-p}},
\]
and on allowed interior balls with \(r<d_\Om(x)/4\),
\begin{equation}\label{eq:interior-scale}
 \Psi(x,r)\asymp
 {d_\Om(x)^{p-2}r^{\sigma_*}}{U(x)^{1-p}}.
\end{equation}
Furthermore, there are exponents \(0<\beta_0\le\beta_1<2\) such that
\begin{equation}\label{eq:variable-weak-scale}
 c\left(\frac{R_1}{r}\right)^{\beta_0}
 \le
 \frac{\Psi(x,R_1)}{\Psi(y,r)}
 \le
 C\left(\frac{R_1}{r}\right)^{\beta_1}
\end{equation}
for \(x,y\in\overline\Om\), \(0<r\le R_1\), and
\(|x-y|\le R_1\), provided that
\[
 \max\{|x-w|,|y-w|\}+\kappa_*R_1<LR.
\]
Thus both \(\mathbb B_{R_1}(x)\) and
\(\mathbb B_{R_1}(y)\) are required to be allowed. For each fixed \(x\in\overline\Om\cap B_{LR}(w)\),
the function \(r\mapsto\Psi(x,r)\) may be chosen
continuous and strictly increasing on $0<r<{(LR-|x-w|)}/{\kappa_*}$. All constants and exponents depend only on \(n,s,p\)
and the fixed geometric constants.
\end{lem}

\begin{proof}
The fixed enlargements and changes of center used below
are covered by the choice of \(\kappa_*\).
The corresponding boundary normalization and Harnack
estimates are therefore available within the original
fractional $p$-harmonicity patch \(D_{\Lambda R}(w)\).

Proposition~\ref{prop:boundary-normalization} gives the
first comparison in \eqref{eq:variable-volume}, and
interior Harnack gives the second.
Harnack chains compare corkscrew values at comparable
radii and nearby boundary centers.
Iterating the comparison between consecutive dyadic
radii, and using \eqref{eq:boundary-normalized-decay}
for the opposite inequality, gives an exponent
\(\gamma>0\) such that
\[
 c\left(\frac{R_1}{r}\right)^\beta
 \le \frac{a_{R_1}(z)}{a_r(z)}
 \le C\left(\frac{R_1}{r}\right)^\gamma
\]
for the boundary centers and radii under consideration.
Consequently, concentric boundary balls satisfy
\[
 c\left(\frac{R_1}{r}\right)^{n+\beta}
 \le
 \frac{\mu(\mathbb B_{R_1}(z))}
      {\mu(\mathbb B_r(z))}
 \le
 C\left(\frac{R_1}{r}\right)^{n+\gamma}.
\]

To obtain the volume comparisons at an interior center
\(x\), put \(d=d_\Om(x)\) and choose
\(z_x\in\partial\Om\) with \(|x-z_x|=d\).
For \(t<d/4\), interior Harnack gives
\[
 \mu(\mathbb B_t(x))\asymp U(x)t^n.
\]
For \(t\ge d/4\), we instead have
\[
 \mu(\mathbb B_t(x))\asymp a_t(z_x)t^n.
\]
Indeed, the upper bound follows from
\(B_t(x)\subset B_{5t}(z_x)\), the boundary volume
comparison, and comparable-scale Harnack chains.
If \(t\ge2d\), the lower bound follows from
\(B_{t/2}(z_x)\subset B_t(x)\).
If \(d/4\le t<2d\), integrate over \(B_{t/8}(x)\)
and use interior Harnack together with
\(U(x)\asymp a_t(z_x)\).
These comparisons also show that the two volume
expressions agree up to uniform constants at \(t=d/4\).
Splitting an interval of radii at this transition gives
\[
 c\left(\frac{R_1}{r}\right)^n
 \le
 \frac{\mu(\mathbb B_{R_1}(x))}
      {\mu(\mathbb B_r(x))}
 \le
 C\left(\frac{R_1}{r}\right)^{n+\gamma}.
\]
Thus one may take \(d_-=n\) and \(d_+=n+\gamma\).
In particular, both volume doubling and reverse volume
doubling hold in the stated range.

The positive lower growth of the boundary scale requires
more than comparability at consecutive radii.
Fix a sufficiently large number \(\ell_*>1\), put
\[
 r_j=\ell_*^j r,
 \qquad
 b_j=a_{r_j}(z)^{p-1}r_j^{-sp},
\]
and consider indices for which \(r_j\le R_1\).
For a fixed sufficiently small \(c_0>0\), the corkscrew
balls
\[
 E_j=B_{c_0r_j}(A_{r_j}(z))
\]
lie inside the fractional $p$-harmonicity region and satisfy
\(U\asymp a_{r_j}(z)\) there.
Their distances from \(z\) lie between fixed positive
multiples of \(r_j\).
Choosing \(\ell_*\) large enough therefore makes the
balls \(E_j\) pairwise disjoint.

For every \(k\le m\), apply
Lemma~\ref{lem:positive-tail} on an interior ball
centered at \(A_{r_k}(z)\) whose radius is a fixed
small multiple of \(r_k\).
For \(j>k\), the ball \(E_j\) lies outside this ball,
and
\[
 |y-A_{r_k}(z)|\asymp r_j
 \qquad(y\in E_j).
\]
The Harnack lower bound on \(E_j\) consequently gives
\[
 \int_{E_j}
 \frac{U(y)^{p-1}}
      {|y-A_{r_k}(z)|^{n+sp}}\,\mathrm d y
 \ge c\,a_{r_j}(z)^{p-1}r_j^{-sp}
 =c b_j.
\]
The positive-tail estimate bounds the sum of these
disjoint contributions by \(Cb_k\).
Including the term \(j=k\) and increasing \(C>1\),
we obtain
\begin{equation}\label{eq:scale-sum}
 \sum_{j=k}^m b_j\le Cb_k.
\end{equation}
Set \(S_k=\sum_{j=k}^m b_j\).
Since \(b_k\ge C^{-1}S_k\),
\[
 S_{k+1}=S_k-b_k
 \le(1-C^{-1})S_k.
\]
Iteration and \eqref{eq:scale-sum} at \(k=0\) give
\[
 b_m\le C(1-C^{-1})^m b_0.
\]
As \(b_j=\Psi(z,r_j)^{-1}\), this proves power growth
of the boundary scale along the sequence \(r_j\).
More precisely, with
\[
 \theta=
 \frac{-\log(1-C^{-1})}{\log\ell_*}>0,
\]
we obtain the lower bound in
\begin{equation}\label{eq:boundary-weak-scale}
 c\left(\frac{R_1}{r}\right)^\theta
 \le\frac{\Psi(z,R_1)}{\Psi(z,r)}
 \le C\left(\frac{R_1}{r}\right)^{sp-\beta(p-1)}.
\end{equation}
For arbitrary \(R_1\ge r\), choose \(m\ge0\) so that
\(r_m\le R_1<\ell_*r_m\) and compare the scales at
\(r_m\) and \(R_1\) by comparable-scale Harnack chains.
The upper bound in \eqref{eq:boundary-weak-scale}
follows directly from boundary decay.
By the choice of \(\beta\), $0<sp-\beta(p-1)<2$.

We next extend the scale to interior centers.
For \(x\in\Om\), put \(d=d_\Om(x)\) and define
the preliminary scale
\[
 \Psi_0(x,t)=
 \begin{cases}
 d^{p-2}t^{\sigma_*}U(x)^{1-p},
     &0<t<d/4,\\[2mm]
 t^{sp}a_t(z_x)^{1-p},
     &t\ge d/4.
 \end{cases}
\]
For \(x\in\partial\Om\), use the boundary expression
\(\Psi_0(x,t)=t^{sp}a_t(x)^{1-p}\).
Only allowed radii are considered.
When the transition radius \(d/4\) lies in this range,
a Harnack chain gives
\[
 U(x)\asymp a_{d/4}(z_x).
\]
Since \(\sigma_*+p-2=sp\), the two expressions defining
\(\Psi_0\) are comparable at the transition.
If all radii under consideration are smaller than
\(d/4\), only the interior expression is needed.
Combining \eqref{eq:boundary-weak-scale} with the
interior power law and splitting at \(d/4\) when
necessary gives
\[
 c\left(\frac{R_1}{r}\right)^{\beta_0}
 \le
 \frac{\Psi_0(x,R_1)}{\Psi_0(x,r)}
 \le
 C\left(\frac{R_1}{r}\right)^{\beta_1},
\]
where
\[
 \beta_0=\min\{\theta,\sigma_*\},
 \qquad
 \beta_1=\max\{sp-\beta(p-1),\sigma_*\}.
\]
The range \eqref{eq:range} gives \(0<\sigma_*<2\),
so \(0<\beta_0\le\beta_1<2\).

For completeness, we also verify comparison of nearby
centers at the same radius.
Suppose that \(|x-y|\le t\) and both relative balls
of radius \(t\) are allowed.
If both boundary distances exceed \(8t\), then they
are comparable, interior Harnack gives
\(U(x)\asymp U(y)\), and the interior formulas yield
\(\Psi_0(x,t)\asymp\Psi_0(y,t)\).
Otherwise both boundary distances are at most \(9t\).
Each scale can then be compared with the boundary
expression at a nearby boundary point.
If an interior expression occurs, its boundary
distance is comparable to \(t\), so the same comparison
follows by Harnack.
The relevant boundary points are separated by at most
a fixed multiple of \(t\), and a bounded Harnack chain
compares their corkscrew values.
Thus
\[
 \Psi_0(x,t)\asymp\Psi_0(y,t)
 \qquad(|x-y|\le t).
\]
Applying this with \(t=R_1\), followed by the
same-center growth estimate at \(y\), proves
\eqref{eq:variable-weak-scale} for \(\Psi_0\).

The preliminary scale may be taken measurable by
making fixed choices on countable Whitney and dyadic
coverings. Such choices change it only by uniform
multiplicative constants.
Finally, define
\[
 \Psi(x,r)=\int_0^r\Psi_0(x,t)\,\frac{\mathrm d t}{t}.
\]
The positive lower growth exponent gives
\[
 \Psi_0(x,t)
 \le C\left(\frac{t}{r}\right)^{\beta_0}\Psi_0(x,r)
 \qquad(0<t\le r),
\]
so the integral is finite and bounded above by
\(C\Psi_0(x,r)\).
Integrating over \(r/2<t<r\) and using comparable-scale
bounds gives the reverse inequality.
Hence \(\Psi\asymp\Psi_0\).
For each fixed center, this integral is continuous
and strictly increasing throughout the allowed
range of radii, and all the asserted comparisons
are preserved.
In particular, \eqref{eq:boundary-scale} holds up to
uniform comparability after this replacement.
\end{proof}

\begin{rem} Estimate \eqref{eq:scale-sum} uses the equation through
the positive-tail bound. Iterating this summability
estimate gives the positive lower growth exponent
for \(\Psi\), without assuming the sharper estimate
in \eqref{eq:growth-target}.
\end{rem}

\begin{lem}\label{lem:two-value-balls}
Assume the hypotheses of
Lemma~\ref{lem:variable-volume-scale}.
Let \(z\in\partial\Om\) and \(r>0\) satisfy $|z-w|+\kappa_*r<LR$, so that \(\mathbb B_r(z)\) is allowed. Then there are two disjoint Euclidean balls
\(G_0,G_1\subset D_{2r}(z)\), each of radius \(c_0r\),
such that
\begin{equation}\label{eq:two-value-balls}
 \begin{gathered}
 U\asymp a_r(z)\quad\text{on }G_0\cup G_1,
 \quad
 \mu(G_0)\asymp a_r(z)r^n\asymp\mu(G_1),\quad
 \sup_{G_0}U+c_1a_r(z)\le\inf_{G_1}U.
 \end{gathered}
\end{equation}
The positive constants \(c_0,c_1\) and the comparison
constants depend only on \(n,s,p\) and the fixed
geometric constants.
\end{lem}

\begin{proof}
Put \(A=A_r(z)\) and \(a=a_r(z)=U(A)\).
The corkscrew condition gives
\(d_\Om(A)\asymp r\), and
Corollary~\ref{cor:reifenberg-gradient} gives
\[
 c\,{a}/{r}
 \le|\nabla U(A)|
 \le C\,{a}/{r}.
\]
In particular, \(\nabla U(A)\ne0\). We also need uniform control of the variation of
the gradient near \(A\).
On an interior ball centered at \(A\) whose radius
is a fixed small multiple of \(r\), interior Harnack
and Lemma~\ref{lem:positive-tail} bound the local
supremum and tail by \(Ca\).
Proposition~\ref{prop:C1} therefore gives a normalized
H\"older estimate for \(\nabla U\).
Combining this estimate with the lower bound for
\(|\nabla U(A)|\), we may choose a fixed
\(\rho_0>0\) such that
\[
 B_{\rho_0r}(A)\subset D_{2r}(z),
 \qquad
 \sup_{\zeta\in B_{\rho_0r}(A)}
 |\nabla U(\zeta)-\nabla U(A)|
 \le\frac14|\nabla U(A)|.
\]
All constants in this choice are uniform. Set
\[
 e={\nabla U(A)}/{|\nabla U(A)|},
 \qquad
 \tau=\rho_0/8,
 \qquad
 P_0=A-\tau r e,\quad P_1=A+\tau r e.
\]
The segment joining \(P_0\) to \(P_1\) lies in
\(B_{\rho_0r}(A)\), where
\[
 \nabla U\cdot e\ge\frac34|\nabla U(A)|.
\]
Integration along this segment gives
\[
 U(P_1)-U(P_0)
 \ge\frac32\tau r|\nabla U(A)|.
\]
Choose a fixed \(0<\varepsilon\le1/5\), and put $G_i=B_{\varepsilon\tau r}(P_i)$ for $i=0,1$. These balls are disjoint and contained in
\(B_{\rho_0r}(A)\).
The same gradient comparison gives
\[
 \sup_{G_i}|U-U(P_i)|
 \le\frac54\varepsilon\tau r|\nabla U(A)|.
\]
Consequently,
\[
 \inf_{G_1}U-\sup_{G_0}U
 \ge
 \left(\frac32-\frac52\varepsilon\right)
 \tau r|\nabla U(A)|
 \ge c\tau a.
\]
Since \(\tau\) is fixed, this is the required gap
with a uniform \(c_1>0\).
The common radius is \(c_0r\), where
\(c_0=\varepsilon\tau\). Finally, interior Harnack on a fixed interior
neighborhood of \(A\) gives \(U\asymp a\) on both
balls. Hence
\[
 \mu(G_i)=\int_{G_i}U(y)\,\mathrm d y
 \asymp a\,|G_i|
 \asymp ar^n,
 \qquad i=0,1.
\]
This proves \eqref{eq:two-value-balls}.
\end{proof}

\begin{prop}\label{prop:variable-poincare}
Assume the hypotheses of
Lemma~\ref{lem:variable-volume-scale}.
There are constants \(\kappa\ge2\) and \(C\ge1\),
depending only on \(n,s,p\) and the fixed geometric
constants, such that the following holds.
For every allowed relative ball \(\mathbb B_r(x)\)
and every
\(f\in\mathrm L^2(\mathbb B_{\kappa r}(x),\mu)\)
with
\(\mathscr E_{U,\mathbb B_{\kappa r}(x)}(f)<\infty\),
\begin{equation}\label{eq:variable-poincare}
 \int_{\mathbb B_r(x)}
 |f-f_{\mathbb B_r(x),\mu}|^2\,\mathrm d\mu
 \le
 C\Psi(x,r)
 \mathscr E_{U,\mathbb B_{\kappa r}(x)}(f),\qquad f_{G,\mu}:=\frac{1}{\mu(G)}\int_G f\,\mathrm d\mu,
\end{equation}
where \(\kappa\) is among the fixed enlargement factors
accommodated by the choice of \(\kappa_*\). Suppose also that \(u\) and \(v\) each satisfy the
hypotheses imposed on \(U\), in the same domain and
fractional $p$-harmonicity patch. Let \(z\in\partial\Om\), let \(\mathbb B_r(z)\)
be allowed, and put
\[
 a=u(A_r(z))+v(A_r(z)).
\]
Then, for every \(f\in\mathrm L^2(D_{2r}(z))\) with finite
energy in the right-hand side below,
\begin{equation}\label{eq:pair-boundary-poincare}
 \int_{D_r(z)}|f-f_{D_r(z)}|^2\,\mathrm d x
 \le
 C\frac{r^{sp}}{a^{p-2}}
 \iint_{D_{2r}(z)\times D_{2r}(z)}
 |f(x)-f(y)|^2K_{u,v}(x,y)
 \,\mathrm d y\,\mathrm d x.
\end{equation}
Here \(K_{u,v}\) is defined in
\eqref{eq:averaged-kernel}, and
\(f_G=|G|^{-1}\int_G f\,\mathrm d x\)
denotes the Lebesgue average.
\end{prop}

\begin{proof}
We first prove the weighted estimate for a boundary
center \(x=z\).
Write \(a=a_r(z)\), and choose \(G_0,G_1\) as in
Lemma~\ref{lem:two-value-balls}.
Set
\[
 m=\frac12\biggl(\sup_{G_0}U+\inf_{G_1}U\biggr),
 \qquad
 X_0=D_r(z)\cap\{U\le m\},
 \qquad
 X_1=D_r(z)\cap\{U>m\}.
\]
The gap between the two value ranges gives
\[
 |U(\xi)-U(\eta)|\ge ca
\]
on each of
\[
 X_0\times G_1,\qquad
 X_1\times G_0,\qquad
 G_0\times G_1.
\]
All these pairs satisfy \(|\xi-\eta|\le4r\).
Consequently,
\[
 K_U(\xi,\eta)\ge
 c\,a^{p-2}r^{-n-sp}
\]
on the three families of pairs.
For \(p=2\), the same bound follows directly from
\(K_U(\xi,\eta)=|\xi-\eta|^{-n-2s}\). Put
\[
 f_i=f_{G_i,\mu},
 \qquad
 E=\mathscr E_{U,D_{2r}(z)}(f).
\]
Restricting the energy to \(X_0\times G_1\) and
applying Jensen's inequality in the second variable
gives
\[
 E\ge
 c\,a^{p-2}r^{-n-sp}\mu(G_1)
 \int_{X_0}|f-f_1|^2\,\mathrm d\mu.
\]
The family \(X_1\times G_0\) gives the corresponding
estimate with the indices reversed.
For \(G_0\times G_1\), Jensen's inequality in both
variables gives
\[
 E\ge
 c\,a^{p-2}r^{-n-sp}
 \mu(G_0)\mu(G_1)|f_0-f_1|^2.
\]
Since \(\mu(G_i)\asymp ar^n\), these inequalities imply
\begin{align*}
 \int_{X_0}|f-f_1|^2\,\mathrm d\mu
 +\int_{X_1}|f-f_0|^2\,\mathrm d\mu
 &\le Cr^{sp}a^{1-p}E,\\
 |f_0-f_1|^2
 &\le Cr^{sp-n}a^{-p}E.
\end{align*}
Thus the first two families control \(f\) relative
to the averages on the interior balls, while the
third controls the difference between those averages. To obtain a single comparison constant on \(D_r(z)\),
use
\[
 |f-f_0|^2
 \le2|f-f_1|^2+2|f_1-f_0|^2
 \qquad\text{on }X_0.
\]
Together with
\(\mu(D_r(z))\le Car^n\), this gives
\[
 \int_{D_r(z)}|f-f_0|^2\,\mathrm d\mu
 \le Cr^{sp}a^{1-p}E.
\]
The weighted average minimizes the squared
\(\mathrm L^2(\mu)\) distance over all constants.
Hence
\[
 \int_{D_r(z)}|f-f_{D_r(z),\mu}|^2\,\mathrm d\mu
 \le
 C\Psi(z,r)\mathscr E_{U,D_{2r}(z)}(f).
\]
This proves the boundary-centered estimate.

We next consider an interior center \(x\in\Om\).
Put \(d=d_\Om(x)\) and \(a_x=U(x)\).
Let \(\varepsilon\in(0,1/8)\) be a fixed constant
for which Proposition~\ref{prop:whitney-pair} applies
on every ball \(B_r(x)\) with
\(0<r\le\varepsilon d\).
Interior Harnack gives
\[
 U\asymp a_x\quad\text{on }B_{2r}(x).
\]
Apply \eqref{eq:whitney-pair-energy} with \(u=v=U\),
so that \(K_{u,v}=K_U\).
The kernel coefficient is comparable to
\((a_x/d)^{p-2}\), and the two weighted measure
factors contribute \(a_x^2\).
Therefore
\[
 \mathscr E_{U,B_{2r}(x)}(f)
 \ge
 c\,a_x^p d^{2-p}
 [f]_{\mathrm H^{\sigma_*/2}(B_r(x))}^2.
\]
The smaller-radius version of
Proposition~\ref{prop:whitney-pair} is used here,
with constants uniform for \(r\le\varepsilon d\). Using the minimizing property of the weighted mean,
followed by fractional Poincar\'e with the Lebesgue
average \(f_{B_r(x)}\), we obtain
\begin{align*}
 \int_{B_r(x)}|f-f_{B_r(x),\mu}|^2\,\mathrm d\mu
 &\le
 \int_{B_r(x)}|f-f_{B_r(x)}|^2\,\mathrm d\mu\\
 &\le
 Ca_xr^{\sigma_*}
 [f]_{\mathrm H^{\sigma_*/2}(B_r(x))}^2\le
 C d^{p-2}r^{\sigma_*}a_x^{1-p}
 \mathscr E_{U,B_{2r}(x)}(f).
\end{align*}
By \eqref{eq:interior-scale}, the coefficient is
comparable to \(\Psi(x,r)\).
This proves the estimate for the small interior balls. It remains to consider \(r>\varepsilon d\).
Choose a nearest boundary point \(z\), and put
\(t=r+d\).
Then
\[
 \mathbb B_r(x)\subset\mathbb B_t(z),
 \qquad
 \mathbb B_{2t}(z)\subset
 \mathbb B_{\kappa r}(x)
\]
provided that \(\kappa\ge2+3/\varepsilon\).
The choice of \(\kappa_*\) also accommodates the
fixed enlargements needed for the boundary argument
at \((z,t)\).
Thus all normalization and two-value-ball estimates
used in that argument remain available within the
original fractional $p$-harmonicity patch. Since \(t\asymp r\), the boundary and interior scale
comparisons in the proof of
Lemma~\ref{lem:variable-volume-scale} give
\[
 t^{sp}a_t(z)^{1-p}\asymp\Psi(x,r).
\]
Apply the boundary-centered estimate on \(D_t(z)\).
Minimizing over constants on the smaller ball and
using the two inclusions above yields
\begin{align*}
 \int_{\mathbb B_r(x)}
 |f-f_{\mathbb B_r(x),\mu}|^2\,\mathrm d\mu
 &\le
 \int_{D_t(z)}
 |f-f_{D_t(z),\mu}|^2\,\mathrm d\mu\\
 &\le
 Ct^{sp}a_t(z)^{1-p}
 \mathscr E_{U,D_{2t}(z)}(f)\le
 C\Psi(x,r)
 \mathscr E_{U,\mathbb B_{\kappa r}(x)}(f).
\end{align*}
This completes the proof of
\eqref{eq:variable-poincare}.

Finally, consider the two-background estimate.
Interchanging \(u\) and \(v\) if necessary, assume
\(u(A_r(z))\ge a/2\).
Apply Lemma~\ref{lem:two-value-balls} to \(u\), and
construct \(G_0,G_1,X_0,X_1\) using its two separated
value ranges.
On the same three families of pairs,
\[
 |u(\xi)-u(\eta)|\ge ca.
\]
Lemma~\ref{lem:scalar} therefore gives
\[
 K_{u,v}(\xi,\eta)
 \ge c\,a^{p-2}r^{-n-sp}.
\]
For \(p=2\), this follows from the explicit fractional
kernel. Repeat the preceding three-family Jensen argument
with Lebesgue measure. Since \(|G_i|\asymp r^n\) and
\(|D_r(z)|\le Cr^n\), it gives
\[
 \int_{D_r(z)}|f-f_{G_0}|^2\,\mathrm d x
 \le Cr^{sp}a^{2-p}
 \iint_{D_{2r}(z)^2}|f(x)-f(y)|^2K_{u,v}(x,y)\dxdy.
\]
Minimization over constants proves
\eqref{eq:pair-boundary-poincare}.
\end{proof}

\begin{rem} The boundary-centered argument in Proposition \ref{prop:variable-poincare} connects every point of
\(D_r(z)\) to one of the two interior sets \(G_0,G_1\)
through interactions with a quantitative kernel lower bound.
The interaction between these sets controls the difference
between their averages, providing a common reference value
throughout the boundary ball.
Thus the energy controls oscillation across distinct Whitney
balls without an iteration along chains of adjacent balls.
\end{rem}
\subsection{Jump moments, exterior interactions, and averaging}

We next estimate the short-range and distant interactions
of \(K_U\) relative to the measure \(\mu\).
These estimates also control the energy of cutoff functions.
For a bounded Lipschitz function \(\eta\) on \(S\), put
\[
 \Gamma_U(\eta)(x)
 =
 \int_S|\eta(x)-\eta(y)|^2K_U(x,y)\,\mathrm d\mu(y).
\]
The integration extends over the global set \(S=\{U>0\}\),
so this energy density includes interactions with all
positive exterior values of \(U\).

\begin{prop}\label{prop:variable-moment}
Assume the hypotheses of
Lemma~\ref{lem:variable-volume-scale}.
For \(\mu\)-almost every \(x\in D_{LR}(w)\) and every
\(r>0\) satisfying $|x-w|+\kappa_*r<LR$, we have
\begin{equation}\label{eq:variable-moment}
 \frac{1}{r^2}\int_{S\cap B_r(x)}
 |x-y|^2K_U(x,y)\,\mathrm d\mu(y)
 +
 \int_{S\setminus B_r(x)}
 K_U(x,y)\,\mathrm d\mu(y)
 \le {C}/{\Psi(x,r)}.
\end{equation}
In particular, if \(0\le\eta\le1\) on \(S\) and
\(\operatorname{Lip}(\eta)\le C_0/r\), then
\[
 \Gamma_U(\eta)(x)\le {C}/{\Psi(x,r)}
\]
for the same centers and radii. Moreover, for every allowed boundary ball
\(\mathbb B_r(z)\),
\begin{equation}\label{eq:variable-moment+}
\frac{1}{r^2}\int_{S\cap B_r(x)}
 |x-y|^2K_U(x,y)\,\mathrm d\mu(y)
 +
 \int_{S\setminus B_r(x)}
 K_U(x,y)\,\mathrm d\mu(y)
 \le Ca_r(z)^{p-1}r^{-sp}
\end{equation}
for \(\mu\)-almost every \(x\in D_r(z)\).
The constants depend only on \(n,s,p\) and the fixed
geometric constants, and additionally on \(C_0\) in
the cutoff estimate.
\end{prop}

\begin{proof}
Denote the left-hand side of
\eqref{eq:variable-moment} by \(I(x,r)\). We first record the local kernel bound used below.
Put \(d=d_\Om(x)\).
Interior Harnack, Lemma~\ref{lem:positive-tail}, and
Proposition~\ref{prop:C1} give
\[
 U\asymp U(x),
 \qquad
 |\nabla U|\le C\frac{U(x)}{d}
 \quad\text{on }B_{d/2}(x).
\]
These estimates follow by working on fixed interior
balls inside \(B_d(x)\).
The larger fractional $p$-harmonicity patch provides the required
room, since
\[
 B_d(x)\subset B_{2LR}(w)\subset B_{\Lambda R}(w).
\]
Thus the argument does not require
\(B_{d/2}(x)\) to lie in the smaller controlled patch. Integration along segments gives
\[
 |U(x)-U(y)|
 \le C\frac{U(x)}{d}|x-y|
 \qquad(|x-y|<d/2).
\]
Consequently,
\[
 K_U(x,y)U(y)
 \le
 C U(x)^{p-1}d^{2-p}|x-y|^{-n-\sigma_*},
 \qquad |x-y|<d/2,
\]
where \(\sigma_*=sp-p+2\in(0,2)\).

We first prove the boundary-scale estimate in \eqref{eq:variable-moment+}.
Let \(\mathbb B_r(z)\) be an allowed boundary ball,
take \(x\in D_r(z)\), and write \(a=a_r(z)\).
Then \(d\le r\).
We split the interactions into the three regions
\[
 B_{d/2}(x),\qquad
 B_{4r}(z)\setminus B_{d/2}(x),\qquad
 \R^n\setminus B_{4r}(z).
\]
The first two regions lie within the fractional $p$-harmonicity patch,
where local estimates for \(U\) are available.
The third is treated using the full positive tail. On the first region, the local kernel bound gives
\begin{align*}
 &\int_{S\cap B_{d/2}(x)}
 \frac{|x-y|^2}{r^2}K_U(x,y)\,\mathrm d\mu(y)\le
 C r^{-2}U(x)^{p-1}d^{2-p}
 \int_0^{d/2}t^{1-\sigma_*}\,\mathrm d t
 \le
 C r^{-2}U(x)^{p-1}d^{2-sp}.
\end{align*}
The last equality of exponents uses
\(4-p-\sigma_*=2-sp\).
Boundary decay gives
\[
 U(x)\le Ca\bigl({d}/{r}\bigr)^\beta.
\]
Hence the contribution of this region is bounded by
\[
 Ca^{p-1}r^{-sp}
 \bigl({d}/{r}\bigr)^\gamma,
 \qquad
 \gamma=2-sp+\beta(p-1)>0.
\]
Since \(d\le r\), this is at most
\(Ca^{p-1}r^{-sp}\). For the second region, apply boundary decay on a fixed
enlargement of \(B_r(z)\), and compare its corkscrew
value with \(a\) by a Harnack chain.
For \(y\in S\cap B_{4r}(z)\), this gives
\[
 U(x)+U(y)
 \le
 Ca\left(\frac{d+|x-y|}{r}\right)^\beta,
\]
because \(d_\Om(y)\le d+|x-y|\).
When \(|x-y|\ge d/2\), the numerator
\(d+|x-y|\) is comparable to \(|x-y|\).
It follows that
\[
 K_U(x,y)U(y)
 \le
 Ca^{p-1}r^{-\beta(p-1)}
 |x-y|^{-n-sp+\beta(p-1)}.
\]
Decompose this region into dyadic annuli
\(\ell\le|x-y|<2\ell\).
Here \(d/2\le\ell\le5r\), up to fixed endpoint factors.
Using
\[
 \min\biggl\{1,{|x-y|^2}/{r^2}\biggr\}
 \le{|x-y|^2}/{r^2},
\]
the contribution of such an annulus is at most
\[
 Ca^{p-1}r^{-\beta(p-1)-2}
 \ell^{2-sp+\beta(p-1)}
 =
 Ca^{p-1}r^{-sp}\bigl({\ell}/{r}\bigr)^\gamma.
\]
Since \(\gamma>0\), summing these contributions gives
\(Ca^{p-1}r^{-sp}\).
This is precisely where the choice
\(sp-\beta(p-1)<2\) is used. For the third region, \(|y-z|>4r\), we have
\(|x-y|\asymp|y-z|\).
We do not use a gradient estimate at the distant
endpoint.
Instead, use
\begin{equation}\label{eq:far-jump-split}
 |U(x)-U(y)|^{p-2}U(y)
 \le
 C\bigl(U(x)^{p-2}U(y)+U(y)^{p-1}\bigr).
\end{equation}
The normalized mass estimate
\eqref{eq:boundary-normalized-mass} gives
\[
 \int_{|y-z|>4r}
 \frac{U(y)^{p-1}}{|y-z|^{n+sp}}\,\mathrm d y
 \le Ca^{p-1}r^{-sp}.
\]
For \(p>2\), H\"older's inequality also gives
\begin{align*}
 \int_{|y-z|>4r}
 \frac{U(y)}{|y-z|^{n+sp}}\,\mathrm d y
 &\le
 \left(
 \int_{|y-z|>4r}
 \frac{U(y)^{p-1}}{|y-z|^{n+sp}}\,\mathrm d y
 \right)^{1/(p-1)}\times
 \left(
 \int_{|y-z|>4r}
 \frac{\mathrm d y}{|y-z|^{n+sp}}
 \right)^{(p-2)/(p-1)}\\
 &\le Ca r^{-sp}.
\end{align*}
For \(p=2\), this is already the natural tail integral.
Since \(U(x)\le Ca\), both terms in
\eqref{eq:far-jump-split} contribute at most
\(Ca^{p-1}r^{-sp}\).
Combining the estimates on the three regions proves
\[
 I(x,r)\le Ca_r(z)^{p-1}r^{-sp}
 \qquad(x\in D_r(z)).
\]
All fixed enlargements used in this calculation are
covered by the allowed boundary ball.
No additional allowed-ball condition centered at \(x\)
is needed. This completes the proof of \eqref{eq:variable-moment+}.

We next prove \eqref{eq:variable-moment} at an interior
center \(x\) when \(r<d/4\).
For interactions with \(|x-y|<d/2\), split the radial
integral at \(r\).
The local kernel bound yields
\begin{align*}
 &\int_{S\cap B_{d/2}(x)}
 \min\bigl\{1,{|x-y|^2}/{r^2}\bigr\}
 K_U(x,y)\,\mathrm d\mu(y)\\
 &\le
 C U(x)^{p-1}d^{2-p}
 \left(
 r^{-2}\int_0^r t^{1-\sigma_*}\,\mathrm d t
 +
 \int_r^{d/2}t^{-1-\sigma_*}\,\mathrm d t
 \right)\le
 C U(x)^{p-1}d^{2-p}r^{-\sigma_*}.
\end{align*}
Both integrals are controlled because
\(0<\sigma_*<2\). For \(|x-y|\ge d/2\), apply
Lemma~\ref{lem:positive-tail} at \(x\), with radius
\(d/4\).
It gives
\[
 \int_{|x-y|\ge d/2}
 \frac{U(y)^{p-1}}{|x-y|^{n+sp}}\,\mathrm d y
 \le C U(x)^{p-1}d^{-sp}.
\]
Using \eqref{eq:far-jump-split} and the same
H\"older argument as above bounds the remaining
contribution by \(C U(x)^{p-1}d^{-sp}\).
Moreover,
\[
 U(x)^{p-1}d^{-sp}
 =
 U(x)^{p-1}d^{2-p}r^{-\sigma_*}
 \bigl ({r}/{d}\bigr)^{\sigma_*},
\]
so this is bounded by the preceding local estimate.
Equation~\eqref{eq:interior-scale} now gives
\(I(x,r)\le C/\Psi(x,r)\). Finally, suppose that \(r\ge d/4\).
Choose a nearest boundary point \(z_x\), and put
\(t=r+d\le5r\).
Then \(x\in D_t(z_x)\), and
\[
 \min\bigl\{1,{|x-y|^2}/{r^2}\bigr\}
 \le
 \bigl({t}/{r}\bigr)^2
 \min\bigl\{1,{|x-y|^2}/{t^2}\bigr\}.
\]
The choice of \(\kappa_*\) ensures that the fixed
enlargements needed for the boundary calculation
at \((z_x,t)\) lie within the fractional $p$-harmonicity patch.
Applying that calculation gives
\[
 I(x,r)\le25I(x,t)
 \le Ca_t(z_x)^{p-1}t^{-sp}.
\]
Comparable-scale Harnack chains and the construction
of the scale in
Lemma~\ref{lem:variable-volume-scale} imply
\[
 a_t(z_x)^{p-1}t^{-sp}\asymp {1}/{\Psi(x,r)}.
\]
This proves \eqref{eq:variable-moment} for all allowed
radii.

To obtain the cutoff estimate, observe that
\(0\le\eta\le1\) and
\(\operatorname{Lip}(\eta)\le C_0/r\) imply
\[
 |\eta(x)-\eta(y)|^2
 \le
 C\min\bigl\{1,{|x-y|^2}/{r^2}\bigr\}.
\]
Integration against \(K_U(x,y)\,\mathrm d\mu(y)\)
therefore gives
\[
 \Gamma_U(\eta)(x)\le CI(x,r)
 \le{C}/{\Psi(x,r)}.
\]
This completes the proof.
\end{proof}

\begin{prop}\label{prop:variable-ujs}
Assume the hypotheses of
Lemma~\ref{lem:variable-volume-scale}.
For \(\mu\)-almost every \(x\in D_{LR}(w)\) and every
\(\rho>0\) satisfying $|x-w|+\kappa_*\rho<LR$, we have
\begin{equation}\label{eq:variable-ujs}
 K_U(x,y)
 \le
 \frac{C}{\mu(\mathbb B_\rho(x))}
 \int_{\mathbb B_\rho(x)}
 K_U(\xi,y)\,\mathrm d\mu(\xi)
\end{equation}
for \(\mu\)-almost every \(y\in S\) with
\(|x-y|\ge2\rho\).
The constant depends only on \(n,s,p\) and the fixed
geometric constants.
There is no location restriction on \(y\), apart
from its separation from \(x\).
\end{prop}

\begin{proof}
For \(\xi\in\mathbb B_\rho(x)\) and
\(|x-y|\ge2\rho\), the triangle inequality gives
\[
 \frac12|x-y|
 \le|\xi-y|
 \le\frac32|x-y|.
\]
When \(p=2\), the kernel is
\(K_U(x,y)=|x-y|^{-n-2s}\), and
\eqref{eq:variable-ujs} follows immediately by
integrating this distance comparison.
We therefore assume \(p>2\).

To proceed we first construct two subsets
\(G_0,G_1\subset\mathbb B_\rho(x)\) and a number
\(b>0\) such that
\[
 \mu(G_i)\ge c\mu(\mathbb B_\rho(x)),
 \qquad i=0,1,
\]
and
\[
 \sup_{G_0}U+cb\le\inf_{G_1}U,
 \qquad
 |U(\xi)-U(x)|\le Cb
 \quad(\xi\in G_0\cup G_1).
\]
Thus each subset carries a fixed fraction of the
measure of the averaging ball, their value ranges
are separated, and neither range is too far from
\(U(x)\). Put \(d=d_\Om(x)\), and first suppose that
\(0<\rho<2d\).
As in the proof of Lemma~\ref{lem:two-value-balls},
Corollary~\ref{cor:reifenberg-gradient} and the
normalized interior estimate in
Proposition~\ref{prop:C1} give a fixed
\(\varepsilon_0\in(0,1/8)\) such that
\[
 |\nabla U(x)|\asymp\frac{U(x)}{d},
 \qquad
 \sup_{\zeta\in B_{\varepsilon_0d}(x)}
 |\nabla U(\zeta)-\nabla U(x)|
 \le\frac14|\nabla U(x)|.
\]
Set
\[
 e={\nabla U(x)}/{|\nabla U(x)|},
 \qquad
 \ell=\frac14\min\{\rho,\varepsilon_0d\}.
\]
For a fixed sufficiently small
\(\varepsilon_1>0\), choose
\[
 G_0=B_{\varepsilon_1\ell}(x-\ell e),
 \qquad
 G_1=B_{\varepsilon_1\ell}(x+\ell e).
\]
These balls lie inside both \(B_\rho(x)\) and
\(B_{\varepsilon_0d}(x)\).
Integration in the direction \(e\), followed by the
gradient upper bound on each ball, gives
\[
 \inf_{G_1}U-\sup_{G_0}U
 \ge c\,U(x)\frac{\ell}{d},
 \qquad
 |U(\xi)-U(x)|
 \le C\,U(x)\frac{\ell}{d}
 \quad(\xi\in G_0\cup G_1).
\]
Since \(\rho<2d\), we have \(\ell\asymp\rho\),
with constants depending only on \(\varepsilon_0\).
The required separation and proximity estimates
therefore hold with $b=U(x){\rho}/{d}$. Interior Harnack gives \(U\asymp U(x)\) on both
sets, so
\[
 \mu(G_i)\asymp U(x)\rho^n.
\]
Moreover,
\[
 \mu(\mathbb B_\rho(x))\asymp U(x)\rho^n
 \qquad(0<\rho<2d).
\]
For \(\rho<d/4\), this is the interior volume estimate.
For \(d/4\le\rho<2d\), it follows from the
boundary-scale volume comparison in
Lemma~\ref{lem:variable-volume-scale} and a bounded
Harnack chain comparing \(U(x)\) with the corresponding
corkscrew value.
This proves the required measure fractions in this
regime. Now suppose that \(\rho\ge2d\).
Choose a nearest boundary point \(z\), and put
\(t=\rho/8\).
The ball \(\mathbb B_t(z)\) is allowed, since
\(|x-z|=d\le\rho/2\) and the allowed-ball condition
holds at \((x,\rho)\).
Apply Lemma~\ref{lem:two-value-balls} at \((z,t)\).
The resulting sets satisfy
\[
 G_0,G_1\subset D_{2t}(z)
 \subset\mathbb B_\rho(x),
\]
because \(2t+d\le3\rho/4\).
Put \(b=a_t(z)\).
The lemma gives
\[
 \mu(G_i)\asymp b\rho^n,
 \qquad
 \sup_{G_0}U+cb\le\inf_{G_1}U,
 \qquad
 U\asymp b\quad\text{on }G_0\cup G_1.
\]
Boundary normalization on fixed enlargements and
comparable-scale Harnack chains also give
\[
 U(x)\le Cb,
 \qquad
 \mu(\mathbb B_\rho(x))\asymp b\rho^n.
\]
All required enlargements are covered by the choice
of \(\kappa_*\).
The desired measure fractions and the estimate
\(|U(\xi)-U(x)|\le Cb\) follow.

We now compare the kernel with its average.
Fix \(y\in S\) with \(|x-y|\ge2\rho\), and write
\[
 c_y=U(y),
 \qquad
 D_y=|U(x)-c_y|.
\]
Let \(C_1\) be the constant in
\(|U(\xi)-U(x)|\le C_1b\) on the two subsets.
If \(D_y\ge2C_1b\), then
\[
 |U(\xi)-c_y|
 \ge D_y-C_1b
 \ge\frac12D_y
 \quad(\xi\in G_0\cup G_1).
\]
If \(D_y<2C_1b\), use the separation between their
value ranges.
According to which side of
\[
 \frac12\biggl(\sup_{G_0}U+\inf_{G_1}U\biggr)
\]
contains \(c_y\), choose \(G_1\) or \(G_0\),
respectively.
On the chosen set,
\[
 |U(\xi)-c_y|\ge cb\ge c'D_y.
\]
Thus, in either case, one of the two subsets has
measure at least \(c\mu(\mathbb B_\rho(x))\) and
satisfies
\[
 |U(\xi)-U(y)|
 \ge c|U(x)-U(y)|.
\]
Combining this with the distance comparison at the
beginning of the proof gives
\[
 K_U(\xi,y)\ge cK_U(x,y)
\]
on that subset.
Integration proves \eqref{eq:variable-ujs}.
The argument places no bound on \(U(y)\) and uses
no geometric information at the distant endpoint.
\end{proof}

For later use, we record the integrated cutoff consequence
of Proposition~\ref{prop:variable-moment}. To state it precisely,
let \(x_0\in\overline\Om\), \(t,\rho>0\), and suppose
that $|x_0-w|+t+(\kappa_*+2)\rho<LR$. Consider the concentric relative balls
\[
 B_0=\mathbb B_t(x_0),
 \qquad
 B_1=\mathbb B_{t+\rho}(x_0),
 \qquad
 G=\mathbb B_{t+2\rho}(x_0).
\]
The displayed condition ensures that
\(\mathbb B_\rho(x)\) is allowed for every
\(x\in\overline G\).
Choose a function \(0\le\eta\le1\) on \(S\),
equal to \(1\) on \(S\cap B_t(x_0)\) and to \(0\)
outside \(B_{t+\rho}(x_0)\), with
\(\operatorname{Lip}(\eta)\le C/\rho\).
For every \(f\in\mathrm L^2(G,\mu)\), the pointwise
cutoff estimate gives
\begin{equation}\label{cutt}
 \int_G |f|^2\Gamma_U(\eta)\,\mathrm d\mu
 \le
 C\int_G\frac{|f(x)|^2}{\Psi(x,\rho)}\,\mathrm d\mu(x)
 \le
 C\left(\sup_{x\in G}\frac{1}{\Psi(x,\rho)}\right)
 \int_G|f|^2\,\mathrm d\mu.
\end{equation}
The supremum is finite.
Indeed, \(\overline G\) is compact and can be covered
by finitely many balls of radius \(\rho/2\).
Nearby-center comparability bounds
\((\Psi(x,\rho))^{-1}\) on each such ball by its value
at a fixed center. Thus \eqref{cutt} is the required cutoff
Sobolev bound, with no energy term involving \(f\) on
the right. It also controls localized cutoff densities,
since restricting the second endpoint only decreases
\(\Gamma_U(\eta)\). Neither this estimate nor kernel
averaging requires doubling of the distant exterior measure.

\subsection{Boundary cutoffs and the quotient equation}

As in the half-space case, the weak equation for the quotient
is initially tested only against functions supported away
from the boundary. We now show that, for bounded quotient
solutions, it remains valid for admissible test functions
whose supports meet the boundary within the controlled patch.
This is the counterpart of Lemma~\ref{lem:model-boundary-polar}
for the variable background \(U\) and the Reifenberg boundary.

Assigning the boundary zero \(\mu\)-measure does not by itself
justify this enlargement of the test function class. We construct
cutoff functions equal to one near each compact boundary
piece whose \(\mathrm L^2(\mu)\) norms and full
\(\mathscr E_U\)-energies tend to zero, including the energy
of interactions with distant points in \(S\).
These cutoffs establish zero capacity and allow us to
approximate test functions meeting the boundary by functions
supported away from it.

\begin{lem}\label{lem:variable-boundary-capacity}
Assume the hypotheses of
Lemma~\ref{lem:variable-volume-scale}.
Let \(\kappa_0\) and \(\beta\) be the exponents fixed
in Subsection~\ref{subprelim}, and put $\varepsilon_0=\kappa_0-p(s-\beta)>0$. Let \(z\in\partial\Om\) and \(r>0\) satisfy $|z-w|+\kappa_*r<LR$. For every \(0<\rho<r/8\), there is a Lipschitz
function \(\zeta_\rho\) on \(S\), extended continuously
to the locally adjoined boundary, such that
\(0\le\zeta_\rho\le1\),
\(\operatorname{Lip}(\zeta_\rho)\le C/\rho\),
and
\[
 \operatorname{supp}\zeta_\rho
 \subset \mathbb B_{2r}(z)\cap\{d_\Om<2\rho\},
 \qquad
 \zeta_\rho=1
 \quad\text{on }\mathbb B_r(z)\cap\{d_\Om\le\rho\}.
\]
The support and the equality are understood relative
to \(\overline\Om\), with
\(d_\Om=0\) on \(\partial\Om\). Moreover,
\begin{align}
 \|\zeta_\rho\|_{\mathrm L^2(S,\mu)}^2
 &\le
 Ca_r(z)r^n
 \bigl({\rho}/{r}\bigr)^{\kappa_0+\beta},
 \label{eq:strip-measure}\\
 \mathscr E_U(\zeta_\rho)
 &\le
 Ca_r(z)^p r^{n-sp}
 \bigl({\rho}/{r}\bigr)^{\varepsilon_0}
 \longrightarrow0
 \quad\text{as }\rho\downarrow0.
 \label{eq:strip-capacity}
\end{align}
The energy is the full form on \(S\times S\).
All constants depend only on \(n,s,p\) and the fixed
geometric constants. Consequently, every compact set
\(K\subset\partial\Om\cap B_{LR}(w)\) has zero capacity
in the following sense. There are Lipschitz functions \(\eta_j\), supported
compactly in the controlled patch $B_{LR}(w)$ and extended by
zero to the remaining part of \(S\), equal to one
on relative neighborhoods of \(K\), such that
\[
 \|\eta_j\|_{\mathrm L^2(S,\mu)}^2
 +\mathscr E_U(\eta_j)\longrightarrow0.
\]
Here relative neighborhoods are taken in the locally
adjoined space \(\overline\Om\).
\end{lem}

\begin{proof}
Uniform exterior corkscrews imply uniform porosity
of the boundary.
The corresponding covering iteration gives
\[
 |D_{2r}(z)\cap\{d_\Om<2\rho\}|
 \le
 Cr^n\bigl({\rho}/{r}\bigr)^{\kappa_0},
 \qquad 0<\rho<r/8.
\]
The exponent \(\kappa_0\) is the one fixed at the
initial geometric flatness threshold in
Subsection~\ref{subprelim}.
Subsequent reductions of the flatness constant
preserve this estimate.
In the Reifenberg setting, it can also be obtained
by iterating the boundary-layer estimate in
\cite[Lemma~2.3]{BKS} at that fixed threshold.

Choose a Lipschitz function \(\chi\) on \(\R^n\)
such that
\[
 0\le\chi\le1,\qquad
 \chi=1\text{ on }B_r(z),\qquad
 \operatorname{supp}\chi\subset B_{5r/4}(z),
 \qquad
 \operatorname{Lip}(\chi)\le C/r.
\]
The slightly smaller support will ensure that the
nearby corkscrew points used below lie in \(B_{2r}(z)\).
Choose also a Lipschitz function
\(\vartheta:[0,\infty)\to[0,1]\), equal to one on
\([0,1]\) and zero on \([3/2,\infty)\), and define
\[
 \zeta_\rho(x)
 =
 \chi(x)\vartheta\left(\frac{d_\Om(x)}{\rho}\right),
 \qquad x\in S.
\]
Since \(d_\Om\) is globally \(1\)-Lipschitz, $\operatorname{Lip}(\zeta_\rho)
 \le C/r+C/\rho\le C/\rho$. The same formula defines its values on the adjoined
boundary, where \(d_\Om=0\).
It also makes \(\zeta_\rho\) identically zero on
the part of \(S\) outside \(B_{5r/4}(z)\).
The choice of \(\vartheta\) ensures that its support
lies strictly inside \(\{d_\Om<2\rho\}\).

Write \(a=a_r(z)\), and put
\[
 T_\rho
 =
 S\cap B_{5r/4}(z)\cap\{d_\Om<2\rho\}.
\]
This set contains the support of \(\zeta_\rho\)
up to the adjoined boundary, which has zero
\(\mu\)-measure.
For \(x\in T_\rho\), choose a nearest boundary
point \(q_x\).
Both \(x\) and \(q_x\) lie in \(B_{2r}(z)\).
The boundary H\"older estimate in
Proposition~\ref{prop:boundary-normalization},
together with \(U(q_x)=0\), gives
\[
 U(x)
 \le
 Ca\left(\frac{d_\Om(x)}{r}\right)^\beta
 \le
 Ca\bigl({\rho}/{r}\bigr)^\beta.
\]
Combining this with the boundary-layer estimate yields
\[
 \mu(T_\rho)
 \le
 Ca r^n\bigl({\rho}/{r}\bigr)^{\kappa_0+\beta}.
\]
Since \(0\le\zeta_\rho\le1\) and
\(\zeta_\rho=0\) outside \(T_\rho\) up to null sets,
this proves \eqref{eq:strip-measure}.

We next estimate the full energy.
Fix \(x\in T_\rho\), and use the boundary point
\(q_x\) chosen above.
Since \(|x-q_x|<2\rho\), we have
\(x\in D_{4\rho}(q_x)\).
The ball \(\mathbb B_{4\rho}(q_x)\) is allowed.
Indeed,
\[
 |q_x-w|+4\kappa_*\rho
 \le
 |z-w|+\frac54r+2\rho+4\kappa_*\rho
 <
 |z-w|+\kappa_*r
 <LR,
\]
where we used \(\rho<r/8\) and \(\kappa_*\ge8\). Let \(A_{4\rho}(q_x)\) be its corkscrew point.
Then
\[
 |A_{4\rho}(q_x)-z|
 \le\frac54r+6\rho<2r.
\]
Thus both \(q_x\) and this corkscrew point lie in
the ball on which the normalized boundary H\"older
estimate is available.
Consequently,
\[
 a_{4\rho}(q_x)
 =
 U(A_{4\rho}(q_x))
 \le
 Ca\bigl({\rho}/{r}\bigr)^\beta.
\]

The bounds on \(\zeta_\rho\) and its Lipschitz
constant imply
\[
 |\zeta_\rho(x)-\zeta_\rho(y)|^2
 \le
 C\min\left\{1,\frac{|x-y|^2}{(4\rho)^2}\right\}.
\]
Apply the boundary-scale estimate in
Proposition~\ref{prop:variable-moment} at
\((q_x,4\rho)\).
It gives
\begin{align*}
 \Gamma_U(\zeta_\rho)(x)
 &\le
 C a_{4\rho}(q_x)^{p-1}\rho^{-sp}\le
 C a^{p-1}r^{-\beta(p-1)}
 \rho^{\beta(p-1)-sp}.
\end{align*}
This bound includes integration over all of \(S\),
including its distant exterior part. Since \(\zeta_\rho\) vanishes outside \(T_\rho\),
every pair contributing to its energy has at least
one endpoint in \(T_\rho\).
Symmetry of the kernel therefore gives
\[
 \mathscr E_U(\zeta_\rho)
 \le
 \int_{T_\rho}
 \Gamma_U(\zeta_\rho)(x)\,\mathrm d\mu(x).
\]
Using the preceding energy-density bound and the
estimate for \(\mu(T_\rho)\), we obtain
\begin{align*}
 \mathscr E_U(\zeta_\rho)
 &\le
 Ca^p r^{n-sp}
 \bigl({\rho}/{r}\bigr)^{
       \kappa_0+\beta+\beta(p-1)-sp}=
 Ca^p r^{n-sp}\bigl({\rho}/{r}\bigr)^{
       \kappa_0-p(s-\beta)}.
\end{align*}
This proves \eqref{eq:strip-capacity}.
The positivity of \(\varepsilon_0\) is precisely
what makes this energy tend to zero.

If \(K\subset\partial\Om\cap B_r(z)\) is compact,
then \(\zeta_\rho=1\) on a relative neighborhood
of \(K\).
Equations \eqref{eq:strip-measure} and
\eqref{eq:strip-capacity} therefore prove the
zero-capacity assertion for such a set.
A general compact subset of
\(\partial\Om\cap B_{LR}(w)\) is covered by finitely
many boundary balls satisfying the stated condition.
Take the sum of the corresponding cutoffs and
truncate it at one.
The resulting function equals one near the compact
set, and its squared norm and energy are bounded
by a fixed multiple of the sum of the corresponding
squared norms and energies.
They therefore tend to zero, proving the final
assertion.
\end{proof}

\begin{prop}\label{prop:bounded-quotient-removal}
Assume the hypotheses of
Lemma~\ref{lem:variable-volume-scale}.
Let \(q\) be measurable on \(S\), with
\(0\le q\le M_q<\infty\).
Suppose that
\[
 \mathscr E_{U,V}(q)<\infty
 \quad\text{for every open }V\Subset D_{LR}(w),
\]
and that
\begin{equation}\label{eq:quotient-cutoff-energytest}
 \mathscr E_U(q,\phi)=0
\end{equation}
for every \(\phi\in\mathrm L^2(S,\mu)\) with finite
full energy and compact support in \(D_{LR}(w)\).
Then, for every Lipschitz cutoff \(\eta\) whose
support is a compact subset of
\(\overline\Om\cap B_{LR}(w)\),
\begin{equation}\label{eq:quotient-cutoff-energy}
 \mathscr E_U(q\eta)
 \le M_q^2\mathscr E_U(\eta).
\end{equation}
Furthermore, each \(q\eta\) can be approximated in the full form norm
by Lipschitz functions with compact support in any fixed
relatively open neighborhood of \(\supp\eta\) within the
controlled patch $B_{LR}(w)$.
Thus \(q\) belongs locally to the form domain also at the
adjoined boundary.
Moreover, the weak identity in \eqref{eq:quotient-cutoff-energytest} remains valid for every
\(\phi\in\mathrm L^2(S,\mu)\) with finite full energy
whose support is a compact subset of
\(\overline\Om\cap B_{LR}(w)\), including test functions whose
supports meet \(\partial\Om\). Supports are understood relative to the locally
adjoined space.
All cutoffs and test functions are extended by zero
to the remaining part of \(S\), and their energies
include all interactions in \(S\times S\).
\end{prop}

\begin{proof}
We first note that interior localization is admissible
under the stated hypotheses.
If \(\psi\) is a Lipschitz cutoff compactly supported
in \(D_{LR}(w)\), choose an open set
\(V\Subset D_{LR}(w)\) containing its support with
a positive buffer.
The local energy of \(q\) controls interactions in
\(V\times V\).
Interactions between the support of \(\psi\) and
\(S\setminus V\) are separated by a positive distance
and are controlled by the boundedness of \(q\) and
Proposition~\ref{prop:variable-moment}.
Thus \(q\psi\) and \(q\psi^2\) have finite full energy,
and \(q\psi^2\) is an admissible interior test function.

Let \(\eta\) be a cutoff as in the statement.
Cover the compact set
\(\partial\Om\cap\operatorname{supp}\eta\)
by finitely many boundary balls to which
Lemma~\ref{lem:variable-boundary-capacity} applies.
Taking truncated sums of the corresponding strip
cutoffs gives Lipschitz functions \(\zeta_j\) such that
\[
 0\le\zeta_j\le1,\qquad
 \zeta_j\longrightarrow0
 \quad\mu\text{-almost everywhere on }S,
\]
and
\[
 \|\zeta_j\|_{\mathrm L^2(S,\mu)}^2
 +\mathscr E_U(\zeta_j)\longrightarrow0.
\]
Their supports lie in a fixed compact subset of the
controlled patch, and each \(\zeta_j\) equals one
on a relative neighborhood of
\(\partial\Om\cap\operatorname{supp}\eta\).
Consequently, $\eta_j=\eta(1-\zeta_j)$ has compact support strictly inside \(D_{LR}(w)\).

Testing the equation with \(q\eta_j^2\) and using
\[
 (a-b)(a\alpha^2-b\gamma^2)
 =
 (a\alpha-b\gamma)^2-ab(\alpha-\gamma)^2
\]
gives the exact identity
\[
 \mathscr E_U(q\eta_j)
 =
 \frac12\iint_{S\times S}
 q(x)q(y)|\eta_j(x)-\eta_j(y)|^2
 K_U(x,y)\,\mathrm d\mu(x)\,\mathrm d\mu(y).
\]
Since \(0\le q\le M_q\),
\[
 \mathscr E_U(q\eta_j)
 \le M_q^2\mathscr E_U(\eta_j).
\]

We next show that \(\eta_j\to\eta\) in the full
form norm.
The product estimate gives
\[
 \mathscr E_U(\eta\zeta_j)
 \le
 C\|\eta\|_\infty^2\mathscr E_U(\zeta_j)
 +
 C\int_S\zeta_j^2\Gamma_U(\eta)\,\mathrm d\mu.
\]
The moment estimate at a sufficiently small fixed
radius bounds \(\Gamma_U(\eta)\) uniformly on the
fixed compact set containing the supports of
\(\zeta_j\).
The second term therefore tends to zero with
\(\|\zeta_j\|_{\mathrm L^2(S,\mu)}^2\).
Also,
\[
 \|\eta\zeta_j\|_{\mathrm L^2(S,\mu)}
 \le
 \|\eta\|_\infty
 \|\zeta_j\|_{\mathrm L^2(S,\mu)}
 \longrightarrow0.
\]
Thus \(\eta_j\to\eta\) in energy and in
\(\mathrm L^2(S,\mu)\).
Since \(q\eta_j\to q\eta\) almost everywhere,
Fatou's lemma in the nonnegative pair integral yields
\[
 \mathscr E_U(q\eta)
 \le
 \liminf_{j\to\infty}\mathscr E_U(q\eta_j)
 \le
 M_q^2\mathscr E_U(\eta).
\]
This proves \eqref{eq:quotient-cutoff-energy}.
In particular, \(q\eta\) has finite full energy
even when the support of \(\eta\) meets the boundary.

Before enlarging the test function class, we record continuity
of the weak pairing on a fixed compact support.
Let \(K\subset\overline\Om\cap B_{LR}(w)\) be compact,
and choose a Lipschitz cutoff \(0\le\chi\le1\),
compactly supported in the patch and equal to one
on a neighborhood of \(K\).
For a finite-energy function \(h\) supported in \(K\),
write
\[
 \mathscr E_U(q,h)
 =
 \mathscr E_U(q\chi,h)
 +
 \mathscr E_U(q(1-\chi),h).
\]
The first term satisfies
\[
 |\mathscr E_U(q\chi,h)|
 \le
 \mathscr E_U(q\chi)^{1/2}
 \mathscr E_U(h)^{1/2}.
\]
The first factor is finite by the estimate just proved.  For the second term, choose \(a>0\) sufficiently small that
\(\chi=1\) on \(S\cap B_a(x)\) for every \(x\in K\),
and that the corresponding relative balls are allowed.
Proposition~\ref{prop:variable-moment} gives
\[
 \int_S q(y)(1-\chi(y))K_U(x,y)\,\mathrm d\mu(y)
 \le M_q\int_{S\setminus B_a(x)}
              K_U(x,y)\,\mathrm d\mu(y)
 \le \frac{CM_q}{\Psi(x,a)}
\]
for \(\mu\)-almost every \(x\in K\).
The nearby-center comparison in
Lemma~\ref{lem:variable-volume-scale}, together with
a finite covering of \(K\), bounds \(\Psi(x,a)^{-1}\)
uniformly on \(K\). Since \(\mu(K)<\infty\), Cauchy-Schwarz therefore yields
\[
 |\mathscr E_U(q(1-\chi),h)|
 \le C_{K,\chi}M_q\|h\|_{\mathrm L^2(S,\mu)}.
\]
Consequently, the pairing \(\mathscr E_U(q,h)\)
is well defined and continuous in the full form norm
for functions supported in \(K\).
No finite global energy of \(q\) is needed.

Now let \(f\) be a bounded finite-energy test function supported in \(K\), possibly meeting the boundary.
Choose strip cutoffs \(\zeta_j\) for
\(K\cap\partial\Om\) as above.
The product estimate gives
\[
 \mathscr E_U(f\zeta_j)
 \le
 C\|f\|_\infty^2\mathscr E_U(\zeta_j)
 +
 C\int_S\zeta_j^2\Gamma_U(f)\,\mathrm d\mu
 \longrightarrow0.
\]
Here \(\Gamma_U(f)\) is defined by the same integral
as for Lipschitz functions, and
\[
 \int_S\Gamma_U(f)\,\mathrm d\mu
 =2\mathscr E_U(f)<\infty.
\]
The second term tends to zero by dominated convergence,
as does \(\|f\zeta_j\|_{\mathrm L^2(S,\mu)}\).
Thus \(f_j=f(1-\zeta_j)\to f\) in the full form norm.
Each \(f_j\) has finite energy and compact support
strictly inside \(D_{LR}(w)\).
The assumed weak equation therefore gives
\(\mathscr E_U(q,f_j)=0\).
Continuity of the pairing on \(K\) yields
\(\mathscr E_U(q,f)=0\).

For an arbitrary compactly supported finite-energy
test \(f\), put
\[
 T_N(f)=\max\{-N,\min\{f,N\}\}.
\]
These truncations have the same support restriction
and converge to \(f\) in \(\mathrm L^2(S,\mu)\).
Moreover,
\[
 |\Delta_{xy}(f-T_N(f))|
 \le|\Delta_{xy}f|,
\]
and the left-hand side tends to zero almost everywhere.
Dominated convergence in the pair integral therefore
gives convergence in energy.
Applying the bounded-test conclusion to \(T_N(f)\)
and passing to the limit proves the asserted weak
identity for \(f\).

It remains to prove the asserted approximation of \(q\eta\) by
compactly supported Lipschitz functions and the resulting
local form-domain membership at the adjoined boundary. The truncation and strip-removal construction just
given approximates every compactly supported
finite-energy function by functions with compact
support strictly inside \(\Om\).
For such an interior support, a finite Whitney
covering and the lower comparison in
\eqref{eq:whitney-pair-energy}, together with a smooth
partition of unity, show that the zero extension
belongs to \(\mathrm H^{\sigma_*/2}(\R^n)\).
On these fixed interior sets, \(U\) is bounded above
and below by positive constants. Conversely, the upper kernel bound controls close
pairs by the fractional
\(\mathrm H^{\sigma_*/2}\) seminorm.
The remaining pairs are separated and are controlled
by the moment estimate.
Thus, for functions \(g\) supported in a fixed
compact interior neighborhood,
\[
 \|g\|_{\mathrm L^2(S,\mu)}^2+\mathscr E_U(g)
 \le
 C\left(
 \|g\|_{\mathrm L^2(\R^n)}^2
 +[g]_{\mathrm H^{\sigma_*/2}(\R^n)}^2
 \right).
\]
Smooth approximation inside that neighborhood
therefore converges in the full form norm.
Combining these approximations proves the required
closure by compactly supported Lipschitz functions.
Their supports may be kept in any fixed relatively
open neighborhood of the original compact support
within the controlled patch.
In particular, this applies to \(q\eta\), so \(q\)
belongs locally to the form domain also at the
adjoined boundary.
\end{proof}

\begin{lem}\label{lem:bounded-ground-state}
Assume the hypotheses of
Lemma~\ref{lem:variable-volume-scale}.
Let \(D\subset\R^n\) be a bounded open set with
\(D_{LR}(w)\subset D\), and suppose in addition that
\(U\in\Vs(D)\).
Let \(0\le M_q<\infty\), and let \(F\) be measurable
on \(\R^n\), with
\[
 0\le F\le M_qU\quad\text{almost everywhere in }\R^n,
 \qquad
 \iint_{Q_D}|F(x)-F(y)|^2K_U(x,y)
 \,\mathrm d y\,\mathrm d x<\infty,
\]
where \(Q_D=\R^{2n}\setminus(D^c\times D^c)\). Suppose that \(\mathcal L_UF=0\) weakly in \(D\),
with the weak identity valid for test functions in
\(\Wzero(D)\).
 Then \(q=F/U\) on \(S=\{U>0\}\) satisfies the
hypotheses and conclusions of
Proposition~\ref{prop:bounded-quotient-removal}.
In particular, its transformed weak equation admits
compactly supported finite-energy test functions whose
supports meet the adjoined boundary.
\end{lem}

\begin{proof}
The domination assumption gives
\(0\le q\le M_q\) on \(S\) and
\(F=0\) almost everywhere on \(\{U=0\}\).
We first verify the local transformed energy of \(q\).
The additional assumption \(U\in\Vs(D)\) gives
\[
 \iint_{Q_D}|U(x)-U(y)|^2K_U(x,y)
 \,\mathrm d y\,\mathrm d x
 =
 (p-1)\iint_{Q_D}
 \frac{|U(x)-U(y)|^p}{|x-y|^{n+sp}}
 \,\mathrm d y\,\mathrm d x
 <\infty.
\]
Thus both \(U\) and \(F\) have finite energy on
\(Q_D\) for the original kernel \(K_U\).  Let \(V\Subset D_{LR}(w)\).
By Proposition~\ref{prop:interior}, \(U\) is continuous
and strictly positive in \(D_{LR}(w)\).
Compactness of \(\overline V\) therefore gives $0<m_V\le U(x)\le M_V<\infty$ for $x\in\overline V$. The identity
\[
 U(x)(q(x)-q(y))
 =
 F(x)-F(y)-q(y)(U(x)-U(y))
\]
therefore gives
\[
 U(x)U(y)|q(x)-q(y)|^2
 \le
 C_V\bigl(
 |F(x)-F(y)|^2
 +M_q^2|U(x)-U(y)|^2
 \bigr)
 \qquad(x,y\in V).
\]
Multiplication by \(K_U(x,y)\) and integration show
that \(\mathscr E_{U,V}(q)<\infty\).
Moreover, localization by an interior compactly
supported Lipschitz cutoff gives finite full
transformed energy.
Indeed, local energy controls nearby pairs, while
boundedness of \(q\) and
Proposition~\ref{prop:variable-moment} control the
separated exterior interactions.

We next derive the quotient equation for a smooth
test \(\phi\in C_0^\infty(D_{LR}(w))\).
Since \(U\) solves the nonlinear equation in this
patch, homogeneity gives
\[
 \mathcal L_UU=(p-1)\Fsp U=0
 \quad\text{in }D_{LR}(w).
\]
Only this local equation for \(U\) is used.
No fractional $p$-harmonicity of \(U\) throughout the auxiliary
domain \(D\) is required. Compactly supported multiplication gives
\(U\phi\in\Wzero(D)\), so \(U\phi\) is an admissible
test in the equation for \(F\).
To test the equation for \(U\) with \(F\phi\), we
approximate \(F\phi\) in the original \(K_U\)-energy.
Membership of \(F\phi\) in \(\Wzero(D)\) is not needed. Choose open sets \(O_0,O_1\) such that
\[
 \operatorname{supp}\phi\subset O_0\Subset O_1
 \Subset D_{LR}(w).
\]
The domination \(F\le M_qU\) gives local boundedness
and hence local \(\mathrm L^2\) integrability of \(F\).
Cover \(\operatorname{supp}\phi\) by finitely many
interior Whitney balls whose fixed enlargements
lie in \(O_1\).
Proposition~\ref{prop:whitney-pair}, applied with
both backgrounds equal to \(U\), converts the local
\(K_U\)-energy bound for \(F\) into
\(\mathrm H^{\sigma_*/2}\) regularity on the smaller
balls.
A smooth partition of unity then shows that the
zero extension of \(F\phi\) belongs to
\(\mathrm H^{\sigma_*/2}(\R^n)\).
It can therefore be approximated in that space by
smooth functions supported in one fixed compact
subset of \(O_1\).

We next verify that this approximation also converges
in the full \(K_U\)-energy.
For close pairs in \(O_1\), the local Lipschitz
bound for \(U\) gives
\[
 K_U(x,y)\le C|x-y|^{-n-\sigma_*}.
\]
For the remaining pairs, positive separation gives
\[
 \sup_{x\in K}
 \int_{\R^n\setminus O_1}K_U(x,y)\,\mathrm d y
 <\infty
\]
whenever \(K\Subset O_1\).
For \(p>2\), use
\[
 |U(x)-U(y)|^{p-2}
 \le C_p\bigl(U(x)^{p-2}+U(y)^{p-2}\bigr).
\]
The first term is controlled by local boundedness
of \(U\) and separation.
For the second, H\"older's inequality gives
\begin{align*}
 &\int_{\R^n\setminus O_1}
 \frac{U(y)^{p-2}}{|x-y|^{n+sp}}\,\mathrm d y\le
 \left(
 \int_{\R^n\setminus O_1}
 \frac{U(y)^{p-1}}{|x-y|^{n+sp}}\,\mathrm d y
 \right)^{(p-2)/(p-1)}
 \left(
 \int_{\R^n\setminus O_1}
 \frac{\mathrm d y}{|x-y|^{n+sp}}
 \right)^{1/(p-1)}.
\end{align*}
Both factors are uniformly finite for \(x\in K\),
by the natural tail assumption and the positive
distance between \(K\) and \(\R^n\setminus O_1\).
For \(p=2\), the assertion follows directly from
the separated fractional kernel. The close interactions of the approximation error
are therefore controlled by its
\(\mathrm H^{\sigma_*/2}\) seminorm, and the separated
interactions by its squared \(\mathrm L^2\) norm.
We obtain smooth functions
\(\psi_j\in C_0^\infty(D_{LR}(w))\), with supports
in a fixed compact interior set, such that
\[
 \iint_{Q_D}
 \bigl|\Delta_{xy}(\psi_j-F\phi)\bigr|^2
 K_U(x,y)\,\mathrm d y\,\mathrm d x
 \longrightarrow0.
\]
These integrals include interactions with
\(\{U=0\}\), which remain present in the original
linearized form. The weak equation for \(U\) holds with each
\(\psi_j\) as a test function.
Its finite \(K_U\)-energy on \(Q_D\), established
above, allows passage to the limit by
Cauchy-Schwarz.
Hence \(F\phi\) is an admissible energy test in
that equation.

We may now subtract the equation for \(U\), tested
with \(F\phi\), from the equation for \(F\), tested
with \(U\phi\).
Proposition~\ref{prop:ground-state} gives
\[
 \frac12\iint_{S\times S}
 (q(x)-q(y))(\phi(x)-\phi(y))
 K_U(x,y)\,\mathrm d\mu(x)\,\mathrm d\mu(y)
 =0.
\]
The domination assumption ensures that the source
in \eqref{eq:quotient-source} vanishes, because
\(F=0\) on the entire zero set of \(U\).
Thus
\[
 \mathscr E_U(q,\phi)=0
 \qquad
 \bigl(\phi\in C_0^\infty(D_{LR}(w))\bigr).
\]

It remains to verify the full interior test function class
required by
Proposition~\ref{prop:bounded-quotient-removal}.
On compact interior sets, \(U\) is bounded above
and below by positive constants.
The same Whitney-ball approximation argument
therefore applies to the transformed form, whose
Lebesgue kernel is \(U(x)U(y)K_U(x,y)\).
Its separated interactions are controlled by
Proposition~\ref{prop:variable-moment}.
Consequently, every compactly supported interior
finite-energy test can be approximated by smooth
interior test functions in the full transformed form norm. The pairing with \(q\) is continuous under this
approximation.
To see this, localize \(q\) by a cutoff equal to
one near the common support of the test functions.
The localized part has finite transformed energy
and is handled by Cauchy-Schwarz.
The remaining part contributes only separated
interactions, controlled by \(q\le M_q\) and the
jump tail estimate.
The quotient equation therefore holds for every
compactly supported interior energy test function. Consequently, all hypotheses of
Proposition~\ref{prop:bounded-quotient-removal}
are now verified.
That proposition gives
\eqref{eq:quotient-cutoff-energy}, local membership
in the form domain at the adjoined boundary, and
the weak equation for the enlarged test function class.
\end{proof}

\subsection{Local Harnack inequalities with the full exterior tail}

We now formulate the local quadratic Harnack inequality
needed for \(\mathscr E_U\). Define
\[
 X=S\cup(\partial\Om\cap B_{\Lambda R}(w)),\qquad
 P=X\cap B_{LR}(w),\qquad J=K_U.
\]
Equip \(X\) with the Euclidean metric and the measure
\(\mathrm d\mu=U\,\mathrm d x\) on \(S\), assigning zero
measure to the adjoined boundary.
The resulting quadratic form is \(\mathscr E_U\).
Functions on \(S\) are regarded as \(\mu\)-equivalence
classes on \(X\), so adjoining the boundary imposes
no condition on their boundary values. Our choice of
representative gives
\[
 X\cap B_{\Lambda R}(w)=\overline\Om\cap B_{\Lambda R}(w).
\]
In particular, \(P=X\cap B_{LR}(w)=\overline\Om\cap B_{LR}(w)\)
is open in the relative topology of \(X\).
The allowed relative balls are metric balls of \(X\),
including those centered at boundary points.
Although \(X\) need not be complete, the displayed
identity and \(L<\Lambda\) imply that every closed
metric ball contained in \(P\) is compact.
Moreover, \(\mu\) is Radon and finite on bounded sets,
with full support in \(P\).
These are the local topological and measure properties
required in Lemma~\ref{lem:local-form-harnack}. The space \(X\) retains the entire positivity set \(S\),
so the form includes all distant exterior interactions.
The volume and structural estimates, however, are
available only in \(P\).
In particular, \(\mu\) need not be globally doubling,
and the global theorem of \cite{CY} cannot be applied
directly.

Lemma \ref{lem:local-form-harnack} below is stated in an abstract form to
identify the local hypotheses and the estimates needed
to control the exterior terms. To formulate the lemma, let \((X,d,\mu)\) be a metric measure space, and let
\(P\subset X\) be open in the metric topology of \(X\).

Suppose that \(\mu\) is a Radon measure on \(X\) and that every nonempty relatively open subset of \(P\)
has positive \(\mu\)-measure. Assume that $\mu$ is finite on bounded subsets of \(P\), and that
every closed metric ball contained in \(P\) is compact. All supports, open sets and closures are understood relative to \(X\).
Thus \(G\Subset P\) means that \(\overline G^{\,X}\)
is compact and contained in \(P\).
In our application, such sets may contain points of
\(\partial\Om\). Write \(\mathbb B_\rho(x)=B_X(x,\rho)\). Let \(J\) be a measurable, symmetric, nonnegative kernel,
and define
\[
 \mathscr E(f,g)=\frac12\iint_{X\times X}
 (f(x)-f(y))(g(x)-g(y))J(x,y)
 \,\mathrm d\mu(x)\,\mathrm d\mu(y).
\]
We write \(\mathscr E_G\) when both integration variables
are restricted to \(G\).
Assume that every Lipschitz function on \(X\) with compact
support in \(P\) has finite energy.
For each open \(G\Subset P\), let \(\mathscr D_0(G)\)
be the closure of Lipschitz functions on \(X\) with
compact support in \(G\), in the norm
\[
 \bigl(\|f\|_{\mathrm L^2(X,\mu)}^2
       +\mathscr E(f,f)\bigr)^{1/2}.
\]
Thus the energy defining \(\mathscr D_0(G)\) includes
interactions between \(G\) and \(X\setminus G\). A bounded measurable function \(h\) belongs locally to
the form domain in \(P\) if $\eta h\in\mathscr D_0(G)$ for every open \(G\Subset P\) and every Lipschitz function
\(\eta\) on \(X\) with compact support in \(G\).
Such a function is harmonic in an open set \(O\subset P\)
if
\[
 \mathscr E(h,\phi)=0
 \quad\text{for every }\phi\in\mathscr D_0(G)
 \text{ and every open }G\Subset O.
\]

We are now ready to state the abstract lemma. Note that only the first endpoint
of an exterior interaction is required to lie in the
controlled patch. The second may lie anywhere in \(X\).

\begin{lem}\label{lem:local-form-harnack}
Let \((X,d,\mu)\), \(P\), \(J\), and \(\mathscr E\)
be as above, with the form domains and notion of
harmonicity just defined.
Suppose in addition that a positive measurable scale function
\(\Psi(x,r)\) is given on the local balls under
consideration. On balls whose required fixed enlargements lie in \(P\),
assume uniform volume doubling and reverse volume doubling,
the scale growth and nearby-center comparisons in
\eqref{eq:variable-weak-scale}, and the Poincar\'e
inequality \eqref{eq:variable-poincare}.
Assume also the full jump moment estimate
\eqref{eq:variable-moment} and the averaging estimate
\eqref{eq:variable-ujs}.
These estimates are understood with \(X\), \(J\), and
\(\mathscr E\) in place of \(S\), \(K_U\), and
\(\mathscr E_U\), respectively. Centers lie in \(P\),
distances and balls are those of \(X\), and the original
patch restrictions are replaced by containment of the
required fixed enlargements in \(P\).
In particular, the jump moment estimate includes
interactions with all of \(X\), and the second endpoint
in the averaging estimate may lie anywhere in \(X\),
subject to its distance condition.
All constants and fixed enlargement factors in these
assumptions are uniform. Under these assumptions there exist \(C_0\ge8\), \(H\ge1\), and
\(\delta_H\in(0,1)\), depending only on these structural
constants and enlargement factors, with the following
property.
Let \(z\in P\) and \(R>0\) satisfy $\overline{\mathbb B_{C_0R}(z)}\subset P$. Suppose that \(h\) is bounded on \(X\), belongs locally
to the form domain in \(P\), and is harmonic and
nonnegative in \(\mathbb B_R(z)\).
Then
\begin{equation}\label{eq:local-form-harnack}
 \esssup_{\mathbb B_\ell(z)}h
 \le H\left(
 \essinf_{\mathbb B_\ell(z)}h
 +\Psi(z,\ell)\mathcal T_{z,R}(h_-)
 \right),
 \qquad 0<\ell\le\delta_H R,
\end{equation}
where \(h_-=\max\{-h,0\}\) and
\begin{equation}\label{eq:variable-tail}
 \mathcal T_{z,R}(f)=
 \esssup_{\xi\in\mathbb B_{R/2}(z)}
 \int_{X\setminus\mathbb B_R(z)}
 |f(y)|J(\xi,y)\,\mathrm d\mu(y).
\end{equation}
All essential suprema and infima are taken with respect
to \(\mu\).
The constants are independent of the global bound on \(h\).
\end{lem}

The proof of Lemma~\ref{lem:local-form-harnack} is given in
Appendix~\ref{app:local-form-harnack}. The same hypotheses as in Lemma~\ref{lem:local-form-harnack} also give the following oscillation estimate,
normalized by a local supremum. We state it separately for
use with both the half-space and variable backgrounds.

\begin{cor}\label{cor:local-form-holder}
Assume the structural hypotheses and notation of
Lemma~\ref{lem:local-form-harnack}.
Then there are constants \(C_*\ge2C_0\), \(C\ge1\), and
\(\alpha\in(0,1)\), depending only on the structural data,
with the following property.
Let \(z\in P\) and \(r>0\) satisfy
\(\overline{\mathbb B_{C_*r}(z)}\subset P\).
Suppose that \(h\) is bounded and nonnegative on \(X\),
belongs locally to the form domain in \(P\), and is
harmonic in \(\mathbb B_{2r}(z)\).
Put \(Q=\esssup_{\mathbb B_{2r}(z)}h\).
Then \(h\) has a H\"older continuous representative on
\(\mathbb B_r(z)\), and
\begin{equation}\label{eq:local-form-holder}
 |h(x)-h(y)|\le CQ
 \left(\frac{d(x,y)}r\right)^\alpha
 \qquad(x,y\in\mathbb B_r(z)).
\end{equation}
The constants are independent of the global bound on \(h\).
\end{cor}

The proof of Corollary \ref{cor:local-form-holder} is given at the end of
Appendix~\ref{app:local-form-harnack}.

\begin{rem}\label{rem:local-harnack-scope}
These results require volume doubling only in \(P\),
but retain all exterior interactions through the moment
and averaging hypotheses. They apply to bounded functions
with constants independent of their global bound.
The passage to possibly unbounded nonnegative
\(q\in\mathscr D_h+\R\) in the flat model uses
Lemma~\ref{lem:model-harmonic-approximation} separately.
\end{rem}

\begin{thm}\label{thm:variable-linear-bhi}
Assume \eqref{eq:range}, and fix the flatness threshold
and the parameters \(\kappa_*,L,\Lambda\) as in
Subsection~\ref{subprelim}.
Let \(\Om\) be \((\delta,r_0)\)-Reifenberg flat with
\(\delta\) below this threshold.
Let \(w\in\partial\Om\), \(R>0\), and \(\Lambda R<r_0\),
and let \(D\) be a bounded open set containing \(D_{LR}(w)\).
Suppose that $U\in\Vs(D)\cap\mathrm L^{p-1}_{sp}(\R^n)$ is nontrivial, globally nonnegative, and satisfies
\[
 \Fsp U=0\quad\text{in }D_{\Lambda R}(w),
 \qquad
 U=0\quad\text{on }B_{\Lambda R}(w)\setminus\Om
\]
in the sense of Definition~\ref{def:local-zero}.
Let \(F\) be measurable and satisfy
\[
 0\le F\le M_qU\quad\text{almost everywhere on }\R^n,
 \qquad
 \iint_{Q_D}|F(x)-F(y)|^2K_U(x,y)
 \,\mathrm d y\,\mathrm d x<\infty
\]
for some \(0\le M_q<\infty\).
Suppose also that \(\mathcal L_UF=0\) weakly in \(D\),
with test functions in \(\Wzero(D)\). Then \(q=F/U\) has a H\"older continuous representative
on \(\mathbb B_R(w)=\overline\Om\cap B_R(w)\).
With \(A=A_R(w)\), this representative satisfies
\begin{align}
 C^{-1}q(A)&\le q(x)\le Cq(A)
 &&\text{for }x\in\mathbb B_R(w),
 \label{eq:variable-linear-comparison}\\
 |q(x)-q(y)|
 &\le Cq(A)\left(\frac{|x-y|}{R}\right)^\alpha
 &&\text{for }x,y\in\mathbb B_R(w),
 \label{eq:variable-linear-holder}
\end{align}
where \(C\ge1\) and \(\alpha\in(0,1)\) depend only on
\(n,s,p\).
Pointwise values of \(F\) in \(D_R(w)\) are understood
through the representative \(F=Uq\), so that
\(q(A)=F(A)/U(A)\).
The constants are independent of \(M_q\).
\end{thm}

\begin{proof}
We use the space \(X\), patch \(P\), and kernel \(J=K_U\)
specified at the start of this subsection. The adjoined
boundary has zero measure, so the values initially
assigned to \(q\) there do not affect its energy or
weak equation.

The measure \(\mu=U\,\mathrm d x\) is locally finite,
since \(p\ge2\) and \(U\in\mathrm L^{p-1}_{\mathrm{loc}}(\R^n)\).
It has full support in \(P\), by interior positivity
and the corkscrew property.
Closed metric balls contained in \(P\) are compact.
Lemma~\ref{lem:variable-volume-scale} gives the local
volume and scale comparisons.
Propositions~\ref{prop:variable-poincare},
\ref{prop:variable-moment}, and \ref{prop:variable-ujs}
give Poincar\'e, the full jump moment bound, and kernel
averaging, respectively. The moment bound, symmetry,
and local scale comparison also give finite energy for
every Lipschitz function compactly supported in \(P\).
Finally, Lemma~\ref{lem:bounded-ground-state} and
Proposition~\ref{prop:bounded-quotient-removal} show that
\(q\) belongs locally to the form domain and satisfies
the quotient equation for the required energy test functions,
including those whose supports meet the adjoined boundary.
These results verify the hypotheses of
Lemma~\ref{lem:local-form-harnack}.

Let \(C_0\), \(H\), and \(\delta_H\) be the constants
in Lemma~\ref{lem:local-form-harnack}. Together with \(C_*\) in Corollary~\ref{cor:local-form-holder},
they depend only on the structural data. The initial
choice of \(L\) can therefore include
\[
 \kappa_*\max\{C_0R_H,C_*R\}<LR,
 \qquad R_H={3R}/{\delta_H}.
\]
Apply Lemma~\ref{lem:local-form-harnack} at \(w\), with outer
radius \(R_H\) and inner radius \(3R\).
The quotient is nonnegative on all of \(S\), so its
negative tail vanishes.
Consequently,
\[
 Q:=\esssup_{D_{3R}(w)}q
 \le H\essinf_{D_{3R}(w)}q.
\]
If \(Q=0\), the zero representative on
\(\mathbb B_R(w)\) gives both conclusions.
We henceforth assume \(Q>0\).

Apply Corollary~\ref{cor:local-form-holder} at \(w\)
with radius \(R\). Its local normalization is at most
\(Q\), since \(\mathbb B_{2R}(w)\subset\mathbb B_{3R}(w)\).
It gives a H\"older continuous representative on
\(\mathbb B_R(w)\) and
\[
 |q(x)-q(y)|\le CQ
 \left(\frac{|x-y|}{R}\right)^\alpha
 \qquad(x,y\in\mathbb B_R(w)).
\]
The initial essential Harnack inequality, continuity,
and full support imply
\[
 Q/H\le q(x)\le Q
 \qquad(x\in\mathbb B_R(w)).
\]
In particular, \(q(A)\le Q\le Hq(A)\).
Substitution proves \eqref{eq:variable-linear-comparison}
and \eqref{eq:variable-linear-holder}, with constants
depending only on \(n,s,p\), independently of \(M_q\).
\end{proof}

\begin{rem} The global bound \(0\le F\le M_qU\) ensures that the
quotient is bounded and that \(F\) vanishes on the
entire zero set of \(U\).
These properties justify the boundary energy estimates,
the cutoff limits, and the homogeneous quotient equation.
The comparison and H\"older constants are independent
of the size of \(M_q\).
\end{rem}

\begin{rem} Theorem~\ref{thm:variable-linear-bhi} does not cover an
arbitrary nonnegative linearized solution satisfying
only a local exterior zero condition.
Its quotient may be unbounded, and values of \(F\)
on the distant zero set of \(U\) may produce the source
in \eqref{eq:quotient-source}.
The ordered parameter derivatives considered in the
next section satisfy the global domination required here.
\end{rem}

\section{Ordered deformations and boundary comparison}
\label{sec:completion}

In this section we complete the proof of
Theorem~\ref{thm:main}. We first establish boundary
comparison for solutions in the variational energy
class, using the linear comparison estimates from
Section~\ref{sec:variable}. We then pass to the natural
local energy and tail class by approximation.

The first step concerns solutions with ordered exterior
data. We join these data by an affine path and consider
the corresponding family of variational solutions.
Without assuming that this family is continuously
differentiable in \(\Vs(D)\), we construct weak parameter
derivatives and show that they satisfy the linearized
equation. Order and homogeneity provide the domination
needed to apply the bounded-quotient estimates of the
preceding section. Since the quotient of a parameter
derivative by the corresponding solution is the
parameter derivative of its logarithm, integration
along the path yields boundary comparison and
H\"older estimates for the ratio of the endpoint
solutions.

To compare solutions with unrelated exterior values,
we construct a common harmonic majorant in a smaller
domain, prescribing the maximum of their exterior
data. A supersolution estimate controls the
normalization of this majorant by interior values
of the original solutions, while retaining their
distant exterior data. Applying the ordered comparison
to each solution and the majorant proves the
finite-energy result in
Theorem~\ref{thm:intro-finite}.

Finally, we approximate the distant exterior data and
perform fractional $p$-harmonic replacement. The approximating
solutions preserve the local weak zero condition and
converge both locally uniformly and in the natural
tail norm. Because the comparison constants are
independent of the energy norms, no uniform bound
on the full \(p\)-energies is needed to pass to the
limit and obtain Theorem~\ref{thm:main}.

\subsection{Differentiation of an ordered variational family}

The energy continuity in Proposition~\ref{prop:path-holder} is not
by itself a differentiability theorem. Ordering supplies an additional
comparison with scalar multiples, and this gives a direct weak
differentiation argument.

\begin{prop}\label{prop:ordered-derivative}
Let \(D\subset\mathbb R^n\) be bounded, \(p\ge2\), and \(g_0,g_1\in\Vs(D)\)
with \(0\le g_0\le g_1\) almost everywhere on \(\R^n\).
Let \(U_t\) be the variational solution in \(D\) with data
\(g_t=(1-t)g_0+tg_1\). Then
\begin{equation}\label{eq:ordered-parameter-bound}
 U_\tau\le U_t\le(t/\tau)U_\tau
 \quad(0<\tau<t\le1).
\end{equation}
Moreover, there is a jointly measurable function \(F_t(x)\), defined for almost
every \((t,x)\), which is the parameter derivative and satisfies
\begin{equation}\label{eq:ordered-derivative-bound}
 0\le F_t\le U_t/t\quad\hbox{a.e. on }\R^n,
 \qquad F_t=g_1-g_0\quad\hbox{on }D^c.
\end{equation}
For almost every \(t\in(0,1)\),
\begin{equation}\label{eq:ordered-derivative-equation}
 \iint_{Q_D}|\Delta F_t|^2K_{U_t}\dxdy<\infty,
 \qquad \mathcal B_{U_t}(F_t,\phi)=0
 \quad(\phi\in\Wzero(D)).
\end{equation}
The exceptional set in this last assertion is independent of \(\phi\).
\end{prop}

\begin{proof}
On \(D^c\), ordering gives
\(g_\tau\le g_t\le(t/\tau)g_\tau\).
Homogeneity and the comparison principle prove
\eqref{eq:ordered-parameter-bound} in \(D\).
For interior continuous representatives, the inequalities hold
pointwise. Thus \(t\mapsto U_t(x)\) is locally Lipschitz on
\((0,1)\) for every such \(x\), and for almost every exterior
\(x\). Difference quotients are jointly measurable.
Rademacher's theorem and Fubini give \(F_t\) and
\eqref{eq:ordered-derivative-bound}, on a common full-measure set
of parameters for almost every \(x\).

Fix a parameter in this set and let \(h>0\) tend to zero. Put
\[
 F_h=(U_{t+h}-U_t)/h,\qquad H=g_1-g_0,
 \qquad K_{t,h}=K_{U_t,U_{t+h}}.
\]
The exact difference equation and the test
\(F_h-H\in\Wzero(D)\) give
\begin{equation}\label{eq:derivative-energy-bound}
 \iint_{Q_D}|\Delta F_h|^2K_{t,h}\dxdy
 \le\iint_{Q_D}|\Delta H|^2K_{t,h}\dxdy\le C_g.
\end{equation}
The last bound follows by H\"older's inequality and the uniform
\(p\)-energy bounds for \(U_t\). The constant may depend on the
data, but not on \(h\).
Lemma~\ref{lem:scalar} implies \(K_{U_t}\le C_pK_{t,h}\).
Also \(0\le F_h\le U_t/t\), and \(F_h\to F_t\) almost
everywhere. Since \(D\) is bounded and \(p\ge2\), dominated
convergence gives convergence in \(L^2(D)\). Fatou's lemma and
weak compactness identify \(\Delta F_t\) as the weak limit in
\(L^2(Q_D,K_{U_t}\dxdy)\), and give its finite energy.

It remains to pass the varying coefficient to the limit. Use
\(\d\nu=|x-y|^{-n-sp}\dxdy\) on \(Q_D\), and write
\[
 a_h=a_p(\Delta U_t,\Delta U_{t+h}),\qquad
 a=(p-1)|\Delta U_t|^{p-2}.
\]
Proposition~\ref{prop:path-holder} gives strong convergence of
\(\Delta U_{t+h}\) to \(\Delta U_t\) in \(L^p(\nu)\).
For \(p>2\), continuity of the power map therefore gives
\(a_h\to a\) in \(L^{p/(p-2)}(\nu)\).
Since \(a\le C_pa_h\),
\[
 {|a_h-a|^2}/{a_h}\le C_p|a_h-a|,
\]
with zero assigned to the ratio when both coefficients vanish.
For every \(\phi\in\Wzero(D)\), Cauchy-Schwarz and
\eqref{eq:derivative-energy-bound} imply
\begin{align*}
 \left|\int_{Q_D}(a_h-a)\Delta F_h\Delta\phi\d\nu\right|
 &\le C_g^{1/2}
 \left(\int_{Q_D}\frac{|a_h-a|^2}{a_h}|\Delta\phi|^2\d\nu
 \right)^{1/2}\\
 &\le C C_g^{1/2}
 \|a_h-a\|_{\mathrm L^{p/(p-2)}(\nu)}^{1/2}
 \|\Delta\phi\|_{\mathrm L^p(\nu)}\longrightarrow0.
\end{align*}
The fixed-coefficient term converges by weighted weak convergence.
Passing to the limit in the equation for \(F_h\) proves
\eqref{eq:ordered-derivative-equation}. For \(p=2\) the coefficients
are constant and this step is immediate.
The chosen parameter set came only from almost-everywhere convergence
in \(x\), and the convergence argument applies to every test at
such a parameter. No test-dependent exceptional set is needed.
\end{proof}

\subsection{Pointwise integration and ordered comparison}

\begin{thm}\label{ass:paths}
Assume \eqref{eq:range}, and fix the buffer \(\Lambda\) and the
flatness threshold from Section~\ref{sec:variable}.
Let \(\Om\) be \((\delta,r_0)\)-Reifenberg flat with \(\delta\)
below this threshold, let \(w\in\partial\Om\), and suppose that
\(\Lambda r<r_0\). Set
\[
 D=D_{\Lambda r}(w),\qquad E=D_r(w),\qquad A=A_r(w).
\]
Suppose that \(U_0,U_1\in\Vs(D)\) are nontrivial, globally
nonnegative weak solutions of \(\Fsp U_i=0\) in \(D\), \(i=0,1\),
that \(U_0\le U_1\) on \(\R^n\), and that both have zero exterior values on
\(B_{\Lambda r}(w)\setminus\Om\) in the sense of
Definition~\ref{def:local-zero}.
Let \(U_t\) be the variational solution in \(D\) with exterior
data \((1-t)U_0+tU_1\). For almost every \(t\in(0,1)\), the function
\(F_t\) of Proposition~\ref{prop:ordered-derivative} has a continuous
representative on \(E\), and \(F_t/U_t\) satisfies
\eqref{eq:variable-linear-comparison} and
\eqref{eq:variable-linear-holder}, with \(R=r\) and constants
depending only on \(n,s,p\), independently of \(t\) and the energy
norms of the endpoint solutions.
For every \(x\in E\), the map \(t\mapsto U_t(x)\) is absolutely
continuous on \([0,1]\), and
\begin{equation}\label{eq:path-pointwise-ftc}
 U_b(x)-U_a(x)=\int_a^b F_t(x)\d t
 \quad(0\le a<b\le1).
\end{equation}
The exceptional set for its derivative on \((0,1)\) can be chosen
independently of \(x\in E\).
\end{thm}

\begin{proof}
Put \(g_t=(1-t)U_0+tU_1\).
Uniqueness in Theorem~\ref{thm:dirichlet} identifies the
variational solutions at \(t=0,1\) with the prescribed endpoints.
Since \(U_t-g_t\in\Wzero(D)\), we have
\(U_t\in\Ws(B_{\Lambda r}(w))\), and
\[
 \eta U_t=(1-t)\eta U_0+t\eta U_1+\eta(U_t-g_t)
 \in\Wzero(D)
 \qquad(\eta\in C_0^\infty(B_{\Lambda r}(w))).
\]
Also \(U_t=g_t=0\) almost everywhere on
\(B_{\Lambda r}(w)\setminus\Om\). Thus every \(U_t\)
has the local weak exterior zero condition.
Proposition~\ref{prop:ordered-derivative} supplies a common
full-measure set of parameters on which \(F_t\) satisfies
the weak linearized equation and \(0\le F_t\le U_t/t\).
Theorem~\ref{thm:variable-linear-bhi} applies there with
\(M_q=1/t\) and uniform constants. It gives continuous
representatives of \(F_t/U_t\), and hence of \(F_t\), in
\(E\). No regularity of \(F_t\) at possible critical points
outside this comparison patch is needed.

Fix \(t_0\in(0,1)\), an open set \(O\Subset E\), and
\(K=\overline O\). The ordered comparison gives
\[
 \|U_t-U_\tau\|_{C(K)}
 \le\frac{|t-\tau|}{t_0}\|U_1\|_{C(K)}
 \qquad(t,\tau\in[t_0,1]).
\]
In particular, \(U_t(x)\) is jointly continuous on
\([t_0,1]\times K\). To choose the continuous representatives
of \(F_t\) measurably, take limits of interior local averages
of the jointly measurable quotient \(F_t/U_t\) with respect
to Lebesgue measure in \(E\), multiply by
\(U_t\), and assign zero on the exceptional parameter set.
The resulting function is jointly measurable and satisfies
\(0\le F_t\le U_1/t_0\) on \(K\).
Since \(C(K)\) is separable and its Borel structure is
determined by evaluations on a countable dense subset,
\(t\mapsto F_t|_K\) is strongly measurable. Its norm is
bounded by \(\|U_1\|_{C(K)}/t_0\), so it is Bochner
integrable on \([t_0,1]\).
For almost every \(x\in O\), the scalar fundamental theorem
of calculus and Fubini give \eqref{eq:path-pointwise-ftc}
for \(t_0\le a<b\le1\). Both sides are continuous on
\(K\) and the right side is the evaluation of the Bochner
integral. The identity therefore holds throughout \(K\),
because a full-measure subset of \(O\) is dense in
\(K=\overline O\). It is consequently an identity in
\(C(K)\). Differentiating this Bochner integral gives
\(\partial_tU_t=F_t\) in \(C(K)\) for almost every
\(t\). A countable exhaustion \(O_j\Subset E\), together
with \(t_0\downarrow0\) through a sequence, produces one
exceptional set valid for every \(x\in E\).

It remains to include \(t=0\). Proposition~\ref{prop:path-holder}
gives convergence to \(U_0\) in energy and in \(L^p(D)\).
The bound \(U_t\le U_1\) controls both the local suprema
and the tails, so the interior estimates give equicontinuity
on compact subsets of \(E\). Hence \(U_t\to U_0\)
locally uniformly as \(t\downarrow0\).
Since \(F_t\ge0\), monotone convergence permits
\(a\downarrow0\) in \eqref{eq:path-pointwise-ftc} and gives
\[
 \int_0^b F_t(x)\d t=U_b(x)-U_0(x)<\infty
 \qquad(x\in E,\ 0<b\le1).
\]
This integral representation proves absolute continuity on
\([0,1]\) and the asserted identity at both endpoints.
\end{proof}

The preceding theorem provides the pointwise parameter
regularity needed to integrate logarithmic derivatives
in the comparison patch. This is sufficient for the
ordered comparison argument, without requiring
\(C^1\) dependence in \(\Vs(D)\).

\begin{prop}\label{prop:ordered-completion}
Under the hypotheses of Theorem~\ref{ass:paths}, there
are uniform constants \(C_L\ge1\) and
\(\alpha_L\in(0,1)\) such that, for all \(x,y\in E\),
\begin{align}
 \left|\log\frac{U_1(x)}{U_0(x)}
       -\log\frac{U_1(y)}{U_0(y)}\right|
 &\le
 C_L\log\frac{U_1(A)}{U_0(A)}
 \left(\frac{|x-y|}{r}\right)^{\alpha_L},
 \label{eq:ordered-holder}\\
 C_L^{-1}\log\frac{U_1(A)}{U_0(A)}
 \le \log\frac{U_1(x)}{U_0(x)}
 &\le
 C_L\log\frac{U_1(A)}{U_0(A)}.
 \label{eq:ordered-size}
\end{align}
The constants have the same structural dependence as
those in Theorem~\ref{thm:variable-linear-bhi}.
\end{prop}

\begin{proof}
By Theorem~\ref{ass:paths}, for almost every \(t\in(0,1)\),
the quotient \(q_t=F_t/U_t\) satisfies
\[
 C_L^{-1}a_t\le q_t(x)\le C_La_t,
 \qquad
 |q_t(x)-q_t(y)|\le C_La_t
       \left(\frac{|x-y|}{r}\right)^{\alpha_L},
 \qquad a_t=\frac{F_t(A)}{U_t(A)},
\]
for \(x,y\in E\), with constants independent of \(t\).
The same theorem gives absolute continuity of \(U_t(x)\)
on \([0,1]\). Since \(U_t(x)\ge U_0(x)>0\), the chain
rule yields
\begin{equation}\label{eq:log-integration}
 \log\frac{U_1(x)}{U_0(x)}
 =\int_0^1\frac{F_t(x)}{U_t(x)}\,\mathrm d t.
\end{equation}
In particular,
\[
 \int_0^1a_t\,\mathrm d t
 =\log\frac{U_1(A)}{U_0(A)}<\infty.
\]
This supplies an integrable normalization even near \(t=0\),
where the available domination constant is \(1/t\).
Integrating the two-sided estimate gives \eqref{eq:ordered-size}.
Integrating the oscillation estimate and using
\eqref{eq:log-integration} gives \eqref{eq:ordered-holder}.
\end{proof}

\subsection{Removal of ordering}

For two solutions with independent distant data, take a maximum lift
in a domain strictly inside their common fractional $p$-harmonicity patch. The
following comparison estimate controls the lift by the local size and
the entire distant tail of its data.

\begin{lem}\label{lem:buffered-supersolution}
Let \(p\ge2\), \(0<s<1\), and let \(D\) be a nonempty open
subset of \(B_R(z)\). Suppose that \(G\in\Vs(D)\), \(G\ge0\)
on \(\R^n\), and
\[
 L=\esssup_{B_{2R}(z)}G<\infty,\qquad
 T=\Tail(G,z,2R)<\infty.
\]
Let \(W\in G+\Wzero(D)\) be the variational solution in \(D\).
Then
\begin{equation}\label{eq:buffered-dirichlet-bound}
 0\le W\le C(n,s,p)(L+T)\quad\hbox{almost everywhere in }D.
\end{equation}
No geometric regularity of \(D\) is required.
\end{lem}

\begin{proof}
Choose \(H_*\ge L\), to be fixed below, and define
\begin{equation}\label{eq:constant-tail-barrier}
 \Phi_{H_*}(x)=
 \begin{cases}
 H_*,&x\in B_{2R}(z),\\
 G(x),&x\in\R^n\setminus B_{2R}(z).
 \end{cases}
\end{equation}
Then \(\Phi_{H_*}\ge G\) on \(\R^n\).
The function \(\Phi_{H_*}\) belongs to \(\Vs(D)\). Indeed, it is
constant on a neighborhood of \(D\), and the remaining cross
integral is finite by \(G\in\Vs(D)\), boundedness of \(G\)
on \(B_{2R}(z)\), and the separation of \(D\) from
\(\R^n\setminus B_{2R}(z)\).
Its possible jump on \(\partial B_{2R}(z)\) occurs a positive
distance from \(D\).

For \(a,b\ge0\),
\begin{equation}\label{eq:scalar-barrier-lower}
 J_p(a-b)\ge 2^{2-p}a^{p-1}-b^{p-1}.
\end{equation}
If \(b\le a\), this follows from
\(a^{p-1}\le2^{p-2}((a-b)^{p-1}+b^{p-1})\).
If \(b\ge a\), use
\(b^{p-1}\ge(b-a)^{p-1}+a^{p-1}\).
Thus, for \(x\in D\),
\begin{align}
 \Fsp\Phi_{H_*}(x)
 &=\int_{\R^n\setminus B_{2R}(z)}
       \frac{J_p(H_*-G(y))}{|x-y|^{n+sp}}\d y\notag\\
 &\ge 2^{2-p}H_*^{p-1}
       \int_{\R^n\setminus B_{2R}(z)}
                   \frac{\d y}{|x-y|^{n+sp}}
       -\int_{\R^n\setminus B_{2R}(z)}
                   \frac{G(y)^{p-1}}{|x-y|^{n+sp}}\d y\notag\\
 &\ge R^{-sp}\bigl(cH_*^{p-1}-CT^{p-1}\bigr).
 \label{eq:barrier-operator-lower}
\end{align}
For the positive term, integrate over \(\R^n\setminus B_{3R}(x)\),
which is contained in \(\R^n\setminus B_{2R}(z)\).
For the negative term use
\(|x-y|\ge|y-z|/2\).
All integrals are absolutely convergent, and hence the calculation
also gives the weak supersolution inequality.
Take \(H_*=C_*(L+T)\), with \(C_*\) depending only on \(n,s,p\),
large enough that \(H_*\ge L\) and the last line is nonnegative. For completeness, we note that the comparison test is admissible in the precise
variational class. Indeed, since \(W=G\le\Phi_{H_*}\) outside \(D\),
\(h=(W-\Phi_{H_*})_+\) has finite global \(W^{s,p}\) energy and
vanishes outside \(D\). Moreover,
\[
 0\le h\le(W-G)_+\in\Wzero(D).
\]
The order ideal property of \(\Wzero(D)\) gives
\(h\in\Wzero(D)\). Proposition~\ref{prop:comparison} therefore
yields \(W\le\Phi_{H_*}=H_*\) in \(D\).
Comparison with zero gives \(W\ge0\).
\end{proof}

\begin{prop}\label{prop:max-lift}
Let \(p\ge2\), \(0<s<1\), \(\Lambda\ge1\), and
\(R=\Lambda r\).
There exists \(\delta_{\max}=\delta_{\max}(n,s,p)>0\),
independent of \(\Lambda\), such that the following holds.
Suppose that \(\Om\) is \((\delta,r_0)\)-Reifenberg flat with
\(\delta\le\delta_{\max}\),
\(w\in\partial\Om\), and \(8R<r_0\).
Let \(u,v\) be nontrivial, globally nonnegative functions
satisfying
\[
 \Fsp u=\Fsp v=0\quad\text{in }D_{8R}(w),
 \qquad
 u=v=0\quad\text{on }B_{8R}(w)\setminus\Om
\]
in the local weak sense of
Definition~\ref{def:local-zero}.
Assume also that \(u,v\in\Vs(D_R(w))\).
Put
\[
 A=A_r(w),\qquad D=D_R(w),\qquad G=\max\{u,v\}.
\]
Let \(W\in G+\Wzero(D)\) be the variational solution of
\[
 \Fsp W=0\quad\text{in }D
\]
with exterior data \(G\).
Then $W$ satisfies
\begin{equation}\label{eq:proved-max-lift}
 W\ge u,v\quad\hbox{on }\R^n,\qquad
 \sup_D W\le C\Lambda^{sp/(p-1)}\bigl(u(A)+v(A)\bigr).
\end{equation}
The global comparison holds almost everywhere, and pointwise in
\(D\) for the continuous representatives.
Here \(C\) depends only on \(n,s,p\) and the uniform geometric
constants, independently of \(\Lambda\).
In particular, if \(u(A)=v(A)=1\) and \(\Lambda\)
is fixed, then \(1\le W(A)\le K_0\) with a uniform \(K_0\).
The function \(W\) has weak zero values on
\(B_R(w)\setminus\Om\).
\end{prop}

\begin{proof}
Choose \(\delta_{\max}\) no larger than the threshold in
Proposition~\ref{prop:boundary-normalization} with \(\beta=s/2\).
We first retain the normalization at \(A_r(w)\) while estimating
on the larger scale \(R\). For \(f=u\) or \(f=v\), choose
\(\rho=cr\) with \(B_{2\rho}(A)\subset D_{2r}(w)\).
Lemma~\ref{lem:positive-tail} and interior Harnack give
\[
 \int_{\R^n\setminus B_\rho(A)}
       \frac{f(y)^{p-1}}{|y-A|^{n+sp}}\d y
 \le Cr^{-sp}f(A)^{p-1},\qquad
 \sup_{B_\rho(A)}f\le Cf(A).
\]
On \(B_{4R}(w)\), use \(|y-A|\le5R\), apart from the small
ball handled by Harnack. Outside \(B_{4R}(w)\), compare
\(|y-A|\) with \(|y-w|\). This proves
\begin{align}
 N_R(f)&=R^{-n}\int_{B_{4R}(w)}f^{p-1}\d y
   +R^{sp}\int_{\R^n\setminus B_{4R}(w)}
             \frac{f(y)^{p-1}}{|y-w|^{n+sp}}\d y\le C\Lambda^{sp}f(A)^{p-1}.
 \label{eq:buffer-normalized-mass}
\end{align}
The boundary local boundedness estimate used in
Proposition~\ref{prop:boundary-normalization}, namely
\cite[Lemma~2.6]{BKS}, applied on \(B_{4R}(w)\), gives
\[
 \|f\|_{\mathrm L^\infty(B_{2R}(w))}
 \le C N_R(f)^{1/(p-1)}.
\]
Splitting the tail at \(2R\) into the annulus between \(2R\)
and \(4R\) and its complement also gives
\(\Tail(f,w,2R)^{p-1}\le C N_R(f)\).
Consequently,
\begin{equation}\label{eq:max-data-normalization}
 \|G\|_{\mathrm L^\infty(B_{2R}(w))}+\Tail(G,w,2R)
 \le C\Lambda^{sp/(p-1)}\bigl(u(A)+v(A)\bigr),
\end{equation}
since \(G^{p-1}\le u^{p-1}+v^{p-1}\). The maximum operation preserves \(\Vs(D)\), so the variational
solution exists by Theorem~\ref{thm:dirichlet}.
Lemma~\ref{lem:buffered-supersolution} and
\eqref{eq:max-data-normalization} give its upper bound.
Comparison with \(u\) and \(v\) gives the lower bounds.
For example,
\((u-W)_+\le(G-W)_+\in\Wzero(D)\), so the order ideal
property again justifies the comparison test.
Interior continuity makes the upper bound valid at \(A\). Finally, for \(\eta\in C_0^\infty(B_R(w))\),
\[
 \eta W=\eta G+\eta(W-G)\in\Wzero(D_R(w)).
\]
The first term belongs to this space by localization of the weak
zero conditions for \(u,v\) and stability under maxima.
The second belongs to it since \(W-G\in\Wzero(D)\).
The almost everywhere exterior zero condition is immediate from
\(W=G\) outside \(D\). This proves the final assertion.
\end{proof}

The larger fractional $p$-harmonicity patch in Proposition~\ref{prop:max-lift}
places a full annulus between the artificial boundary of the lift
domain and the uncontrolled distant data. It permits comparison in the lift domain while preserving the
original distant data.
Its proof uses positive-tail control and boundary local boundedness,
and does not use any boundary comparison theorem for quotients.

\begin{thm}\label{thm:completion}
Assume \eqref{eq:range}, sufficiently small Reifenberg flatness,
and the hypotheses of Proposition~\ref{prop:max-lift}, with
\(\Lambda\) fixed as in Section~\ref{sec:variable}.
Then, for \(E=D_r(w)\),
\begin{equation}\label{eq:completed-holder}
 \left|\log\frac{u(x)}{v(x)}-\log\frac{u(y)}{v(y)}\right|
 \le 2C_L\log(2K_0)
       \left(\frac{|x-y|}{r}\right)^{\alpha_L}
 \quad(x,y\in E).
\end{equation}
Moreover,
\[
 C^{-1}\frac{u(A)}{v(A)}\le\frac{u(x)}{v(x)}
 \le C\frac{u(A)}{v(A)}\quad(x\in E),
\]
where \(C\) depends only on \(n,s,p,K_0\).
\end{thm}

\begin{proof}
Normalize \(u(A)=v(A)=1\) by multiplying the solutions by
positive constants, and let \(W\) be the maximum lift supplied by
Proposition~\ref{prop:max-lift}. Its normalization constant \(K_0\)
depends only on \(n,s,p\), since \(\Lambda\) was fixed in terms
of these parameters. The functions \(u,v,W\) are all fractional $p$-harmonic
in the lift domain \(D_{\Lambda r}(w)\), and have the required
weak zero values there.
Apply Proposition~\ref{prop:ordered-completion} to
\((u,2W)\) and \((v,2W)\). Their logarithmic ratios at \(A\)
are at most \(\log(2K_0)\).
Subtract the two estimates, using
\[
 \log\frac{u}{v}=\log\frac{2W}{v}-\log\frac{2W}{u}.
\]
This proves \eqref{eq:completed-holder}. Set \(y=A\), use
\(|x-A|\le2r\), and exponentiate to get comparison.
Undoing the normalization gives the assertion.
\end{proof}

\begin{proof}[Proof of Theorem~\ref{thm:intro-finite}]
The preceding theorem has a fixed buffer \(\Lambda\), whereas the
introduction uses the radii \(r\) and \(4r\).
Take \(\rho=r/(16\Lambda)\).
For every boundary point within distance \(2\rho\) of
\(D_r(w)\), both the lift domain \(D_{\Lambda\rho}(z)\) and
its larger harmonicity patch \(D_{8\Lambda\rho}(z)\) lie in
\(D_{2r}(w)\). Indeed, \(8\Lambda\rho=r/2\) and
\(|z-w|<r+2\rho\).
Restriction therefore preserves the required finite energy.
The original equation and exterior zero condition hold throughout
the larger patch. Thus Proposition~\ref{prop:max-lift} and
Theorem~\ref{thm:completion} apply on \(D_\rho(z)\).
Set \(\theta=\rho/(8M_{\mathrm{geom}}r)\), and decrease the
flatness threshold so that \(\delta\le\theta/16\).
The points \(A_\rho(z)\) and \(A_r(w)\) belong to
\(D_{3r/2}(w)\) and have boundary distance at least \(\theta r\).
Remark~\ref{rem:localized-chains} therefore joins them by a
uniformly bounded chain whose doubled balls lie in \(D_{2r}(w)\).
Interior Harnack along this chain compares their values and gives
uniform ratio comparison near the boundary. The same argument
handles points of \(D_r(w)\) at boundary distance at least
\(\rho/8\). All constants depend only on the fixed parameters,
since \(\rho/r=1/(16\Lambda)\).

For two points separated by at most a sufficiently small multiple
of \(\rho\), either both lie in one such boundary comparison ball,
or their connecting segment lies in an interior ball of radius
comparable to \(\rho\). In the latter case interior regularity,
Harnack, and the normalized gradient upper estimate give
\(|\nabla\log u|+|\nabla\log v|\le C/\rho\).
Thus both cases give the required local H\"older bound.
For more widely separated points, the ratio comparison already
proved bounds the logarithmic oscillation, and the factor
\((|x-y|/r)^{\alpha_L}\) is bounded below by a fixed positive
constant. This gives \eqref{eq:target-holder} on all of \(D_r(w)\),
and completes the proof.
\end{proof}

\subsection{Passage to the natural tail class}

\begin{proof}[Proof of Theorem~\ref{thm:main}]
Let \(u,v\) satisfy its hypotheses, without assuming global
\(p\)-energy on pairs meeting the comparison domain.
Boundary local boundedness from \cite[Lemma~2.6]{BKS},
together with interior local boundedness and a finite covering
inside $B_{4r}(w)$, gives
\[
 \|u\|_{\mathrm L^\infty(B_{3r}(w))}
 +\|v\|_{\mathrm L^\infty(B_{3r}(w))}<\infty.
\]
This step uses only the local \(W^{s,p}\) regularity and natural
tails in the hypothesis. Apply
Lemma~\ref{lem:natural-tail-approximation} with \(R=3r\),
\(z=w\), and \(D=D_{3r}(w)\), separately to \(u\) and \(v\).
The required larger Sobolev ball is exactly \(B_{4r}(w)\).
We obtain globally nonnegative \(u_m,v_m\in\Vs(D_{3r}(w))\),
harmonic in \(D_{3r}(w)\), with weak zero values on
\(B_{3r}(w)\setminus\Om\), converging locally uniformly to
\(u,v\) there and converging in the natural tail norm.
In particular, for \(A=A_r(w)\),
\[
 u_m(A)\longrightarrow u(A)>0,\qquad
 v_m(A)\longrightarrow v(A)>0.
\]
Discard finitely many indices if necessary.

The approximants need not solve the equation outside
\(D_{3r}(w)\). To apply the finite-energy result, put
\(\rho=r/16\). For every boundary point \(z\) within
\(2\rho\) of \(D_r(w)\),
\[
 B_{4\rho}(z)\subset B_{3r}(w),\qquad
 u_m,v_m\in\Vs(D_{2\rho}(z)).
\]
Thus Theorem~\ref{thm:intro-finite} applies on \(D_\rho(z)\)
with constants independent of \(m\). We now use the covering
and chain argument in its proof, with this fixed ratio
\(\rho/r=1/16\). More precisely, take
\(\theta=\rho/(8M_{\mathrm{geom}}r)\) and
\(\delta\le\theta/16\). Remark~\ref{rem:localized-chains}
compares the values at \(A_\rho(z)\) and \(A=A_r(w)\)
through chains whose doubled balls lie in
\(D_{2r}(w)\subset D_{3r}(w)\). The same argument handles
points at boundary distance at least \(\rho/8\), and gives
\begin{equation}\label{eq:approximate-global-comparison}
 C^{-1}\frac{u_m(A)}{v_m(A)}
 \le\frac{u_m(x)}{v_m(x)}
 \le C\frac{u_m(A)}{v_m(A)}
 \quad(x\in D_r(w)).
\end{equation}
The two-point argument in that proof also gives
\[
 \left|\log\frac{u_m(x)}{v_m(x)}
              -\log\frac{u_m(y)}{v_m(y)}\right|
 \le C\left(\frac{|x-y|}{r}\right)^\alpha
 \quad(x,y\in D_r(w)),
\]
with \(C,\alpha\) independent of \(m\).
Indeed, when \(|x-y|<\rho/16\), either both points lie
in one boundary comparison ball, or their joining segment
has boundary distance at least \(\rho/16\), where the
normalized interior gradient estimate applies.
For larger separations, \eqref{eq:approximate-global-comparison}
bounds the logarithmic oscillation.

For each fixed pair of interior points, local uniform convergence
and positivity allow passage to the limit in this estimate and
in \eqref{eq:approximate-global-comparison}.
They give \eqref{eq:target-bhi} and \eqref{eq:target-holder}.
The latter also defines the continuous boundary extension of the
logarithmic ratio. No uniform \(p\)-energy bound for the
approximants enters this passage. This proves the theorem.
\end{proof}

\section{Concluding remarks and further research}\label{sec:concluding}

The proof separates geometric approximation from the structural
estimates for the linearized forms. This separation identifies several
extensions that follow from the same computations, and others that
require additional regularity or coercivity estimates. Throughout this
discussion, any use of Proposition~\ref{prop:C1} retains the range
\eqref{eq:range}.

\subsection{Weaker geometric assumptions}

The variational Dirichlet problem and
Lemma~\ref{lem:natural-tail-approximation} require no flatness
of the bounded comparison domain. Exterior capacity density
supplies continuity and positive boundary decay, while corkscrew
balls and Harnack chains supply the geometry for normalization.
Small Reifenberg flatness is used more substantially in
Theorem~\ref{thm:approximation} and
Corollary~\ref{cor:reifenberg-gradient} which provide the lower
bound in \(|\nabla U|\asymp U/d_\Om\) and cone increment
estimates needed for coercivity and averaging on Whitney balls.
Harnack and interior regularity alone give only the upper
gradient bound after tail normalization.

The geometry is needed only in a fixed larger boundary patch.
Bounded \(C^1\) domains are included at sufficiently small scales,
without a Dini assumption, as are Lipschitz graphs with small
Lipschitz constant. Extending the result to arbitrary Lipschitz or NTA domains
requires further arguments. Lipschitz regularity alone does
not imply the required small Reifenberg flatness, while
Reifenberg flat boundaries need not be locally Lipschitz graphs.

The boundary moment estimate also uses
\begin{equation}\label{eq:future-decay-moment}
 U(x)\le Ca_r\left(\frac{d_\Om(x)}r\right)^\beta,
 \qquad 2-sp+(p-1)\beta>0,
 \qquad a_r=U(A_r(z)).
\end{equation}
The near-\(s\) decay in \cite[Theorem~1.8]{BKS} permits this
choice in sufficiently flat domains. A positive exponent from
capacity density alone need not suffice. When \(sp\le2\),
the restriction is automatic for every \(\beta>0\), making
this subrange a natural starting point for geometric extensions.

Two parts of Section~\ref{sec:variable} require less geometry.
Under normalized boundary decay and the NTA conditions, one can
choose value-separated corkscrew balls at scales \(r\) and
\(\theta r\), with \(\theta>0\) fixed and small.
Decay gives the upper separation, and Harnack chains give positive
lower bounds and weighted measures comparable to \(a_r r^n\).
This proves the boundary two-set Poincar\'e argument without a
lower gradient bound, but does not give coercivity on arbitrarily
small Whitney balls. Finite energy also gives qualitative boundary
cutoffs directly, as follows.

\begin{lem}\label{lem:energy-boundary-capacity}
Assume \eqref{eq:range}, let \(\Om\) be open,
\(w\in\partial\Om\), \(R>0\), and put \(D=D_R(w)\).
Let \(U\in\Vs(D)\) be globally nonnegative, positive in \(D\),
and a local weak solution of \(\Fsp U=0\) in \(D\), with zero exterior values on
\(B_R(w)\setminus\Om\) in the sense of
Definition~\ref{def:local-zero}.
Put \(S=\{U>0\}\), \(\d\mu=U\d x\), and define
\(\mathscr E_U\) by \eqref{eq:variable-form}.
For every \(\eta\in C_0^\infty(B_R(w))\), \(0\le\eta\le1\),
there are compactly supported Lipschitz functions \(\zeta_j\),
\(0\le\zeta_j\le\eta\), equal to \(\eta\) in a neighborhood
of \(\partial\Om\cap\supp\eta\), such that
\begin{equation}\label{eq:energy-boundary-capacity}
 \|\zeta_j\|_{\mathrm L^2(\mu)}^2+\mathscr E_U(\zeta_j)
 \longrightarrow0.
\end{equation}
No geometric regularity of \(\partial\Om\) is assumed.
\end{lem}

\begin{proof}
Put \(W=\eta U\in\Wzero(D)\). Choose nonnegative
\(\phi_j\in C_0^\infty(D)\) converging to \(W\) in
\(\Ws(\R^n)\), and, after passing to a subsequence, almost
everywhere. Define
\[
 e_j=(W-\phi_j)_+,
 \qquad
 \zeta_j=(\eta-\phi_j/U)_+\quad\hbox{in }D,
 \qquad \zeta_j=\eta\quad\hbox{outside }D.
\]
The positive-part contraction gives \(e_j\to0\) in
\(\Ws(\R^n)\). On the compact support of \(\phi_j\),
interior regularity and positivity give \(U\in C^1\) and
\(\inf U>0\). Hence \(\phi_j/U\), extended by zero outside
\(D\), is Lipschitz, and so is \(\zeta_j\).
The asserted boundary equality follows from
\(\supp\phi_j\Subset D\). Also \(U\zeta_j=e_j\) on
\(\R^n\), using the local exterior zero condition.

The function \(U\zeta_j^2\), set to zero where \(U=0\),
belongs to \(\Wzero(D)\). Indeed, it equals
\(\eta e_j^2/W\), with zero assigned when \(W=0\), and
the map \((a,b)\mapsto a^2/b\) is Lipschitz on
\(0\le a\le b\), with its continuous value at \((0,0)\).
Composition and multiplication therefore preserve the required
zero-boundary space. Proposition~\ref{prop:homogeneity} allows
this test in the equation for \(U\). The pointwise ground-state
identity gives
\begin{align}
 \mathscr E_U(\zeta_j)
 &=\mathcal B_U(e_j,e_j)
       -\mathcal B_U(U,U\zeta_j^2)=\mathcal B_U(e_j,e_j)
 \le C E_{p,D}(U)^{(p-2)/p}[e_j]_{s,p}^2
 \longrightarrow0,\label{eq:capacity-ground-state-bound}
\end{align}
where
\[
 E_{p,D}(U)=\iint_{Q_D}
       \frac{|\Delta_{xy}U|^p}{|x-y|^{n+sp}}\dxdy<\infty.
\]
The last estimate is H\"older's inequality, and its energy factor
is interpreted as one when \(p=2\). All bilinear terms are
absolutely integrable by the same inequality.
Finally \(\zeta_j\to0\) almost everywhere with respect to
\(\mu\), and \(0\le\zeta_j\le\eta\).
Since \(\int_{\supp\eta}U\d x<\infty\), dominated convergence
gives the \(\mathrm L^2(\mu)\) conclusion.
\end{proof}

These cutoffs can replace those of
Lemma~\ref{lem:variable-boundary-capacity} in the proof of
Proposition~\ref{prop:bounded-quotient-removal}, under its
remaining hypotheses, and give an alternative capacity
argument for the ordered variational backgrounds.
We retain Lemma~\ref{lem:variable-boundary-capacity}
in the main proof because it gives a uniform rate and applies
without assuming \(U\in\Vs(D)\). Neither argument replaces
the jump moment or kernel averaging estimates.

An NTA extension would still require the structural and local
energy hypotheses of Lemma~\ref{lem:local-form-harnack}, together
with normalization and maximum-lift bounds, uniformly over the
ordered backgrounds. Under these inputs the quotient and deformation
arguments need no approximating hyperplane. General NTA geometry
alone has not been shown here to provide them. For comparison,
when \(p=2\), \cite{RSBHI} proves local boundary comparison in
arbitrary open sets and quotient H\"older continuity in Lipschitz
domains, for globally nonnegative continuous solutions with common
local exterior zero values.

\subsection{More general operators of \texorpdfstring{\(p\)}{p}-Laplace type} We here discuss extensions of the preceding analysis to more general
nonlocal operators of \(p\)-Laplace type. We indicate the additional
structural assumptions needed for the boundary comparison argument. Consider first the homogeneous class
\begin{equation}\label{eq:general-homogeneous-operator}
 \mathcal F_{a,s,p}u(x)
 =\operatorname{PV}\int_{\R^n}
 a(x,y)\frac{J_p(u(x)-u(y))}{|x-y|^{n+sp}}\d y,
\end{equation}
where \(a\) is measurable, independent of the solution, and
\begin{equation}\label{eq:general-kernel-ellipticity}
 a(x,y)=a(y,x),\qquad
 0<\lambda_0\le a(x,y)\le\Lambda_0<\infty
 \quad\hbox{almost everywhere}.
\end{equation}
For rough coefficients the symmetric weak formulation is primary.
Ellipticity and monotonicity preserve the variational, comparison,
Caccioppoli, positive-tail, distant-data approximation, and
buffered-supersolution arguments. Interior Harnack inequalities
for this class are available in \cite{DKP}.
The difference and derivative kernels are
\begin{equation}\label{eq:general-linearized-kernels}
 K_{a,u,v}=aK_{u,v},\qquad K_{a,U}=aK_U.
\end{equation}
For a fixed background, energy comparability preserves coercivity
and Poincar\'e estimates, and the upper ellipticity bound preserves
moment estimates. Kernel averaging is also stable as
\[
 a(x,y)K_U(x,y)
 \le\frac{C\Lambda_0}{\lambda_0\mu(\mathbb B_\rho(x))}
 \int_{\mathbb B_\rho(x)}a(\xi,y)K_U(\xi,y)\d\mu(\xi).
\]
These assertions require the increment and value bounds for the
actual background. Form comparability alone does not supply them.

Homogeneity retains \(\mathcal L_{a,U}U=0\) and the ordered
parameter bounds, so the ground-state transformation and logarithmic
integration keep their algebraic form.
Lemma~\ref{lem:energy-boundary-capacity} also extends with the
coefficient inserted, provided its local regularity input holds.
The principal remaining task is therefore to establish the
uniform solution estimates used in the geometric approximation.

An exact anisotropic extension follows by a linear change of
variables. For an invertible matrix \(T\), replace the denominator
in \(\Fsp\) by \(|T(x-y)|^{n+sp}\) and call the resulting
operator \(\mathcal F^T_{s,p}\). If \(v(X)=u(T^{-1}X)\), then $\mathcal F^T_{s,p}u(T^{-1}X)
 =|\det T|^{-1}\Fsp v(X)$. Thus Theorem~\ref{thm:main} applies on the transformed domain,
with thresholds and constants adjusted by the distortion of \(T\).

For the larger translation-invariant class
\[
 a(x,y)=b\left(\frac{x-y}{|x-y|}\right),\qquad
 b(\theta)=b(-\theta),\qquad \lambda_0\le b\le\Lambda_0,
\]
the half-space profile remains explicit. For a unit vector \(\nu\),
put \(h_\nu(x)=(x\cdot\nu)_+^s\) and \(t=x\cdot\nu>0\).
Pairing opposite directions in polar coordinates gives
\begin{align}\label{eq:anisotropic-profile}
 \mathcal F_{a,s,p}h_\nu(x)
 &=c_b(\nu)\operatorname{PV}\int_\R
 \frac{J_p(t^s-(t+\rho)_+^s)}{|\rho|^{1+sp}}\d\rho=0,\quad c_b(\nu):=\frac12\int_{\mathbb S^{n-1}}
 b(\theta)|\theta\cdot\nu|^{sp}\d\cH^{n-1}(\theta).
\end{align}
The cancellation is that of Proposition~\ref{prop:profile}.
A uniform half-space comparison and classification theorem and
interior gradient estimates remain necessary. The theorem of
\cite{GJS} used here concerns the isotropic operator.
Spatially varying kernels would additionally require stability
under freezing and control of coefficient errors, including
interactions with the distant exterior. Flattening a general
Lipschitz graph leads to such spatial dependence, so it does not
by itself extend the argument to arbitrary Lipschitz constants.

Nonhomogeneous dependence on the increment presents a further
obstruction. Suppose that
\[
 \mathcal F_\Phi u(x)=\operatorname{PV}\int_{\R^n}
 \frac{\Phi(x,y,u(x)-u(y))}{|x-y|^{n+sp}}\d y,
\]
where \(\Phi\) is measurable and symmetric in \((x,y)\),
odd and continuously differentiable in its last variable, and
\[
 \lambda_0|t|^{p-2}\le\partial_t\Phi(x,y,t)
 \le\Lambda_0|t|^{p-2}.
\]
Difference kernels retain the expected comparability. The derivative
kernel is
\[
 K_U^\Phi(x,y)=
 \frac{\partial_t\Phi(x,y,\Delta_{xy}U)}{|x-y|^{n+sp}},
\]
but its operator need not have \(U\) as a positive solution.
Indeed, with
\(r_\Phi(x,y,t)=t\partial_t\Phi(x,y,t)-(p-1)\Phi(x,y,t)\),
\begin{align}\label{eq:nonhomogeneous-residual}
 \mathcal L_U^\Phi U&=(p-1)\mathcal F_\Phi U+\mathcal R_U,\quad \langle\mathcal R_U,\varphi\rangle
 =\frac12\iint_{\R^{2n}}
 \frac{r_\Phi(x,y,\Delta_{xy}U)\Delta_{xy}\varphi}
 {|x-y|^{n+sp}}\dxdy.
\end{align}
The residual vanishes for \(\Phi(x,y,t)=a(x,y)J_p(t)\),
but need not vanish under general \(p\)-growth assumptions.
Scalar multiples also need not remain solutions. Both the
homogeneous quotient equation and the ordered parameter bounds
would therefore require replacement.

Two further questions concern the present operator.
First, removing the global domination \(0\le F\le M_qU\)
from linear boundary comparison requires control of unbounded
quotients and the exterior source in \eqref{eq:quotient-source}.
Second, a non-local counterpart of \cite[Lemma~3.15]{LN} would give
\begin{equation}\label{eq:growth-target}
 C_\varepsilon^{-1}
 \left(\frac{\rho}{R}\right)^{s+\varepsilon}
 \le \frac{u(A_\rho(w))}{u(A_R(w))}
 \le C_\varepsilon
 \left(\frac{\rho}{R}\right)^{s-\varepsilon},
 \qquad 0<\rho<R.
\end{equation}
An iteration of flat comparisons might yield this sharper growth,
with suitable control of the exterior data. A power bound of order
\(s+\varepsilon\) is integrable at infinity in the natural
\((p-1)\)-tail when \(\varepsilon<s/(p-1)\).
Neither question is settled or used in Theorem~\ref{thm:main}. Instead,
the approximation proof identifies escaped exterior mass
and excludes the resulting constant right side.

\subsection{Extensions to sign-changing solutions}

Negative exterior values affect normalization and boundary
non-degeneracy, as discussed in
Remarks~\ref{rem:global-nonnegativity} and
\ref{rem:sign-changing-exterior-data}. They also obstruct the
weighted quotient argument. Indeed, if \(U\) changes sign,
\(U\,\d x\) is no longer a positive measure, and
\[
 \widetilde K_U(x,y)=U(x)U(y)K_U(x,y)
\]
is negative on pairs where \(U\) has opposite signs.
Although \(K_U\) remains nonnegative, the transformed form
falls outside the Dirichlet-form framework used here.
Restricting to \(\{U>0\}\) leaves interactions with
\(\{U<0\}\) in the quotient equation. An extension would
therefore require control of negative tails and a modification
of the ground-state argument.

\appendix
\section{Proof of the local quadratic Harnack inequality}
\label{app:local-form-harnack}

We here prove Lemma~\ref{lem:local-form-harnack} and Corollary~\ref{cor:local-form-holder}.
Throughout this appendix, we use the notation and
hypotheses of Lemma~\ref{lem:local-form-harnack}.

We give a direct proof for the bounded quadratic class in the
statement of the lemma, following the organization of the Harnack argument
in \cite[Sections~4 and~5]{CY}.
We establish the local functional inequalities, control the
exterior interactions, and carry out the required iterations.
All volume comparisons are used inside \(P\).
No doubling assumption is imposed on the distant exterior.

We write \(B_t=\mathbb B_t(z)\) and
\(\mathscr E(f)=\mathscr E(f,f)\).
All balls used below have their required fixed enlargements
inside \(P\). The constant \(C_0\) will be chosen to accommodate
these enlargements.

The increment operator \(Df(x,y)=f(x)-f(y)\), defined
on the finite-energy subspace of \(\mathrm L^2(X,\mu)\),
has a closed graph into the pair \(\mathrm L^2\) space
with measure $J(x,y)\,\mathrm d\mu(x)\,\mathrm d\mu(y)$. This follows by passing to almost everywhere convergent
subsequences in both spaces. Consequently, the finite-energy
domain is complete in the form norm.
In particular, \(\mathscr D_0(G)\) is a Hilbert space
for every open \(G\Subset P\). A normal contraction is a map \(T:\R\to\R\) satisfying
\(T(0)=0\) and \(|T(a)-T(b)|\le|a-b|\).
Composition with such a map preserves the space \(\mathscr D_0(G)\)
and does not increase the form norm. Multiplication by a Lipschitz cutoff is continuous for
functions supported in a fixed compact subset of \(P\).
Indeed, the product inequality gives
\[
 \mathscr E(\eta f)
 \le C\|\eta\|_\infty^2\mathscr E(f)
      +C\int_X|f|^2\Gamma(\eta)\,\mathrm d\mu,\qquad
 \Gamma(\eta)(x)
 =
 \int_X|\eta(x)-\eta(y)|^2J(x,y)\,\mathrm d\mu(y).
\]
The moment estimate \eqref{eq:variable-moment} and scale comparison \eqref{eq:variable-weak-scale} bound
\(\Gamma(\eta)\) uniformly on the compact support of \(f\).
These facts justify the localized truncations and Lipschitz
compositions used below.

Furthermore, the weak pairings are well defined by the localization
argument in the proof of
Proposition~\ref{prop:bounded-quotient-removal}.
Indeed, let \(h\) be bounded on \(X\) and belong locally
to the form domain, and let \(\phi\in\mathscr D_0(G)\),
where \(G\Subset P\) is open.
Choose a Lipschitz cutoff \(\chi\), compactly supported
in \(P\) and equal to one near \(\overline G^{\,X}\).
The decomposition $\mathscr E(h,\phi)
 =\mathscr E(\chi h,\phi)
  +\mathscr E((1-\chi)h,\phi)$ has two absolutely convergent terms.
The first is controlled by Cauchy-Schwarz, since
\(\chi h\) has finite energy.
The second involves only separated pairs and is
controlled by boundedness of \(h\) and the full jump
moment estimate \eqref{eq:variable-moment}.
Pairs with both endpoints outside \(\overline G^{\,X}\)
contribute nothing, so no global energy assumption
on \(h\) is needed.

Let \(d_+>\beta_0\) be a local volume comparison exponent (as in Lemma~\ref{lem:variable-volume-scale}),
enlarged if necessary.
For \(f\in\mathscr D_0(B_b)\), extended by zero to \(X\),
we first prove
\begin{equation}\label{eq:local-averaging-energy}
 \|f\|_2^2
 \le C\Psi(z,b)(t/b)^{\beta_0}\mathscr E(f)
 +\frac{C(b/t)^{d_+}}{\mu(B_b)}\|f\|_1^2,
 \qquad 0<t<b.
\end{equation}
All norms in this proof use \(\mu\).
Local doubling and full support give a covering of \(B_b\)
by balls \(B_t(x_i)\) whose fixed enlargements have uniformly
bounded overlap.
Scale and volume comparison give
\[
 \Psi(x_i,t)\le C(t/b)^{\beta_0}\Psi(z,b),
 \qquad
 \mu(B_t(x_i))\ge c(t/b)^{d_+}\mu(B_b).
\]
Subtract the weighted mean on each covering ball and apply
Poincar\'e \eqref{eq:variable-poincare} to the resulting fluctuation.
Bounded overlap controls the sum of the local energies by
\(C\mathscr E(f)\).
For the mean terms, use
\[
 \mu(B_t(x_i))|f_{B_t(x_i),\mu}|^2
 \le
 \frac{1}{\mu(B_t(x_i))}
 \left(\int_{B_t(x_i)}|f|\,\mathrm d\mu\right)^2.
\]
The volume lower bound and bounded overlap then give
\eqref{eq:local-averaging-energy}.
All covering balls and their enlargements lie in a fixed
enlargement of \(B_b\).

We next derive a local Faber-Krahn inequality, \eqref{eq:local-faber-krahn} below. Suppose that \(f=0\) almost everywhere outside a measurable
set \(V\subset B_b\). Then \(\|f\|_1^2\le\mu(V)\|f\|_2^2\).
If \(\mu(V)/\mu(B_b)\) is sufficiently small, choose \(t/b\)
to be a sufficiently large fixed multiple of
\((\mu(V)/\mu(B_b))^{1/d_+}\).
The second term in \eqref{eq:local-averaging-energy} can then
be absorbed into the left-hand side.
For the remaining volume ratios, reverse doubling gives
a fixed \(K_*>1\) such that
\[
 \mu(B_{K_*b})\ge2\mu(B_b).
\]
Since \(f=0\) outside \(B_b\),
\[
 \int_{B_{K_*b}}
 |f-f_{B_{K_*b},\mu}|^2\,\mathrm d\mu
 \ge \frac12\|f\|_2^2.
\]
Poincar\'e \eqref{eq:variable-poincare} on this enlarged ball and scale comparison \eqref{eq:variable-weak-scale}
therefore give the required estimate also in this case.
With \(\nu=\beta_0/d_+\in(0,1)\), we obtain
\begin{equation}\label{eq:local-faber-krahn}
 \|f\|_2^2
 \le C\Psi(z,b)
 \left(\frac{\mu(V)}{\mu(B_b)}\right)^\nu
 \mathscr E(f).
\end{equation}
The fixed enlargement used in the argument is included
in the choice of \(C_0\). The iterations below use \eqref{eq:local-faber-krahn} directly.
Its energy is that of the zero extension to \(X\), so it
retains all exterior interactions without requiring exterior
volume comparisons.

We next control the positive exterior tail by local values
and the negative exterior tail.
For \(0<a<b\), write
\[
 T_{a,b}(f)
 =
 \esssup_{\xi\in B_a}
 \int_{X\setminus B_b}|f(y)|J(\xi,y)\,\mathrm d\mu(y).
\]
Suppose that \(h\) is harmonic and nonnegative in \(B_b\),
and put
\[
 H_b=\esssup_{B_b}h,\qquad \delta=b-a.
\]
Choose \(0\le\eta\le1\), equal to one on \(B_{a+\delta/4}\),
compactly supported in \(B_{a+\delta/2}\), and with
Lipschitz constant at most \(C/\delta\).
Assume first that \(H_b>0\), and test with
\((h-2H_b)\eta^2\). The integral over \(B_b\times B_b\) is bounded below by
\(-CH_b^2\mathscr E(\eta)\).
For a cross pair, with \(x\) in the support of \(\eta\)
and \(y\notin B_b\), use
\[
 (h(x)-h(y))(h(x)-2H_b)
 \ge H_bh_+(y)-2H_bh_-(y)-2H_b^2.
\]
The moment estimate \eqref{eq:variable-moment}, scale comparison \eqref{eq:variable-weak-scale}, and the separation
of the cutoff support from \(X\setminus B_b\) give
\[
 \mathscr E(\eta)
 \le
 \frac{C\mu(B_b)}{\Psi(z,b)}
 \left(\frac b\delta\right)^{\beta_1}.
\]
After division by \(H_b\), the weak equation consequently
yields
\begin{align*}
 &\int_{B_{a+\delta/4}}
   \int_{X\setminus B_b}
   h_+(y)J(x,y)\,\mathrm d\mu(y)\,\mathrm d\mu(x)\le
 C\mu(B_b)\left[
 \left(\frac b\delta\right)^{\beta_1}
 \frac{H_b}{\Psi(z,b)}
 +T_{a+\delta/2,b}(h_-)\right].
\end{align*}
To pass from this averaged estimate to a supremum over
first endpoints, fix \(\xi\in B_a\) and apply the averaging
estimate \eqref{eq:variable-ujs} on \(\mathbb B_{\delta/8}(\xi)\).
This ball lies in \(B_{a+\delta/4}\), and every
\(y\notin B_b\) satisfies \(d(\xi,y)\ge\delta\).
Moreover,
\[
 \mu(\mathbb B_{\delta/8}(\xi))
 \ge c(\delta/b)^{d_+}\mu(B_b).
\]
Integrating the averaging inequality in \(y\) therefore gives
\begin{equation}\label{eq:local-positive-tail}
 T_{a,b}(h_+)
 \le
 C\left(\frac b{b-a}\right)^{d_++\beta_1}
 \left(
 \frac{H_b}{\Psi(z,b)}
 +T_{(a+b)/2,b}(h_-)
 \right).
\end{equation}
If \(H_b=0\), use
\((h-2(H_b+\varepsilon))\eta^2\), divide by
\(H_b+\varepsilon\), and let \(\varepsilon\downarrow0\).
The averaging assumption is essential in this step.
It controls the supremum over first endpoints while
allowing the second endpoint to range over all of
\(X\setminus B_b\).

We will also use the following quadratic cutoff estimate.
Let \(w=(h-k)_+\), where \(k\ge0\), and let \(\eta\) be
compactly supported in \(B_b\).
Testing with \(w\eta^2\), and splitting the form into
interior pairs and cross pairs, gives
\[
 \mathscr E(w\eta)
 \le
 C\int_{B_b}w^2\Gamma(\eta)\,\mathrm d\mu
 +C T_{\supp\eta,b}(w)
       \int_{B_b}w\eta^2\,\mathrm d\mu.
\]
Here \(T_{\supp\eta,b}\) denotes the same tail supremum
with first endpoint restricted to \(\supp\eta\).
This estimate involves only the first power of the
exterior values of \(w\).
In particular, no exterior \(\mathrm L^2\) bound enters
the iteration.

We now prove the weak Harnack estimate.
Choose \(\ell/R\) sufficiently small that all fixed
enlargements used below lie in \(B_{R/2}\), and put
\[
 T_-=\mathcal T_{z,R}(h_-),\qquad
 k=C_*\Psi(z,\ell)T_-+\varepsilon,\qquad
 v=h+k,\qquad \varepsilon>0.
\]
The constant \(C_*\) is structural.
On \(B_R\), \(v\ge k\), while \(v_-\le h_-\) on \(X\).
The quantity \(T_-\) is finite by boundedness of \(h\)
and the separated jump moment estimate \eqref{eq:variable-moment}.

We first obtain a uniform logarithmic oscillation bound.
Let \(B=\mathbb B_t(x)\), where \(x\in B_{4\ell}\) and
\(0<t\le4\ell\).
Choose \(\ell/R\) small enough that the fixed Poincar\'e
and cutoff enlargements of every such ball lie in
\(B_{R/2}\).
If \(\kappa\) is the Poincar\'e enlargement factor,
choose \(0\le\eta\le1\), equal to one on \(\kappa B\),
compactly supported in \(2\kappa B\), and with Lipschitz
constant at most \(C/t\).
Use the test function
\[
{\eta^2}/{(h_++k)}.
\]
On the support of \(\eta\), this equals \(\eta^2/v\).
The strict regularization \(k\ge\varepsilon>0\) makes
the reciprocal a Lipschitz function of \(h_+\).
After subtracting its value at zero, the local form-domain
property and cutoff multiplication justify this test.
It is well defined even where \(h\) is negative outside
\(B_R\).

For \(a,b>0\) and \(\alpha,\gamma\in\R\), the identity
\[
 (a-b)\biggl({\alpha^2}/{a}-{\gamma^2}/{b}\biggr)
 = (\alpha-\gamma)^2
 -\biggl(\alpha\sqrt{b/a}
        -\gamma\sqrt{a/b}\biggr)^2
\]
controls the pairs in \(B_R\times B_R\), where \(v>0\).
On \(\kappa B\times\kappa B\), both cutoff values are one,
and the last square satisfies
\[
\biggl(\sqrt{b/a}
        -\sqrt{a/b}\biggr)^2
 ={(a-b)^2}/{ab}
 \ge |\log a-\log b|^2.
\]
For the other interior pairs, discard that nonnegative
square. On cross pairs, with \(x\in\supp\eta\) and
\(y\notin B_R\), use
\[
 (h(x)-h(y))\frac{\eta(x)^2}{h(x)+k}
 \le \eta(x)^2\left(1+\frac{h_-(y)}{k}\right).
\]
The weak equation consequently gives
\[
 \mathscr E_{\kappa B}(\log v)
 \le \mathscr E(\eta)
      +\frac{T_-}{k}\int_X\eta^2\,\mathrm d\mu
 \le \frac{C\mu(B)}{\Psi(x,t)}
      +C\frac{T_-}{k}\mu(B).
\]
Here the first endpoints of the cross pairs lie in
\(B_{R/2}\), and \(h_-=0\) in \(B_R\), so their negative
exterior contribution is bounded by \(T_-\).
The last inequality uses the full moment estimate \eqref{eq:variable-moment} and
local volume and scale comparison.
The localized logarithm is equally justified by the
Lipschitz composition \(\log(h_++k)-\log k\).

Applying Poincar\'e \eqref{eq:variable-poincare} and using scale comparison \eqref{eq:variable-weak-scale}, we obtain
\[
 \avgint_B
 |\log v-(\log v)_{B,\mu}|^2\,\mathrm d\mu
 \le
 C\left(1+\frac{\Psi(x,t)T_-}{k}\right)
 \le C.
\]
The proof of the ball form of the John-Nirenberg inequality
in \cite[Theorem~2.2]{Buckley} uses only the preceding
mean-oscillation bounds and doubling comparisons in a
fixed enlargement of \(B_{2\ell}\).
We choose \(\ell/R\) still smaller, if necessary, so that
\(B_{16\ell}\) and its required structural enlargements
lie in \(B_{R/2}\).
That proof therefore gives a structural \(q_0\in(0,1)\)
such that
\[
 \avgint_{B_{2\ell}}
 \exp\left(q_0|\log v-(\log v)_{B_{2\ell},\mu}|\right)
 \,\mathrm d\mu\le C.
\]
Applying this estimate to the positive and negative signs
of the centered logarithm and multiplying gives
\begin{equation}\label{eq:local-log-crossover}
 \left(\avgint_{B_{2\ell}}v^{q_0}\,\mathrm d\mu
 \right)^{1/q_0}
 \left(\avgint_{B_{2\ell}}v^{-q_0}\,\mathrm d\mu
 \right)^{1/q_0}
 \le C.
\end{equation}
This is the crossover step. Compare \cite[Lemma~4.9]{CY}.
Its constants are independent of \(k\), \(\varepsilon\),
and the bound on \(h\).

For the sublevel iteration, fix \(a\ge k\), and set
\[
 r_j=\ell(1+2^{-j}),\qquad
 a_j=\frac a2(1+2^{-j}),\qquad
 A_j=B_{r_j}\cap\{v<a_j\},\qquad
 Y_j=\frac{\mu(A_j)}{\mu(B_{2\ell})}.
\]
Choose \(\eta_j=1\) on \(B_{r_{j+1}}\), compactly supported
in \(B_{(r_j+r_{j+1})/2}\), with Lipschitz constant at most
\(C2^j/\ell\), and put \(w_j=(a_j-v)_+\).
Testing with \(w_j\eta_j^2\) gives
\begin{align*}
 0=\mathscr E(v,w_j\eta_j^2)
 &\le-\mathscr E(w_j\eta_j)+\frac12\iint_{X\times X}
 w_j(x)w_j(y)|\eta_j(x)-\eta_j(y)|^2
 J(x,y)\,\mathrm d\mu(x)\,\mathrm d\mu(y).
\end{align*}
By symmetry, every nonzero contribution in the last
integral can be estimated with its first endpoint in
\(A_j\cap\supp\eta_j\).
On \(B_R\), \(w_j\le a\).
Outside \(B_R\), we have \(w_j\le a+h_-\).
The moment estimate \eqref{eq:variable-moment} and the definition of \(T_-\) in \eqref{eq:variable-tail}
therefore yield
\[
 \mathscr E(w_j\eta_j)
 \le
 \frac{Ca^2 2^{j\beta_1}}{\Psi(z,\ell)}\mu(A_j)
 +CaT_-\mu(A_j)
 \le
 \frac{Ca^2 2^{j\beta_1}}{\Psi(z,\ell)}\mu(A_j).
\]
The last inequality uses
\(a\ge k\ge C_*\Psi(z,\ell)T_-\).
This calculation includes the annuli outside \(B_{r_j}\)
as well as the distant exterior. The function \(w_j\eta_j\) vanishes outside \(A_j\), and
on \(A_{j+1}\) it is at least
\[
 a_j-a_{j+1}=a2^{-j-2}.
\]
Applying \eqref{eq:local-faber-krahn} on \(B_{2\ell}\)
and using scale comparison \eqref{eq:variable-weak-scale} gives
\[
 Y_{j+1}
 \le C2^{j(2+\beta_1)}Y_j^{1+\nu}.
\]
This sequence tends to zero when \(Y_0\) is below a
structural threshold.
Consequently, there is \(\varepsilon_*>0\) such that
\[
 \mu(B_{2\ell}\cap\{v<a\})
 \le\varepsilon_*\mu(B_{2\ell})
 \quad\Longrightarrow\quad
 \essinf_{B_\ell}v\ge a/2.
\]
This proves the sublevel implication directly. Compare
\cite[Lemma~4.4]{CY}.

Choose
\[
 a=c_*
 \left(\avgint_{B_{2\ell}}v^{-q_0}\,\mathrm d\mu
 \right)^{-1/q_0},
 \qquad c_*^{q_0}\le\varepsilon_*.
\]
Chebyshev's inequality gives the required sublevel
measure bound.
If \(a<k\), the conclusion follows directly from \(v\ge k\).
Otherwise, the preceding iteration applies.
Combining the resulting lower bound with
\eqref{eq:local-log-crossover}, and letting
\(\varepsilon\downarrow0\), proves
\begin{equation}\label{eq:local-weak-harnack}
 \left(\avgint_{B_{2\ell}}h^{q_0}\,\mathrm d\mu
 \right)^{1/q_0}
 \le
 C\left(
 \essinf_{B_\ell}h
 +\Psi(z,\ell)\mathcal T_{z,R}(h_-)
 \right).
\end{equation}

It remains to estimate the supremum by the integral
on the left-hand side.
Fix \(0<b\le R/2\), and put
\[
 A_b=H_b+\Psi(z,b)T_-,\qquad
 b_j=\frac b2(1+2^{-j}),\qquad
 k_j=N(1-2^{-j}),\qquad N>0.
\]
Let
\[
 w_j=(h-k_j)_+,\qquad
 I_j=\int_{B_{b_j}}w_j^2\,\mathrm d\mu,
\]
and choose \(\eta_j=1\) on \(B_{b_{j+1}}\), compactly
supported in \(B_{(b_j+b_{j+1})/2}\), with Lipschitz
constant at most \(C2^j/b\).
Testing with \(w_{j+1}\eta_j^2\) and using
\[
 w_{j+1}
 \le\frac{w_j^2}{k_{j+1}-k_j}
 \quad\text{where }w_{j+1}>0,
\]
together with \eqref{eq:variable-moment} and \eqref{eq:variable-weak-scale},
gives
\begin{align*}
 \mathscr E(w_{j+1}\eta_j)
 &\le
 C\left[
 \frac{2^{j\beta_1}}{\Psi(z,b)}
 +\frac{
 T_{(b_j+b_{j+1})/2,b_j}(h_+)}
 {k_{j+1}-k_j}
 \right]I_j\le
 \frac{C2^{j(d_++\beta_1+1)}}{\Psi(z,b)}
 \left(1+\frac{A_b}{N}\right)I_j.
\end{align*}
For the second inequality, apply
\eqref{eq:local-positive-tail} at radius \(b_j\).
The local supremum is at most \(H_b\).
Moreover, every first endpoint lies in \(B_{R/2}\),
and \(h_-=0\) in \(B_R\), so each negative tail is
bounded by \(T_-\).

The support of \(w_{j+1}\eta_j\) has measure at most
\[
 \mu(B_{b_j}\cap\{h>k_{j+1}\})
 \le
 \frac{I_j}{(k_{j+1}-k_j)^2}
 =\frac{2^{2j+2}I_j}{N^2}.
\]
Apply \eqref{eq:local-faber-krahn} to this function.
Volume and scale comparison between radii \(b/2\) and \(b\)
then give
\[
 Y_{j+1}
 \le Cb_*^j\left(1+\frac{A_b}{N}\right)Y_j^{1+\nu},
 \qquad
 Y_j=\frac{I_j}{N^2\mu(B_b)},
\]
where \(b_*>1\) is structural.
The iteration converges to zero provided that
\[
 Y_0\le c\left(1+\frac{A_b}{N}\right)^{-1/\nu}.
\]
For \(0<\varepsilon\le1\), this condition is ensured by
\[
 N=\varepsilon A_b+
 C\varepsilon^{-1/(2\nu)}
 \left(\avgint_{B_b}h^2\,\mathrm d\mu\right)^{1/2},
\]
with \(C\) sufficiently large.
Indeed, \(A_b/N\le\varepsilon^{-1}\), while the second
term makes \(Y_0\le C^{-2}\varepsilon^{1/\nu}\).
We conclude that
\[
 \esssup_{B_{b/2}}h
 \le
 \varepsilon\bigl(H_b+\Psi(z,b)T_-\bigr)
 +C\varepsilon^{-1/(2\nu)}
 \left(\avgint_{B_b}h^2\,\mathrm d\mu\right)^{1/2}.
\]
This proves the quadratic mean-value estimate, with all
exterior tails retained. Compare \cite[Lemma~5.6]{CY}.

For arbitrary nested radii, the same calculation introduces
only fixed powers of the inverse relative gap.
To replace the \(\mathrm L^2\) average by the
\(\mathrm L^{q_0}\) average, use
\[
 \left(\avgint_Bh^2\,\mathrm d\mu\right)^{1/2}
 \le
 (\esssup_Bh)^{1-q_0/2}
 \left(\avgint_Bh^{q_0}\,\mathrm d\mu\right)^{1/2}.
\]
Young's inequality, together with volume and scale
comparison, therefore yields
\begin{align*}
 H(a)
 &\le\varepsilon H(b)+
 C\varepsilon^{-c_0}
 \left(\frac{\ell}{b-a}\right)^{c_1}
 \left[
 \left(\avgint_{B_{2\ell}}h^{q_0}\,\mathrm d\mu
 \right)^{1/q_0}
 +\Psi(z,\ell)T_-
 \right],
 \qquad \ell\le a<b\le2\ell,
\end{align*}
where \(H(t)=\esssup_{B_t}h\), and \(c_0,c_1>0\)
depend only on the structural constants.
Set \(a_j=\ell(2-2^{-j})\) and fix
\(\varepsilon<2^{-c_1-1}\).
Iteration gives a convergent geometric series.
The residual term \(\varepsilon^jH(a_j)\) tends to zero
because \(h\) is bounded.
Hence
\[
 \esssup_{B_\ell}h
 \le C\left[
 \left(\avgint_{B_{2\ell}}h^{q_0}\,\mathrm d\mu
 \right)^{1/q_0}
 +\Psi(z,\ell)T_-
 \right].
\]
Combining this estimate with
\eqref{eq:local-weak-harnack} proves
\eqref{eq:local-form-harnack}.

Finally, choose \(C_0\) to contain all fixed structural
enlargements, including the enlargement used in the
Faber-Krahn argument.
Choose \(\delta_H\) sufficiently small that the logarithmic
and Harnack iterations take place inside \(B_{R/2}\).
Both choices depend only on the structural data.
Boundedness of \(h\) ensures test admissibility and removes
the residual term in the last iteration.
Its numerical bound does not enter any of the constants. This completes the proof of Lemma~\ref{lem:local-form-harnack}.\qed

\begin{proof}[Proof of Corollary~\ref{cor:local-form-holder}]
If \(Q=0\), the conclusion follows from the zero
representative. Assume \(Q>0\).
Fix \(x\in\mathbb B_r(z)\) and put \(r_0=r/4\).
All required fixed enlargements lie in
\(\mathbb B_{C_*r}(z)\subset P\) after choosing
\(C_*\) sufficiently large. We first normalize the positive tail by \(Q\).
Since \(h\le Q\) on \(\mathbb B_{2r_0}(x)\),
\eqref{eq:local-positive-tail}, with \(a=r_0/2\),
\(b=2r_0\), and \(h_-=0\), gives
\[
 \esssup_{\xi\in\mathbb B_{r_0/2}(x)}
 \int_{X\setminus\mathbb B_{2r_0}(x)}
 h(y)J(\xi,y)\,\mathrm d\mu(y)
 \le \frac{CQ}{\Psi(x,r_0)}.
\]
On \(\mathbb B_{2r_0}(x)\setminus\mathbb B_{r_0}(x)\),
use \(h\le Q\), separation from
\(\mathbb B_{r_0/2}(x)\), and the jump moment estimate \eqref{eq:variable-moment}.
The nearby-center scale comparison \eqref{eq:variable-weak-scale} therefore yields
\begin{equation}\label{eq:local-holder-positive-tail}
 \mathcal T_{x,r_0}(h)
 \le \frac{CQ}{\Psi(x,r_0)}.
\end{equation}
This estimate includes all exterior interactions.

Choose \(0<\lambda<\min\{\delta_H,1/4\}\), and set
\(r_j=\lambda^jr_0\).
We construct nested intervals \([m_j,M_j]\) containing
the essential range of \(h\) on \(\mathbb B_{r_j}(x)\),
with
\[
 [m_0,M_0]=[0,2Q],\qquad
 \omega_j=M_j-m_j=2Q\lambda^{j\alpha}.
\]
The choices of \(\lambda\) and \(\alpha\) are fixed below.
Suppose the intervals have been constructed through step \(j\).
Both \(h-m_j\) and \(M_j-h\) are harmonic and nonnegative
in \(\mathbb B_{r_j}(x)\), although they may be negative
outside that ball. Apply \eqref{eq:local-form-harnack}
to each, with inner radius \(r_{j+1}\), and add the
resulting inequalities. Writing
\(\vartheta=(H-1)/(H+1)<1\), we obtain
\begin{align*}
 &\esssup_{\mathbb B_{r_{j+1}}(x)}h
 -\essinf_{\mathbb B_{r_{j+1}}(x)}h\le\vartheta\omega_j
 +C\Psi(x,r_{j+1})
 \bigl(\mathcal T_{x,r_j}((h-m_j)_-)
       +\mathcal T_{x,r_j}((M_j-h)_-)\bigr).
\end{align*}

On the annulus
\(\mathbb B_{r_k}(x)\setminus\mathbb B_{r_{k+1}}(x)\),
where \(k<j\), the sum of these two negative parts
is at most \(2\omega_k\).
Its distance from any first endpoint in
\(\mathbb B_{r_j/2}(x)\) is at least \(r_{k+1}/2\).
The moment estimate \eqref{eq:variable-moment} and nearby-center comparison \eqref{eq:variable-weak-scale}
thus bound its kernel mass by
\(C/\Psi(x,r_{k+1})\).
Outside \(\mathbb B_{r_0}(x)\), the sum of the negative
parts is at most \(h+2Q\), so
\eqref{eq:local-holder-positive-tail} and the moment
estimate \eqref{eq:variable-moment} control this remaining contribution.
Consequently, the error in the last display is at most
\[
 C\Psi(x,r_{j+1})
 \left(\sum_{k=0}^{j-1}
 \frac{\omega_k}{\Psi(x,r_{k+1})}
 +\frac{Q}{\Psi(x,r_0)}\right)
 \le C\omega_j
 \frac{\lambda^{\beta_0-\alpha}}
      {1-\lambda^{\beta_0-\alpha}}.
\]
The sum is empty when \(j=0\).
Here we used the positive lower scale exponent
\(\beta_0\) in \eqref{eq:variable-weak-scale}.
First choose \(\lambda\) sufficiently small that the
last coefficient is at most \((1-\vartheta)/2\)
whenever \(0<\alpha\le\beta_0/2\).
Then choose \(0<\alpha<\min\{1,\beta_0/2\}\)
small enough that
\(\lambda^\alpha\ge(1+\vartheta)/2\).
The essential oscillation on the next ball is then
at most \(\omega_{j+1}\).
Its essential range fits in an interval of that length
contained in \([m_j,M_j]\), completing the induction.

These bounds are uniform in \(x\in\mathbb B_r(z)\).
Local doubling gives differentiation of locally
integrable functions, and the vanishing oscillations
show that local averages define a representative of
\(h\) at every such point.
Full support of \(\mu\) makes the representatives
consistent on overlapping balls, and the resulting
representative satisfies \(0\le h\le Q\).
For \(d(x,y)<r_0/2\), choose a radius \(r_j\)
comparable to \(d(x,y)\) whose ball centered at \(x\)
contains a neighborhood of \(y\).
The oscillation bound gives
\[
 |h(x)-h(y)|\le CQ
 \left(\frac{d(x,y)}r\right)^\alpha.
\]
For larger distances, the same inequality follows,
after increasing \(C\), from \(0\le h\le Q\) on
\(\mathbb B_r(z)\).
This proves \eqref{eq:local-form-holder}.
\end{proof}

\medskip
\noindent
{\bf Conflict of interest.}
The author states that there is no conflict of interest.

\medskip
\noindent
{\bf Data availability.}
Data sharing is not applicable to this article as no datasets
were generated or analysed during the current study.

\medskip
\noindent
{\bf Declaration of use of AI.}
The author used GPT-6 Astra, developed by OpenAI, during the preparation and revision of this manuscript.  The tool was
used for language editing, organization, LaTeX formatting, and checks of
clarity and internal consistency.  The mathematical arguments, results,
and conclusions were developed and verified by the author, who also
checked the references and takes full responsibility for the content of
the manuscript.

\end{document}